\documentclass[11pt,a4paper,reqno]{amsart}
\usepackage[english]{babel}
\usepackage{amsmath}
\usepackage{amsfonts}
\usepackage{amssymb}
\usepackage{mathrsfs}
\usepackage{latexsym}
\usepackage{natbib}
\usepackage[margin=3cm]{geometry}

\usepackage{color}

\usepackage{graphicx,color}

\usepackage[plainpages=false,colorlinks,hyperindex,bookmarksopen,linkcolor=black,citecolor=blue,urlcolor=blue]{hyperref}
\newcommand*\bigcdot{\mathpalette\bigcdot@{.3}}

\bibpunct{[}{]}{;}{n}{,}{,}

\theoremstyle{theorem}

\newtheorem{theorem}{\sc Theorem}[section]

\newtheorem{definition}{\sc Definition}
\newtheorem{corollary}{\sc Corollary}[section]
\newtheorem{remark}{\sc Remark}
\newtheorem{assumption}{\sc Assumption}
\newtheorem*{notation}{\sc Notation}

\title{Earthquake modelling via Brownian motions on networks}

\author{F. Colantoni}
\address{Department of Basic and Applied Sciences for Engineering,\newline  \indent Sapienza University of Rome, Italy}
\email{fausto.colantoni@uniroma1.it}
\author{M. D'Ovidio}
\address{Department of Statistical Sciences,\newline \indent Sapienza University of Rome, Italy}
\email{mirko.dovidio@uniroma1.it}
\author{F. Tavani}
\address{National Institute of Geophysics and Volcanology, Rome, Italy} 
\email{flavia.tavani@ingv.it}

\begin{document}

\maketitle

\begin{abstract}
We provide a general model for Brownian motions on metric graphs with interactions. In a general setting, for (sticky) Brownian propagations on edges, our model provides a characterization of lifetimes and holding times on vertices in terms of (jumping) Brownian accumulation of energy associated with that vertices.  Propagation and accumulation are given by drifted Brownian motions subjected to non-local (also dynamic) boundary conditions. As the continuous (sticky) process approaches a vertex, then the right-continuous process has a restart (resetting), it jumps randomly away from the zero-level of energy. According with this new energy, the continuous process can start (or not) as a new process in a randomly chosen edge. We provide a self-contained presentation with a detailed construction of the model. The model well extends to a higher order of interactions, here we provide a simple case and focus on the analysis of earthquakes. Earthquakes are notoriously difficult to study. They build up over long periods and release energy in seconds. Our goal is to introduce a new model, useful in many contexts and in particular in the difficult attempt to manage seismic risks.
\end{abstract}

\vspace{1cm}

\tableofcontents

{\bf Keywords (MSC):}
60J60, 
60J55,
60J70,
35R11.


\section{Introduction}

To introduce the subject of our paper, first suppose we can conveniently change a given problem into a different, hopefully simpler, form involving star graphs, that is a network. Then, consider interacting Brownian motions on networks. We are interested in the characterization of holding times (including trap vertices), hitting times (of vertices), sojourn time (on subsets of the network) and lifetimes (on the network). Brownian motions together with interaction rules between edges and vertices describe the dynamic we are interested in. A trap vertex is a vertex in which a Brownian motion can spend an infinite average of time, this obviously includes the case of absorption. An example of application can be given by Data/Traffic flows, vertices represent servers/cities and edges represent connections, the holding time in a given vertex describes a (processing, queuing, transmission, propagation) delay in packet switching networks or waiting time due to traffic lights (and congestion) in the context of vehicles.     

We provide a self-contained presentation of a simple model for the applied sciences. The model can be considered in many contexts in which the observations are affected by some Gaussian noise, for example
\begin{align}
    \mathcal{T}(t)  + \textrm{noise}(t), \quad t \in \{t_i, \, i =1, \ldots, n\}
    \label{ASSobservations}
\end{align}
where $\mathcal{T}$ describes some underlying trend (signal).  As the sample size increases, that is, as $n \to \infty$ we have superposition of Gaussian signals. Then, we consider Brownian motions. Since our model is based on Markov processes we are able to reconstruct data around a sample of observations. This is the case, for example, of generative models as diffusion models in which data are progressively noised (by a diffusion process) and subsequently noise is transformed into new data (by the reverse process). Earthquakes cannot currently be systematically predicted, however generative models can help in terms of prior events/information announcing earthquakes. A further example of application is given by the analysis of the seismic signals in terms of Earthquake Monitoring Networks which refers to a system of interconnected seismic sensors and data processing centres designed to analyse earthquakes. 

There are many possible applications of our model, we decided to bring our attention to the analysis of the energy accumulation and wave propagation in the earthquakes analysis. We fix the following general rules:
\begin{itemize}
\item Brownian motions associated with vertices provide energy;
\item Brownian motions associated with edges provide propagation;    
\item Brownian motions can not run simultaneously for a vertex and an incident edge.
\end{itemize}
We also allow for interaction and fix the following further rule: 
\begin{itemize}
\item Holding times on vertices for Brownian motion associated with edges correspond to running times for Brownian motions associated with vertices.
\end{itemize}
Further specification of the previous rules are possible. We can relax or stiffen that rules with great flexibility. Moreover, with the help of the It\^{o}'s formula for example, transformations of Brownian motions can be considered for general noises in \eqref{ASSobservations}. As we focus on the earthquakes modelling, many difficulties occur facing with collection, topography of the involved area, ground data limitation and time constraint. Thus, noise becomes the big box in which we store all the unknown  quantities and the latent phenomena. Since we aim to describe the wave propagation and the energy accumulation, for the observation above in a very simple model, we assume that $\mathcal{T}(t)=\mu t$ where $\mu$ can be the magnitude $m$ or the velocity $v$ respectively for the energy accumulation or the seismic wave propagation. The noise given by a Brownian motion or its transformation can be also associated with measurement errors as time passes. We stress the fact that the Brownian motion does not describe an earthquake, it helps to understand \eqref{ASSobservations} from the observations.  

\section{An overview of the model}

\subsection{Recurrence}
We consider the Brownian model for recurrent earthquakes. This idea has been introduced in \cite{ellsworth} to describe the probability model for rupture times on a recurrent earthquake source. An earthquake happens according with a load state $\mathfrak{E}_t$, that is a physical quantity like elastic strain, or cumulative stress. Let $\mathfrak{e}_0$ be the ground state and define $\mathfrak{e}_\delta$ as the final state, $\mu>0$ as the mean loading rate. That is $\mathfrak{e}_\delta = \mathfrak{e}_0 + \mu \delta$ and $\delta=\delta(t)$ depends on the time in which an event occurs. Under the (deterministic) relaxation time $r(t)$ before $t>0$, we can introduce the deterministic description of the relaxation oscillator
\begin{align}
    \mathfrak{e}_0 + \mu (t-r(t)), \quad t\geq 0.
\end{align}
Random perturbations lead to the stochastic relaxation oscillator
\begin{align}
    \mathfrak{e}_0 + \mu (\mathfrak{E}_t - \mathfrak{E}_{R_t} ), \quad t\geq 0
\end{align}
where $R_t := \max \{T_k\,:\, k \geq 0,\, T_k \leq t\}$ and $T_k := \min\{t\,:\, t\geq 0,\, \mathfrak{E}_t \geq k \delta\}$ is the first time the load state $\mathfrak{E}_t$ exceeds the level $k\delta$. Rupture is assumed to occur when the process reaches a critical-failure threshold. An earthquake relaxes the load state to a characteristic ground level and begins a new failure cycle. The load-state process is a Brownian relaxation oscillator and the Gaussian distribution of the increments motivates via central-limit arguments the fact that perturbations can be regarded as the sum of many small, independent  contributions.

Further on we refer to the energy  $E$ (see Section \ref{sec:EnergyAccumulation}).  As in \cite{ellsworth} the load state process is subjected to a stochastic restart (or reset).

\subsection{Occurrence}
\label{sec:occurrence}
We describe here the occurrence of a seismic sequence in a specific geographical region. In particular, we illustrate the construction of an earthquake network via earthquake occurrence.

\begin{definition}
We call {\it mainshock} the largest and most significant event of the sequence, while the {\it aftershocks} are smaller tremors that follow, often near the same fault line. The magnitude and frequency of the aftershocks gradually decrease over time.
\end{definition}

Let us consider a rectangular spatial region $D = [a, b] \times[c,d]$ , where $[a,b]$ is an interval of longitudes and $[c,d]$ is an interval of latitudes. Assume that 
the mainshock of a sequence having local magnitude equal to $m_0$ happens at time $t_0$ having epicenter with coordinates $(x_0,y_0)$. Let us $D_0 = [a_0, b_0] \times [c_0,d_0] \subset D$ be a small region around the epicenter. We assume that a mainshock can be followed by other events, the $i$-th event on the region $D_i \subset D$ has magnitudes $m_i$, $i=1.2.\ldots, k$.  

\begin{definition}
For the $i$-th event on $D_i$, we define the epicenter $\{x_i,y_i\}\in D_i$ near $D_0$ for $i=1,...,k$. Thus, $\{D_i\}_{i=1}^k$ are the areas influenced by what happens inside $D_0$. For the sake of simplicity, we consider $D_i\cap D_j = \emptyset$ for any $i,j=0,...k$, $i\neq j$. 
\end{definition}
 
To build up the model, we need to briefly introduce the mathematical definition of graph from a theoretical point of view. A graph, or network, $\mathcal{G}=(\mathcal{V},\mathcal{E})$ is an object defined by a set of vertices (also called nodes) $\mathcal{V}$, and a set of edges $\mathcal{E}$ representing the existing connection among pair of vertices. Given a vertex $v \in \mathcal{V}$, its neighbors are the nodes connected with it by an edge. The set of the edges $\mathcal{E}$ can be written as the pairs $(v_i,v_j)$; a graph is defined as undirected if these couples are unordered, otherwise it is directed. The sequence of edges connecting any pair of vertices is called path. Here we consider the trees, a type of undirected graphs in which any two vertices are connected by exactly one path. In a tree, usually, one of the vertices is designed as the root. A parent of a vertex $v$ is a vertex $u$ directly connected to it by a path to the root, and $v$ is said to be a child of $u$. In a tree, a leaf is a node without children. We define the star graphs as a specific type of trees.
\begin{definition}
   A star graph $S_k$ is a tree with a root $r$, $k$ leaves, $k+1$ nodes.
\end{definition}

In our case we present the interaction between the areas $D_i$ modeling them using a star graph $S_k$. The vertices of the network represent the regions $D_i$, with $i=0,\dots, k$, while the edges between them represent the influence among different sub-regions $D_i$ with $i\in[0,....,k]$. We will consider the area $D_0$ where the mainshock occurs as the root of the $S_k$. The other regions $D_i$ with $i=1,\dots, k$ are the nearby areas where aftershocks can take place. Geographically, we can imagine that $D_i$ are  areas (here simply considered as rectangular regions) located on the faults adjacent to the one which triggered the mainshock. A simplified representation of this structure is provided in Fig. $\ref{fig:step_one}$, where we draw $D_0$ and all $D_i$ with $i=1,...,5$. This representation does not show the real position on the map of the possible influenced areas, not considering the geology of the fault plane in central Italy where the earthquakes can occur. \\

\begin{figure}
\centering
\begin{minipage}{.5\textwidth}
  \centering
  \includegraphics[width=1.1\linewidth]{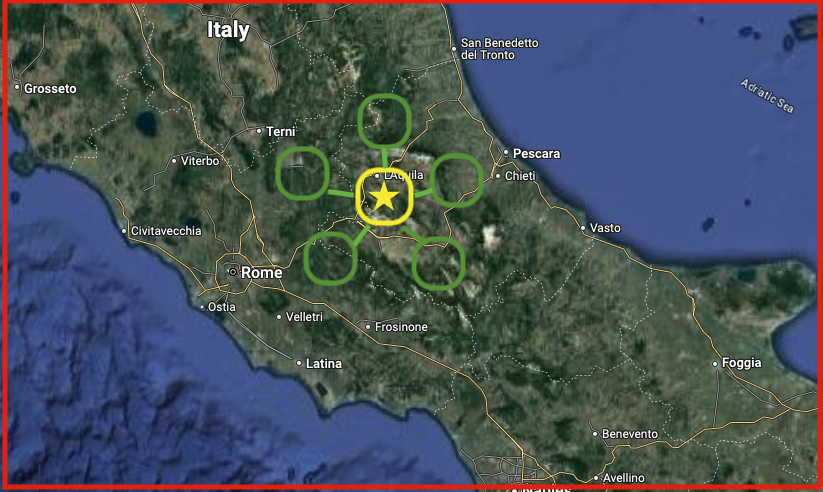}
\end{minipage}%
\begin{minipage}{.5\textwidth}
  \centering
  \includegraphics[width=.8\linewidth]{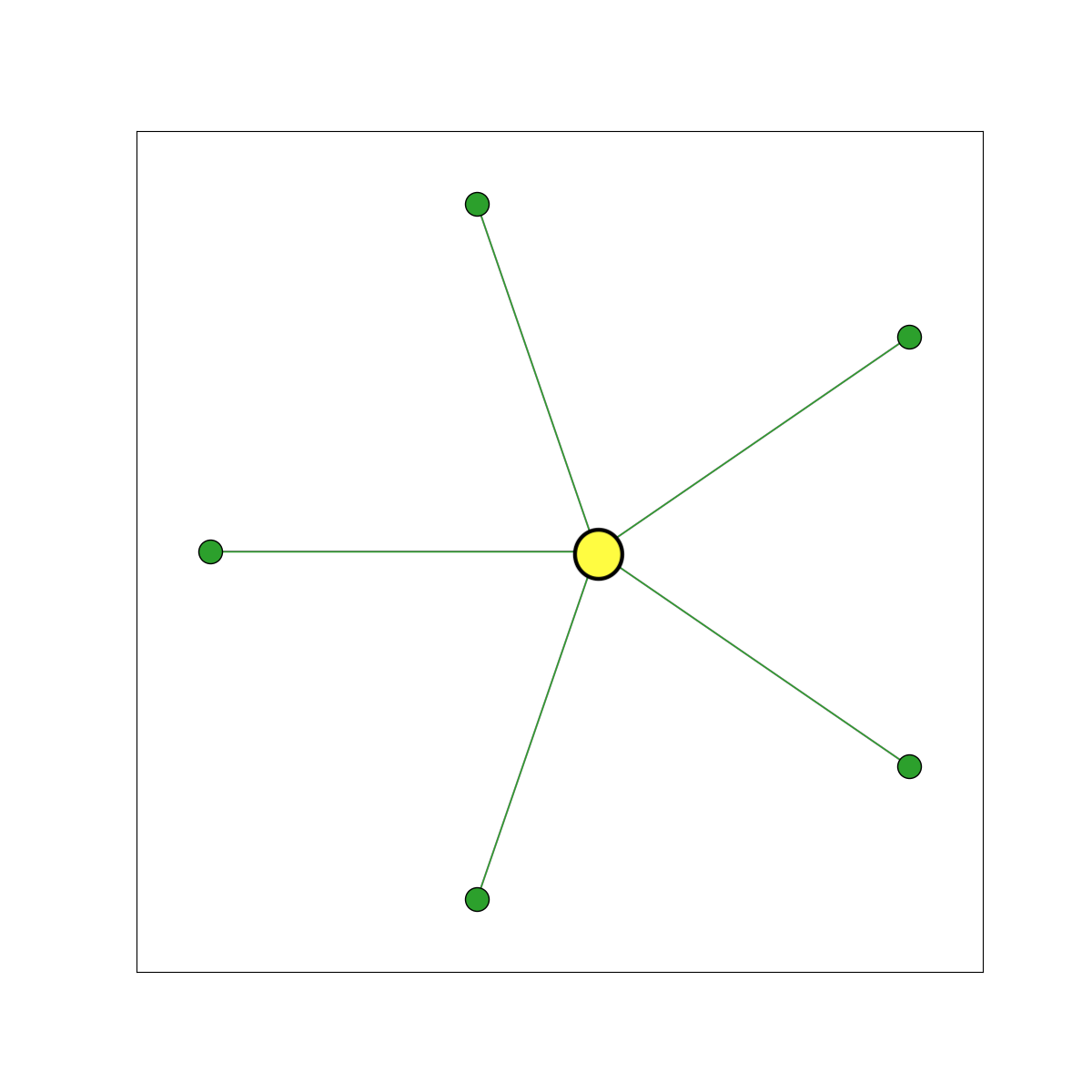}
\end{minipage}
\caption{Structure of the graph $S_k$ representing the areas $\{D_i\}_{i=0}^k$. In this case $k=5$. Left panel: example of a distribution of the regions on a map. The area $D_0$, in yellow, is the one where the mainshock -showed with a star - occurs. This phenomenon can have influence on other $5$ areas around $D_0$, colored in green.  In red it is outlined the  area $D$ taken into account for the  study of seismicity. Right panel:  representation as a star network $S_5$. The root of the graph is the region $D_0$ - in yellow - and the other nodes - in green - are the areas $D_i$, $i=1,...,5$.}
\label{fig:step_one}
\end{figure}

We can proceed one step further, assuming that a possible occurrence of an earthquake inside any area $D_i$ with $i=1,..., k$ may generate new events inside $D_0$, showing the mutual interaction between the regions. Moreover, it can also cause events in other $k$ nearby areas $D_j$ with $j=1,..., k$. In fact, following the ETAS model ($\cite{etas-ogata-88}$, $\cite{etas-ogata-89}$, $\cite{etas-ogata-98}$, $\cite{etas-ogata-99}$), any aftershock, in turns, can cause other subsequent events. These new earthquakes can take place in other regions there around. We assume that the occurrences in each area in $S_k$ can affect what happens inside new $k$ regions. Thus, we add to the graph $S_k$ new nodes and edges, connecting each leaf in $S_k$ to other $k$ vertices, considering them as mutual influences between areas as shown in Fig. $\ref{fig:step_two}$.

\begin{figure}
\centering
\begin{minipage}{.5\textwidth}
\centering
\includegraphics[width=1.1\linewidth]{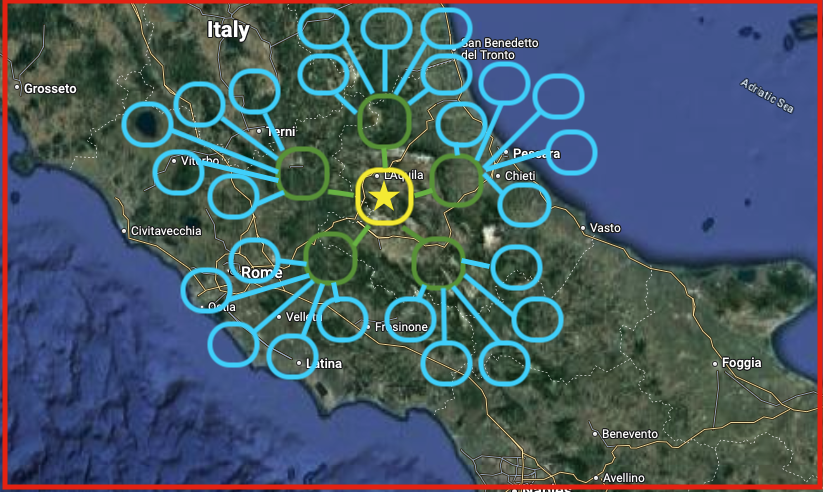}
\end{minipage}%
\begin{minipage}{.5\textwidth}
\centering
\includegraphics[width=0.8\linewidth]{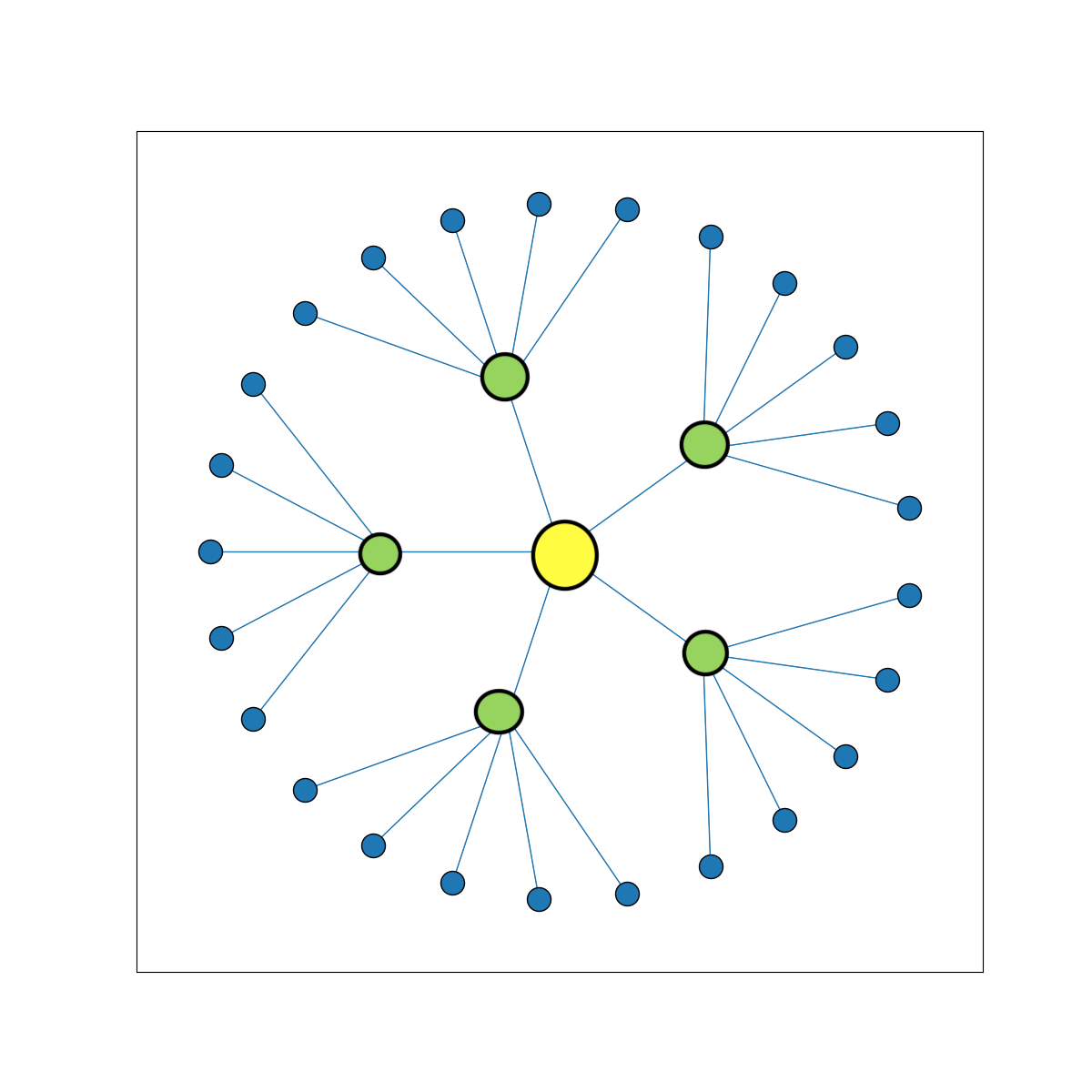}
\end{minipage}
\caption {Structure of the graph $S_k$ representing the areas $\{D_i\}_{i=0}^{1+k+k^2}$. In this case $k=5$. Left panel: example of a distribution of the regions on a map. The area $D_0$, in yellow, is the one where the mainshock -showed with a star - occurs. This phenomenon can have influence on other $5$ areas around $D_0$, colored in green. In this plot they are drawn all around $D_0$, but this representation doesn't show the real geology of the fault plane in central Italy where the earthquakes can occur. In red it is outlined the  area $D$ taken into account for the  study of seismicity. Right panel:  general representation of as a network. The root of the graph is the region $D_0$ - in yellow - and the other nodes are the areas $D_i$, $i=1,...,30$.}
\label{fig:step_two}
\end{figure}

\subsection{Brownian motions}
The description of an earthquake is carried out via suitable combination of Brownian motion functionals. As mentioned above, an earthquake can not be directly described by a Brownian motion. However, Brownian motions can be considered to describe some features of an earthquake.     

A new earthquake occurs as the stress of the tectonic plates overcomes their friction, so there is a release of energy in waves that travel through the earth's crust.
\begin{definition}
We define a load state $E$ as the level of cumulative elastic strain in the lithosphere that creates the seismic wave $W$.
\end{definition}
The	load state (some ground level immediately after an event) increases steadily over time, reaches a failure threshold,  and relaxes instantaneously back to ground level at the next earthquake time. We mainly deal under the assumption that a drifted Brownian motion
\begin{align*}
    B^E_t - mt 
\end{align*}
can be considered to describe the release of the accumulated energy $E=\{E_t\}_{t\geq0}$ and a drifted Brownian motion 
\begin{align*}
 B^W_t + vt   
\end{align*}
can be considered to describe the wave propagation $W=\{W_t\}_{t\geq 0}$ along a selected direction represented by an edge. In particular, given the sequences $E^n$ and $W^n$ together with the times $\tau^{E^n}$ and $\tau^{W^n}$, we write
\begin{align}
\mathbf{E} [f(E^n_t), t < \tau^{E^n}] \to  \mathbf{E}[ f(B^E_t - mt), t < \tau^E], \quad f \in C_b(\mathbb{R})
\label{EtoRDBM}
\end{align}
and
\begin{align}
\mathbf{E} [f(W^n_t), t < \tau^{W^n}] \to \mathbf{E} [f(B^W_t + vt), t < \tau^W], \quad f \in C_b(\mathbb{R})
\label{WtoRDBM}
\end{align}
as the sample size $n \to \infty$. Here, $\tau^{E^n}$ is the time to reach the failure threshold and $\tau^{W^n}$ is the time for the next earthquake. We therefore justify via central-limit arguments that small and independent contributions, as Gaussian increments, provide a mathematical model for $E$ and $W$. The drift $m$ represents a magnitude in a given region. The drift $v$ represents the velocity for the observed region: this physical quantity rules how fast the seismic waves propagate, depending on the Earth's internal structure, material composition and physical states. The velocity $v$ changes as the waves traverse different areas around the world and different Earth's layers. The speed differs according to each layer's properties and also to the respective temperature, composition and pressure.

We provide a description of an earthquake on a region associated with $S_k$. In particular, we consider the couple $(E,W)$ characterizing the earthquake and provide a description via reflected drifted Brownian motions associated with vertices and edges.


%

\subsection{The Mathematical setting}

 In order to streamline the notation as much as possible we write
\begin{align*}
\dot{u} = \frac{\partial u}{\partial t}, \quad u^\prime = \frac{\partial u}{\partial x}, \quad u^{\prime \prime} = \frac{\partial^2 u}{\partial x^2}.
\end{align*}
The symbols $D^\Phi_t$, $D^\Psi_t$, $\mathbf{D}^\Upsilon_x$ denote the non-local (integral) operators we deal with, we provide a detailed discussion in Section \ref{app:NLOs}. The operators $\mathsf{G}_\mu$ and $G_\mu$ will be properly defined below respectively together with the generator of $\mathsf{Q}$ on the graph $\mathsf{S}$ (see Section \ref{sec:BMonG}) and the generator of $X^\mu$ on the half line $[0, \infty)$ (see Section \ref{sec:RDBM}).\\ 

Our model relies on the fact that \lq\lq\textit{Brownian motions on metric graphs are useful in describing some earthquake features}\rq\rq. Based on this, we provide a construction in terms of the process $\mathsf{Q}$ on the star graph $\mathsf{S}_\nu$ driven by the problem
\begin{equation}
\left\lbrace
\begin{array}{ll}
\displaystyle \dot{u}(t, \mathsf{x}) = \mathsf{G}_\mu u(t, \mathsf{x}), & \mathsf{x} \in \mathsf{S} \setminus \{\mathsf{v}\}, \quad \mu>0,\\
\\
\displaystyle \eta \,D^\Phi_t u(t, \mathsf{v}) = \sum_{\varepsilon \in \mathcal{E}} \rho_\varepsilon\, u^\prime_\varepsilon(t, 0) - c \,u(t, \mathsf{v}), & t>0, \quad \eta >0, \quad \varrho_\varepsilon \in (0,1), \quad c \geq 0,\\  
\displaystyle u_\varepsilon(t, \ell)=0, & t>0, \quad \ell >0, \quad \varepsilon \in \mathcal{E},\\ 
\\
\displaystyle u(0, \mathsf{x}) = f(\mathsf{x}), & \mathsf{x} \in \mathsf{S}, \quad f \in C(\mathsf{S}).
\end{array}
\right.
\label{NLBVPGraphINTRO}
\end{equation}
A point of the star graph $\mathsf{S} = \mathsf{S}_\nu$ will be denoted by $\mathsf{x}=(\varepsilon, x)$ where $\varepsilon$ stands for the ray of $\mathsf{S}_\nu$ and $x$ gives the distance from the vertex $\nu$ (along the ray $\varepsilon$). We denote by $\mathcal{E}=\{\varepsilon\}$ the set of rays (of finite length) and by $\mathcal{V}=\{\nu\}$ the set of vertices by means of which we construct the family of star graphs $\{\mathsf{S}_\nu,\, \nu \in \mathcal{V}\}$. The analysis on star graphs can be  extended to the network
\begin{align}
    \mathsf{G} = \bigcup_{\nu \in \mathcal{V}} \mathsf{S}_\nu
\end{align}
which is the metric graph characterizing a geographical area affected by earthquakes. \\

We exploit the equivalence (see Theorem \ref{thm:equiv}) of $\mathsf{Q}$ with the vector $(\Theta, X^{[\nu]}, X^{[\varepsilon]})$ in order to describe an earthquake with energy $E$ and seismic wave propagation $W$. The process $\Theta$ is the edge selector, that is $\mathsf{Q}$ hits the star vertex and moves on a given edge according with $\Theta$. The processes $X^{[\varepsilon]}$ and $X^{[\nu]}$ will be respectively associated with the edges (the seismic wave propagation $W$ on a region/direction) and the vertices (the energy accumulation $E$ on a site/epicenter). In particular, $X^{[\varepsilon]}$ on $[0, \ell)$ and $X^{[\nu]}$ on $[0, \infty)$ are respectively driven by the problems
\begin{equation}
\left\lbrace
\begin{array}{ll}
\displaystyle \dot{u}(t,x) = G_\mu u(t,x), & t>0,\, x \in (0, \ell), \quad \mu>0,\\ 
\\
\displaystyle \eta_\varepsilon  D^\Phi_t u(t, 0) =  u^\prime(t, x) |_{x=0} - c u(t, x), & t>0, \quad \eta_\varepsilon >0,\; c >0,\\
\\
\displaystyle u(t, \ell) = 0, & t>0,\\
\\
\displaystyle u(0, x) = f(x), & x \in [0, \ell), \quad f \in C[0, \ell)
\end{array}
\right .
\label{NLBVPedges}
\end{equation}
and
\begin{equation}
\left\lbrace
\begin{array}{ll}
\displaystyle \dot{u}(t,x) = G_\mu u(t,x), & t>0,\, x \in (0, \infty), \quad \mu<0,\\ 
\\
\displaystyle \eta_\nu D^\Psi_t u(t, 0) = \mathbf{D}^\Upsilon_x u(t, x) |_{x=0}, & t>0, \quad \eta_\nu > 0,\\
\\
\displaystyle u(0, x) = f(x), & x \in [0, \infty), \quad f \in C_b[0, \infty).
\end{array}
\right .
\label{NLBVPnodes}
\end{equation}
The problem \eqref{NLBVPedges} says that $X^{[\varepsilon]}$ is a drifted Brownian motion killed at $\ell$ and slowly reflected at $0$ (see Theorem \ref{thm:NLBVPwithDir}). The problem \eqref{NLBVPnodes} says that $X^{[\nu]}$ is a drifted Brownian motion jumping away the boundary point $0$ to an holding point on $(0, \infty)$ (see Theorem \ref{thm:NLBVP}). The process $\mathsf{Q}$ describes the wave propagation in a given region as well as in a geographical area under the following 

\begin{assumption}
The seismic wave is described by $\mathfrak{N} \in \mathbb{N}$ waves and propagates continuously on the network of $N \leq \mathfrak{N}$ star graphs (regions). Simultaneous waves on different edges (in different directions) are not allowed.
\label{ASS1}
\end{assumption}

We stress the fact that simultaneous waves are not allowed only for the sake of simplicity. Indeed, our model can be extended to the case of multiple processes under some rule for the interaction to be specified. A simple (maybe unrealistic) one would be given by independence between waves. A detailed description of the process $\mathsf{Q}$ on $\mathsf{G}$ will be given in Section \ref{sec:MetricG} whereas, a discussion on $X^{[\varepsilon]}$ and $X^{[\nu]}$ will be given in Section \ref{sec:NLBVPs}. Here we only underline that both processes are driven by non-local boundary value problems introducing jumps and holding times.\\

Our discussion will be given without long proofs. Our feeling is that the presentation of the results in the present work becomes fluent and pleasant for a broad audience of researchers. Thus, we decided to postpone the most part of the proofs in the Appendix. In the same spirit, we collect some Figures at the end of the work.\\

For the reader's convenience we list the main objects we deal with further on:
\begin{itemize}
\item $T_\mu$ is an exponential random variable such that $\mathbf{P}(T_\mu > t)= e^{-\mu t}$, $\mu>0$;
\item $\mathcal{H}=\mathcal{H}^\Phi$ is the tempered subordinator with symbol $\Phi(\lambda)$;
\item $\mathcal{L}=\mathcal{L}^\Phi$ is the inverse of $\mathcal{H}$;
\item $\mathcal{H}^\dagger$ is the killed tempered subordinator with symbol $\Phi^\dagger(\lambda)$;
\item $\mathcal{L}^\dagger$ is the inverse of $\mathcal{H}^\dagger$; 
\item $H =H^\Phi$ is a stable subordinator with symbol $\Phi(\lambda)=\lambda^\alpha$, $\alpha \in (0,1)$;
\item $L$ is the inverse of $H$;
\item $X^{\mu}$ is a Brownian motion on $[0, \infty)$ with drift $\mu$ reflected at zero;
\item $\gamma^\ell_t(Z)$ is the local time at level $\ell\geq 0$ for $Z$. We often write $\gamma_t$ meaning $\gamma^0_t(X^\mu)$;
\item $\mu$ is a real constant unless otherwise specified,
\item $\zeta$ denotes lifetimes,
\item $\tau$ denotes first hitting/exit times,
\item $\Phi$ is the symbol given in \eqref{symbPhi} below,
\item $\Psi$ is the symbol of an independent subordinator $\mathcal{H}^\Psi$. 
\end{itemize}
We denote by $\mathbf{E}_x$ the expected value with respect to the probability measure $\mathbf{P}_x$ where $x$ is the starting point for a given process. For a process, say $Z=\{Z_t\}_{t\geq 0}$, we often write 
\begin{align}
Z_t = Z \circ t \quad \textrm{and} \quad Z_{T_t} = Z \circ T_t    
\end{align}
for $t\geq 0$ and $T_t\geq 0$ respectively denoting deterministic and random times.

\section{Preliminaries}

\subsection{Tempered subordinators}
For the subordinator $H=\{H_t\}_{t\geq 0}$ on $[0, \infty)$ we recall that $H$ started at $H_0=0$ is a strictly increasing process with symbol $\Phi(\lambda) = \lambda^\alpha$, $\alpha \in (0,1)$ and $\mathbf{E}_0 [e^{-\lambda H_t}] = e^{-t \lambda^\alpha}$, $\lambda >0$. The process $L=\{L_t\}_{t\geq 0}$ is defined as
\begin{align*}
L_t : = \inf\{s\,:\, H_s >t\} = \inf\{s\,:\, H_s  \notin (0,t)\}
\end{align*}
that is the inverse process of $H$ and the first exit time of $H$ from $(0,t)$. The last identity holds, in general, for strictly increasing processes. We have the relation $\mathbf{P}_0(H_t > s) = \mathbf{P}_0(L_s < t)$ for $t,s>0$. 
The tempered stable subordinator $\mathcal{H}=\{\mathcal{H}_t\}_{t\geq 0}$ is a process on $[0, \infty)$ with $\mathbf{E}_0[e^{-\lambda \mathcal{H}_t}] = e^{-t \Phi(\lambda)}$, $\lambda>0$ where $\Phi$ is the Bernstein symbol
\begin{align}
\Phi(\lambda) = \sqrt{\lambda + \theta} - \sqrt{\theta} \quad \textrm{with } \quad \theta = (\mu/2)^2, \quad \lambda>0
\label{symbPhi}
\end{align}
characterizing uniquely $\mathcal{H}$. Thus, the process $\mathcal{H}$ is termed tempered subordinator of order $1/2$. As $\mu=0$ we get the stable subordinator $H$ of order $\alpha=1/2$. The inverse process $\mathcal{L}=\{\mathcal{L}_t\}_{t\geq 0}$ is defined as
\begin{align*}
\mathcal{L}_t : = \inf\{s\,:\, \mathcal{H}_s>t\}.
\end{align*}
It holds that $\mathbf{P}_0(\mathcal{H}_t > s) = \mathbf{P}_0(\mathcal{L}_s < t)$ for $t,s>0$ and $\mathcal{H}_0=0$ which implies $\mathcal{L}_0=0$. We also focus on the symbol
\begin{align}
\Phi^\dagger(\lambda) = \mu + \Phi(\lambda), \quad \mu>0
\end{align}
characterizing uniquely $\mathcal{H}^\dagger$ as the killed subordinator
\begin{align}
    \mathcal{H}^\dagger_t = \left\lbrace
        \begin{array}{ll}
             \mathcal{H}_t, & t < T_\mu,  \\
             +\infty, & t\geq T_\mu .
        \end{array}
        \right .
        \label{zeroHittingPosDrift}
\end{align}
It is well-known that
\begin{align*}
\mathbf{P}_0(\mathcal{H}_\ell \in dz) = \frac{\ell}{z} \frac{e^{-(\ell-\mu z)^2 /(4z)}}{\sqrt{4\pi z}} dz, \quad z\geq 0,\; \ell >0,\; \mu >0
\end{align*}
that is, $\mathcal{H}_\ell$ is an inverse Gaussian variable with parameters $\ell,\mu$. This follows from  
\begin{align*}
\mathbf{E}_0[e^{-\lambda \mathcal{H}_\ell}] 
= & \int_0^\infty e^{-\lambda z} \mathbf{P}_0(\mathcal{H}_\ell \in dz) 
= \ell \frac{e^{\ell \mu /2}}{\sqrt{4\pi}} \int_0^\infty  z^{1/2 - 1} e^{-z \ell^2 /4  } e^{-(\lambda + \mu^2/4)/z} dz
\end{align*}
which leads to a Modified Bessel function of the second kind and from which
\begin{align}
\mathbf{E}_0[e^{-\lambda \mathcal{H}_\ell}] 
= & 2\ell \frac{e^{\ell \mu/2}}{\sqrt{2\pi}} \left( \frac{\lambda+\mu^2/4}{\ell^2/4} \right)^{1/4} K_{1/2}\left( \sqrt{(\lambda+\mu^2/4) \ell^2} \right) \notag \\
= & e^{- \ell \left( \sqrt{\lambda +\mu^2/4} - \mu/2 \right) }
\label{genFunOfH}
\end{align}
where we used the identity $K_{1/2}(z) = \sqrt{\pi/2 z} \, e^{-z}$. We observe that in case of strictly increasing subordinators (with infinite activity) if the process starts from zero, then it immediately jumps away never to return.

\subsection{Reflected drifted Brownian motions}
\label{sec:RDBM}
The process $\widetilde{X}^{\mu}= \{\widetilde{X}^{\mu}_t\}_{t\geq 0}$ on $[0, \infty)$ is an elastic drifted Brownian motion with generator  $(G_\mu, D(G_\mu))$ where $G_\mu \varphi = \varphi^{\prime \prime} + \mu \varphi^\prime$ and
\begin{align*}
	D(G_\mu) = \left\lbrace \varphi, G_\mu \varphi \in C_b((0, \infty))\,:\, \varphi^\prime (0^+) = c \, \varphi(0^+) \right\rbrace.
\end{align*}
The constant $c>0$ is the elastic coefficient. The semigroup 
\begin{align*}
	P^\mu_t f(x) = \int_0^\infty f(y) p(t,x,y)dy
\end{align*}
generated by $(G_\mu, D(G_\mu))$ has the probabilistic representations
\begin{align}
	P^\mu_t f(x) 
    = \mathbf{E}_x[f(\widetilde{X}^\mu_t)] 
    = \mathbf{E}_x[f(\widetilde{X}^\mu_t), \, t < \zeta^\mu] = \mathbf{E}_x[f(X^\mu_t) M_t^\mu]
    \label{solmuSemig}
\end{align}
where $X^\mu_t$ is a drifted Brownian motion on $[0, \infty)$ reflected at $0$ and  the Robin boundary condition introduces the lifetime $\zeta^\mu$ of $\widetilde{X}^\mu$ and the multiplicative functional $M^\mu_t = \exp \left( -c\, \gamma^0_t(X^\mu)\right)$ of $X^\mu$. The process $\gamma^0(X^\mu) = \{\gamma^0_t(X^\mu)\}_{t\geq 0}$ is the local time of $X^\mu$ at the boundary point zero. The transition kernel has the explicit form
\begin{align}
	\label{density:pmu}
	p(t,x,y) =  e^{-\frac{\mu^2}{4}t} e^{\frac{\mu}{2}(y-x)} \bigg( p_1(t,x,y) - p_2(t,x,y) 	\bigg)
\end{align}
for $ x\geq 0,\; y>0, t>0 $, where $g(t,z)=e^{-z^2/4t}/\sqrt{4\pi t}$, 
\begin{align}
p_1(t,x,y) = g(t, x-y) + g(t, x+y),
\end{align}
\begin{align}
p_2(t,x,y) 
= & 2 \left(c + \frac \mu 2 \right) \int_0^\infty e^{-\left(c + \frac \mu 2 \right) \,  w} g(t, w+x+y) dw
\end{align}
and $c \geq 0$. 

We now discuss some properties and provide some results to be considered below. Part of the results has been also obtained in \cite{DovIafSpa24} by following different arguments. 

The symbol $\stackrel{d}{=}$ stands for equality in distribution.
\begin{theorem}
We have that,
\begin{align*}
\forall\, t \geq 0, \quad \gamma^0_t(X^\mu)\, | \,( X^\mu_0 =0) \stackrel{d}{=} \left\lbrace
\begin{array}{ll}
\displaystyle \mathcal{L}_t, & \mu\leq 0\\
\displaystyle \mathcal{L}^\dagger_t \stackrel{d}{=} \mathcal{L}_t \wedge T_\mu, & \mu>0
\end{array}
\right.
\end{align*}
where $T_\mu$ is independent from $\mathcal{L}$.
\label{thm:localTime}
\end{theorem}
Proof postponed, see Section \ref{proof-thm:localTime}.\\

We recall that a standard (zero drifted) Brownian motion crosses the starting point infinitely many times. In particular, for a process started from zero, the (zero) set of hitting times is a perfect set. It is closed, has no isolated points, and contains no intervals. In particular, it is a fractal set, the Lebesgue measure is zero and the Hausdorff dimension is $1/2$. Figure \ref{fig:RDBMrandomLevel} shows the paths (simply approximated by random walks) of a reflected Brownian motion with drift stopped at different random levels.\\

Let us consider the first hitting time $\tau_0^\mu = \inf\{t\,:\, X^\mu_t =0  \}$.
\begin{theorem}
We have that, for $x \in (0, \infty)$,
\begin{align*}
\tau_0^\mu \, |\,  (X^\mu_0=x) \stackrel{d}{=} \left\lbrace
\begin{array}{ll}
\displaystyle \mathcal{H}_x, & \mu\leq 0,\\
\displaystyle \mathcal{H}^\dagger_x, & \mu> 0.
\end{array}
\right.
\end{align*}
\label{thm:hittingTime}
\end{theorem}
Proof postponed, see Section \ref{proof-thm:hittingTime}.\\

We can immediately check that, given the level $x> 0$, 
\begin{align}
    \textrm{for } \mu < 0, \quad \mathbf{E}_0[\mathcal{H}_x] =  \frac{x}{|\mu|}
    \label{meanTimeH}
\end{align}
whereas
\begin{align}
    \textrm{for } \mu > 0, \quad \mathbf{E}_0[\mathcal{H}^\dagger_x | x < T_\mu] = \frac{x}{\mu} e^{-\mu x}.
    \label{meanTimeHkilled}
\end{align}
For the first hitting time 
$$\tau_\ell^\mu = \inf\{t\,:\, X^\mu_t =\ell \in (0, \infty)  \}$$ 
of the level $\ell$ we observe that
\begin{align}
\int_0^\infty e^{-\lambda t} \mathbf{E}_x \left[f(X^\mu_t), t < \tau^\mu_\ell \wedge \zeta^\mu \right] dt
    = & \mathbf{E}_x \left[ \int_0^{\tau^\mu_\ell \wedge \zeta^\mu} e^{-\lambda t} f(X^\mu_t) dt \right] \label{tauF1}\\
    = & \mathbf{E}_x \left[ \int_0^{\tau^\mu_\ell} e^{-\lambda t - (c+\frac{\mu}{2}) \gamma_t} f(X^\mu_t) dt \right] \label{tauF2}
\end{align}
solves the problem to find $u=u(\lambda, x)$ such that 
\begin{align*}
    \left\lbrace
    \begin{array}{ll}
        G_\mu u - \lambda u = - f, & \textrm{in } [0, \ell)\\
        u^\prime=cu, & \textrm{on } x=0\\
        u=0, & \textrm{on } x=\ell
    \end{array}
    \right.
\end{align*}
where $f(x)=u(0, x)$ is the initial datum and
\begin{align*}
    u(\lambda, x) = \int_0^\infty e^{-\lambda t} u(t, x) dt, \quad \lambda>0.
\end{align*}
Formula \eqref{tauF1} says that
\begin{align}
    u(\lambda, x) = \frac{1}{\lambda} \mathbf{E}_x \left[ 1 - e^{-\lambda (\tau^\mu_\ell \wedge \zeta^\mu) } \right] \to \mathbf{E}_x[\tau^\mu_\ell \wedge \zeta^\mu], \quad \textrm{as } \lambda\to 0
    \label{meanMinimum}
\end{align}
and formula \eqref{tauF2} gives
\begin{align*}
    u(\lambda, x) = \mathbf{E}_x \left[ \int_0^{\tau^\mu_\ell} e^{-\lambda t - (c + \frac{\mu}{2}) \gamma_t} dt  \right] \to \mathbf{E}_x \left[ \int_0^{\tau^\mu_\ell} e^{- (c + \frac{\mu}{2}) \gamma_t} dt  \right], \quad \textrm{as } \lambda \to 0.
\end{align*}

\begin{theorem}
For $\lambda>0$, $x \in [0, \ell)$, we have that
\begin{align}
    \mathbf{E}_x \left[ \int_0^{\tau^\mu_\ell} e^{-\lambda t - (c + \frac{\mu}{2}) \gamma^\mu_t} dt \right] = \frac{\ell - x}{\mu} + \frac{1+c\ell}{\mu} \frac{e^{(x-\ell) \Phi^-} - e^{(\ell-x)\Phi^+}}{(c + \Phi^+) e^{\ell \Phi^+} -  (c - \Phi^-) e^{\ell \Phi^-} }
\end{align}
where $\Phi^+$ and $\Phi^-$ stand for
\begin{align}
    \Phi^+(\lambda) = \sqrt{\lambda + \mu^2/4} + \mu/2 \quad \textrm{and} \quad \Phi^-(\lambda) = \sqrt{\lambda + \mu^2/4} - \mu/2=: \Phi(\lambda).
\end{align}
    \label{thm:tauLEVELgen}
\end{theorem}
\begin{proof}
For $\lambda>0$, $x \in [0, \ell)$, the problem to find a solution $u=u(\lambda, x)$ to
\begin{align*}
    \left\lbrace
    \begin{array}{ll}
        u^{\prime \prime} + \mu u^\prime = \lambda u - 1, & \textrm{in } [0, \ell)\\
        u^\prime=cu, & \textrm{on } x=0\\
        u=0, & \textrm{on } x=\ell
    \end{array}
    \right.
\end{align*}
is standard in the context of second order differential equations. The solution is unique and the probabilistic representation comes from \eqref{tauF1} and \eqref{tauF2}.
\end{proof}

As a by-product of the theorem above we obtain the following result.
\begin{corollary}
Let $c=0$. We have that
\begin{align}
    \mathbf{E}_x[\tau^\mu_\ell] = \frac{\ell - x}{\mu} - \frac{e^{-\mu x} - e^{-\mu \ell}}{\mu^2}, \quad x \in [0, \ell), \; \mu>0.
    \label{meantauLEVEL}
\end{align}
\label{coro:tauLEVEL}    
\end{corollary}
\begin{proof}
    From the equation $u^{\prime \prime} + \mu u^\prime = - 1$ we arrive at $u^{\prime} + \mu u = - x + C$ which can be easily solved as a first order ODE. After some calculation, we get that \eqref{meantauLEVEL} is the solution to $u^{\prime \prime} + \mu u^\prime = - 1$ with $u(\ell)=0$, $u^\prime(0)=0$. Moreover, we can obtain the same result by taking the limit $\lambda \to 0$ in Theorem \ref{thm:tauLEVELgen} with $c=0$.
\end{proof}
Notice that, from Corollary \ref{coro:tauLEVEL}, as $\mu \to 0$, $\mathbf{E}_x[\tau^\mu_\ell] \to \mathbf{E}_x[\tau^0_\ell] = (\ell^2 - x^2)/2$ where $\tau^0_\ell$ is the hitting time of a reflected Brownian motion on $[0, \infty)$. 

\section{Non-local boundary conditions}
We study the problems \eqref{NLBVPedges} and \eqref{NLBVPnodes} and the probabilistic representation of the solutions.

\subsection{Non local operators}
\label{app:NLOs}
We introduce briefly the non-local operators we deal with. Such operators are based on the well-known Marchaud (\cite{Dzh66,DzhNers68}) and Caputo-Mainardi-D\v{z}rba\v{s}jan (\cite{caputoBook,CapMai71a, CapMai71b}) operators. Consider the symbol
\begin{align}
    \Phi(\lambda) = \int_0^\infty (1-e^{-\lambda y}) \Pi^\Phi(dy), \quad \lambda\geq 0.
\end{align}
For a continuous (causal) function $u$ extended with zero on the negative part of the real line, that is $u(t, x)=0$ if $x \leq 0$, $\forall\, t\geq 0$, in case
\begin{align*}
x \mapsto u(t, x) \textrm{ is locally Lipschitz and belongs to the set } C_b(0, \infty),
\end{align*}
we define the Marchaud (type) derivatives
\begin{align}
\mathbf{D}_{x\mp}^\Phi u(t, x)=\int_0^\infty (u(t, x) - u(t, x\mp y))\Pi^\Phi(dy).
\end{align}
Indeed, for $u$  bounded and locally Lipschitz continuous,
\begin{align}
\label{stimaMarchaud}
\vert \mathbf{D}_{x\mp}^\Phi u(t, x) \vert &\leq \int_0^1 \vert u(t, x) -u(t, x\mp y) \vert \Pi^\Phi(dy) + \int_1^\infty \vert u(t, x) - u(t, x\mp y) \vert \Pi^\Phi(dy) \notag \\
&\leq K \int_0^1 y \Pi^\Phi(dy) + 2 \vert\vert u(t, \cdot) \vert \vert_\infty \int_1^\infty \Pi^\Phi(dy) \notag \\
&\leq(K+2 \vert \vert u(t, \cdot) \vert \vert_\infty) \int_0^\infty (1 \wedge y) \Pi^\Phi(dy) <\infty.
\end{align}
The last inequality emerges directly from $\int (1 \wedge z) \Pi^\Phi(dz) < \infty$ which holds for the L\'{e}vy measure $\Pi^\Phi$.\\

A condition for the Caputo-D\v{z}rba\v{s}jan (type) operator $D^\Phi_t$ to be well defined is given by requiring that, $\forall\, x \in \mathbb{R}$, 
\begin{align*}
t \mapsto u(t, x)  \textrm{ belongs to the set }  W^{1,\infty}(0, \infty)
\end{align*}
of essentially bounded functions with essentially bounded derivatives. This requirement guarantees the application of the Laplace machinery. Indeed, $D^\Phi_t$ is a convolution-type operator and
\begin{align}
\int_0^\infty e^{-\lambda t} D^\Phi_t u(t, x)\, dt 
& = \big(\lambda u(\lambda, x) - u(0,x) \big) \left( \int_0^\infty e^{-\lambda t} \phi(t)dt \right)
\label{LapDPhi}
\end{align}
where (\cite{Bertoin1999}) $\phi(t) = \Pi^\Phi(t, \infty)$ is the tail of $\Pi^\Phi$ with
\begin{align}
\int_0^\infty e^{-\lambda t} \phi(t)dt = \frac{\Phi(\lambda)}{\lambda}, \quad \lambda>0
\label{tailLap}
\end{align} 
and $u(\lambda, x)$ is the Laplace transform of $u(t,x)$. From \eqref{LapDPhi} and the Young's inequality we get
\begin{align}
\|D^\Phi_t u(\cdot, x) \|^p_p \leq \| u(\cdot, x) \|^p_p \left( \lim_{\lambda \to \infty} \frac{\Phi(\lambda)}{\lambda} \right)^p, \quad p \in [1, \infty)
\label{observeFiniteMean}
\end{align}
If $u, \dot{u}$ are bounded $\forall\, x$, then the Laplace transforms of $u, \dot{u}$ are well-defined. 

We introduce a further characterization by asking for the condition 
\begin{align}
\label{condMD}
\exists\, M>0\,:\, \bigg| \frac{\partial u}{\partial s}(s,x) \bigg| \leq  M\, \frac{\kappa(ds)}{ds}
\end{align} 
where 
\begin{align*}
\kappa(ds) = \int_0^\infty \mathbf{P}_0(H_t \in ds) dt
\end{align*}
is the potential measure for the subordinator $H$ with symbol $\Phi$. Since $\kappa$ and $\phi$ are associated Sonine kernels for which 
\begin{align*}
\int_0^t \phi(t-s) \kappa(ds) =1
\end{align*} 
and
\begin{align}
\label{unifBoundD}
| D^\Phi_t u(t, x)| \leq M \int_0^t \phi(t-s) \kappa(ds),
\end{align} 
then we obtain that $|D^\Phi_t u(t,x)|$ is uniformly bounded on $(0, \infty) \times \mathbb{R}$ under \eqref{condMD}. We also consider the Caputo-D\v{z}rba\v{s}jan (type) operator $D^\Psi_t u$ written in terms of the kernel $\psi$.

\subsection{On the dynamic boundary value problems}
Let us briefly discuss on the problems $\eqref{NLBVPedges}$ and \eqref{NLBVPnodes}. We first  recall the problem 
\begin{align}
    \left\lbrace
     \begin{array}{ll}
          \dot{u}(t,x) = u^{\prime \prime}(t,x) + \mu u^\prime(t,x), & t>0,\; x \in (0, \infty) \\
          \dot{u}(t, 0) = \sigma u^\prime(t, 0), & t>0,\\
          u(0, x) = f(x), & x \in [0, \infty)],
     \end{array}
    \right .
    \label{SWpde}
\end{align}
with $\mu, \sigma \in \mathbb{R}$ treated by Stroock and Williams in \cite{SW05,SW06}. They provided a description based on the characterization of $\sigma$ from the analysis and probability point of view. In case $\sigma>0$ the associated process is a drifted Brownian motion with sticky reflection at $\{0\}$. In case $\sigma < 0$, the associated process is a Ray process. It behaves like $X^\mu$ except in $\{0\}$ where it jumps away according with a given measure. Such behaviors are of interest in the present work. However, we proceed with a different characterization of the processes starting from the problems \eqref{NLBVPedges} and \eqref{NLBVPnodes}.\\

In the present work we focus on the case $\Phi$ defined as in \eqref{symbPhi}. Indeed, our aim is to model the seismic wave propagation also in terms of the drift $\mu$. If $\Phi(\lambda) = \lambda$, then $D^\Phi_t u$ becomes  $\dot{u}$, that is the ordinary derivative and our problem falls into the class of dynamic boundary value problems. Since $u^{\prime \prime} + \mu u^\prime \in C[0, \infty)$, the dynamic boundary condition above can be replaced by Feller-Wentzell (type) boundary condition $u^{\prime \prime}(t,0) + \mu u^\prime(t,0) = \sigma u^\prime(t, 0)$,  $t > 0$. Now we consider the problem
    \begin{align}
        \left\lbrace
        \begin{array}{ll}
             \dot{u}(t,x) = u^{\prime \prime}(t,x) + \mu u^\prime(t, x) & t>0,\;  x\in (0, \ell),\; \mu >0  \\
            \eta \Big( u^{\prime \prime} (t, 0)  + \delta \mu u^\prime(t, 0) \Big)= u^\prime (t, 0) & t > 0, \; \eta \geq 0,\; \delta \in \{0,1\}\\
            u(t, \ell) = 0 & t > 0\\
            u(0,x) = f(x) & f \in C(0, \ell).
        \end{array}
        \right .
        \label{PDEdelta}
    \end{align}
which is of interest in the present paper. The sticky behaviour prescribed in \eqref{PDEdelta} can be represented via time change. Let $X^{\delta, \mu}=\{X^{\delta, \mu}_t\}_{t\geq 0}$ be the drifted Brownian motion on $[0, \ell)$  driven by \eqref{PDEdelta}. That is, $X^{\delta, \mu}$ has generator $(G^\delta_\mu, D(G^\delta_\mu))$ with 
\begin{align*}
G^\delta_\mu = G^{\delta, 0}_\mu \quad \textrm{and} \quad D(G^\delta_\mu) = D(G^{\delta, 0}_\mu)
\end{align*}
where
\begin{align}
    G^{\delta, c}_\mu = G_\mu, \quad D(G^{\delta, c}_\mu) = \{ \varphi \in C^2[0, \infty)\,:\, \eta \varphi^{\prime \prime}  (0) + \eta \mu \delta \varphi^\prime (0) = \varphi^\prime(0) - c \, \varphi(0) \}.
    \label{genGmudelta}
\end{align}
Here we assume no elastic kill and $X^{\delta, \mu}_t = X^\mu \circ V^{-1}_t$ where $V_t = t + \eta(\delta) \gamma_t$ with stickiness parameter $\eta(\delta)$. Due to the Dirichlet condition we introduce the stopping time $\tau^\mu_\ell$. Then, we focus on the extra time $\gamma \circ \tau^\mu_\ell$ associated with the sticky behavior in case $\delta \in \{0,1\}$ for which we provide the following result.
\begin{theorem}
    For $\tau^{\delta, \mu}_\ell := \inf\{t\,:\, X^{\delta, \mu}_t =\ell \}$ we have that
        \begin{align}
    \mathbf{E}_x[ \tau^{\delta, \mu}_\ell - \tau^\mu_\ell  ] = \eta(\delta) \mathbf{E}_x[\gamma \circ \tau^\mu_\ell] = \frac{\eta}{1+ \eta (1-\delta) \mu} \frac{e^{-\mu x} - e^{-\mu \ell}}{\mu}, \quad x \in [0, \ell).
    \label{meanStoppedSWcase}
    \end{align}
        \label{thm:meanStoppedSWcase}
\end{theorem}
Proof postponed, see Section \ref{proof-thm:meanStoppedSWcase}.\\

We recall that $\gamma \circ \tau^\mu_\ell$, the local time accumulated up to $\tau^\mu_\ell$, for $\mu=0$, is an exponential r.v. with parameter $1/(\ell-x)$ (see \cite[Theorem 7.7]{ChugWilliams}). Accordingly, formula \eqref{meanStoppedSWcase} gives $\eta (\ell - x)$. We immediately see that the second derivative in the boundary condition guarantees the sticky behavior. On the other hand, as $\eta \to \infty$, the extra time in formula \eqref{meanStoppedSWcase} if finite. This means that $u^{\prime \prime}(t,0) + \delta \mu u^\prime (t, 0) = 0$ does not imply absorption for $\delta \neq 1$. In particular, \textit{the dynamic boundary condition plays a crucial role}.

The sticky condition introduces a sequence of i.i.d. holding times, say $\{e^\mu_i\}_i$. We mainly focus on the case $\delta=1$ for which we obtain \begin{align*}
     \eta \mathbf{E}_x[\gamma \circ \tau^\mu_\ell] = \mathbf{E}[e^\mu_0] \mathbf{E}_x[\gamma \circ \tau^\mu_\ell].
\end{align*}
Notice that (as proved in the next theorem)
\begin{align}
    \mathbf{P}(e^\mu_0 > t) = \mathbf{P}(e^\mu_0 > t \, | \, X^{1, \mu} \circ e^\mu_0 \in (0, \infty),\, X^{1, \mu} _0=0), \quad t>0
    \label{indHTfromX}
\end{align}
and
\begin{align*}
e^\mu_0 := \inf\{t\,:\, X^{1, \mu}_t >0\,|\, X^{1, \mu}_0 = 0\}
\end{align*}
is the first time $X^{1, \mu}$ hits the interior $(0, \ell)$, $\ell>0$. The zero set $\{t\, : \, X^{1, \mu}_t=0 \}$ has positive Lebesgue measure. The inverse of the associated local time is right-continuous and the jumping times define a countable set. Thus, the set of holding times (at zero) of $X^{1,\mu}$ is countable.

\begin{theorem}
For a (positively drifted) sticky Brownian motion, we have that
\begin{align*}
    \mathbf{P}(e^\mu_0 > t) = e^{-t/\eta }, \quad t>0.
\end{align*}
Moreover, $\{e^\mu_i\}_{i \in \mathbb{N}_0}$ is a sequence of i.i.d. random variables.  
\label{thm:HTrdBM}
\end{theorem}
Proof postponed, see Section \ref{proof-thm:HTrdBM}.\\

\subsection{Non-local (dynamic) boundary conditions for drifted Brownian motions}
\label{sec:NLBVPs}

We study the non-local boundary value problems (NLBVPs)  \eqref{NLBVPedges} and \eqref{NLBVPnodes} providing the probabilistic reading in terms of $X^\mu$. Recall that (see Section \ref{app:NLOs} for details)
\begin{align}
    D^\Phi_t u(t) = \int_0^t u^\prime(s) \phi(t-s)ds \quad \textrm{and} \quad D^\Psi_t u(t) = \int_0^t u^\prime(s) \psi(t-s)ds
    \label{CD-derivativesPhiPsi}
\end{align}
are the Caputo-D\v{z}rba\v{s}jan (type) derivatives with kernel $\phi$ and $\psi$ whereas
\begin{align}
\mathbf{D}^\Upsilon_{x} u(x) = \int_0^\infty \big( u(x) - u(x+z) \big) \kappa(dz)
\label{MderivativeUpsilon}
\end{align}
is the Marchaud (type) derivative with a kernel $\kappa$ which can be also singular. We underline that, in case $\kappa$ is non singular, 
\begin{align}
\mathbf{D}^\Upsilon_x u(x) 
= & c_J \int_0^\infty \big( u(x) - u(x+z) \big) d\, \mathbf{P}(J > z)
\label{NLBOperator}
\end{align}
for some random variable $J$ determining the jumps from the boundary. \\

Let us introduce the process $A=\{A_t\}_{t \geq 0}$ as the additive part determining the jumps from zero to a random point in $(0, \infty)$. It is defined as 
\begin{align}
    A_t := H^\Upsilon \circ L^\Upsilon_t - t, \quad t\geq 0.
    \label{jumpsL}
\end{align}
These jumps are controlled by $\Upsilon$, then we are able to manage random jump $J$ associated with \eqref{MderivativeUpsilon} or \eqref{NLBOperator}. The jump $J$ is deterministic in the special (degenerate) case in which $d\mathbf{P}$ is a Dirac measure. In case $\Upsilon$ is the symbol of subordinator with infinite activity, then $\mathbf{D}^\Upsilon_x u$ has the representation \eqref{MderivativeUpsilon} and $\kappa$ is a singular kernel. The representation \eqref{NLBOperator} is obtained in case $\Upsilon$ is the symbol of a subordinator with finite activity. The process \eqref{jumpsL} can be regarded as the remaining lifetime of a subordinator in a given point, for a detailed discussion on  we refer to \cite{BonColDovPag} and the references therein. \\

The symbol $\Phi$ has been already defined and the symbol $\Psi$ introducing the operator $D^\Psi_t$ via the kernel $\psi$ is a Bernstein symbol as well. The class of Bernstein functions contains infinite symbols characterizing infinite subordinators. Both kernels $\phi$ and $\psi$ are obtained as the tail of the L\'{e}vy measure of the associated subordinator.  Recall that a subordinator and its inverse can be respectively regarded as the hitting time and the local time for some Markov process. Also the symbol $\Upsilon$ belongs to the class of Bernstein symbols. In general the time operators in \eqref{CD-derivativesPhiPsi} are associated with inverses to a subordinator whereas, the space operator \eqref{MderivativeUpsilon} is associated with a subordinator. \\

We present the following probabilistic representations.\\

Recall \eqref{genGmudelta}. Let us introduce the space
\begin{align*}
    D_\phi = \{ \varphi : (0, \infty) \times [0, \ell) \to \mathbb{R}\, \textrm{ s.t. for all } t>0, \, (0,t) \ni s \to \dot{\varphi}(s, 0) \phi(t-s) \textrm{ is in } L^1(0,t) \}.
\end{align*}

\begin{theorem}
For the problem \eqref{NLBVPedges}, the solution $u \in D(G^{1,c}_\mu) \cap D_\phi$ is unique and admits the following representation
\begin{align*}
u(t,x) = \mathbf{E}_x[f(X^\mu \circ V^{-1}_t ), \, t < V \circ \zeta^\mu \wedge V \circ \tau^\mu_\ell], \quad x \in [0, \ell), \; t\geq 0,\, f \in C[0, \ell)
\end{align*}
where 
\begin{itemize}
    \item[-] $V^{-1}_t = \inf\{s \,:\, V_s >t\}$ is the inverse of $V_t := t + \mathcal{H} \circ \eta_\varepsilon \, \gamma^0_t(X^\mu)$ with $\mathcal{H}=\mathcal{H}^\Phi$,
    \item[-] $\tau_\ell^\mu = \inf\{t\,:\, X^\mu_t = \ell\, |\, X^\mu_0 = x \in [0, \ell)\}$,
    \item[-] $\zeta^\mu$ is the elastic lifetime of $X^\mu$, 
    \item[-] $X^{\mu}$ has non negative drift $\mu$ and elastic coefficient $c\geq 0$.
\end{itemize} 
\label{thm:NLBVPwithDir}
\end{theorem}
Proof postponed, see Section \ref{proof-thm:NLBVPwithDir}.\\

The path in Figure \ref{fig:BmNLBVPplateaus} is an example of the behavior near the boundary point $\{0\}$ without drift. For a detailed discussion of this problem in case $\mu=0$ we refer to \cite{dovFBVPfcaa} and \cite{dovNLBVPsmooth}. The novelty here is given by the non-trivial case $\mu \neq 0$. We underline that the non-local dynamic boundary condition in \eqref{NLBVPedges} introduces a sticky behaviour at the boundary point $\{0\}$ for which the process $X^\mu \circ V^{-1}_t$ is Markov only in $(0, \ell)$. The set of holding times at $\{0\}$ is countable and the holding times are i.i.d. with law given by $\mathcal{H}$ with exponential time. \\

Recall \eqref{genGmudelta}. Let us introduce the spaces
\begin{align*}
    D_\psi = \{ \varphi : (0, \infty) \times [0, \ell) \to \mathbb{R}\, \textrm{ s.t. for all } t>0,\, (0,t) \ni s \to  \dot{\varphi}(s, 0) \psi(t-s)  \textrm{ is in } L^1(0,t)\}
\end{align*}
and
\begin{align*}
    D_\kappa = \left\{ \varphi  \in C_b(0, \infty)\,:\, \Big| \int_0^\infty \big( \varphi(x) - \varphi(x+z) \big) \kappa(dz) \Big| < \infty,\; \forall\, x \in [0, \infty) \right\}.
\end{align*}
Notice that $D_\kappa$ contains bounded and locally Lipschitz continuous functions as discussed above in Section \ref{app:NLOs} in case of singular kernels.
\begin{theorem}
For the problem \eqref{NLBVPnodes}, the solution $u \in D(G^{1,c}_\mu) \cap D_\kappa \cap D_\psi$ is unique and admits the representation
\begin{align*}
u(t,x) = \mathbf{E}_x[f(X^\bullet \circ V^{-1}_t )], \quad x \in [0, \infty), \; t\geq 0,\, f \in C_b(0, \infty)
\end{align*}
where 
\begin{itemize}
    \item[-] $V^{-1}_t = \inf\{s \,:\, V_s >t\}$ is the inverse of $V_t := t + \mathcal{H} \circ \eta_\nu\, \gamma^0_t(X^\bullet)$ with $\mathcal{H} = \mathcal{H}^\Psi$,
    \item[-] $X^\bullet_t := X^{\mu}_t + A\circ \gamma^0_t(X^{\mu})$ with $A$ depending on $\Upsilon$ according with \eqref{jumpsL},
    \item[-] $X^{\mu}$ has non positive drift $\mu$ and elastic coefficient $c=|\mu|/2$.
\end{itemize} 
\label{thm:NLBVP}
\end{theorem}
Proof postponed, see Section \ref{proof-thm:NLBVP}. \\

We underline that $c=|\mu|/2$ is only due to the sake of simplicity. The parameter $c$ also in this case can be non negative. The non-local boundary condition in \eqref{NLBVPnodes} becomes non homogeneous in case $c>|\mu|/2$. The case $c< |\mu|/2$ leads to interesting probabilistic interpretations (see \cite{SW05, SW06}).\\

The path in Figure \ref{fig:BmNLBVPjumpsANDplateaus} provides and example for the process without drift. The case $\mu=0$ has been treated in \cite{BonColDovPag}. The case $\mu \neq 0$ with slightly different boundary condition has been treated in \cite{ColDovPag} where a clear connection with stochastic resettings emerges. The additive part $A$ introduces a jump as the time $\gamma$ increases. This occurs as the process $X^\mu$ hits the boundary point $\{0\}$. Since $A$ is right-continuous, then the time change $V^{-1}$ acts only after the jump and therefore $X^\mu$ is constant for a random (holding) time at a random level (reached by the jump). For the drifted Brownian motion, the reflection can be realized only by jumps. Jumps can be deterministic or random depending on the kernel $\kappa$. \\

The problems \eqref{NLBVPedges} and \eqref{NLBVPnodes} are respectively associated with edges and vertices. Conversely, the stickiness parameters $\eta_\varepsilon$ and $\eta_\nu$ provide information about processes respectively associated with vertices and edges. We introduce the processes 
\begin{align*}
X^{[\varepsilon]}=\{X^{[\varepsilon]}\}_{t\geq 0} \quad \textrm{and} \quad X^{[\nu]}=\{X^{[\nu]}_t\}_{t\geq 0}
\end{align*}
defined as
\begin{align}
   X^{[\varepsilon]}_t := X^\mu \circ V^{-1}_t, \quad \mu \geq 0 \quad \textrm{given in Theorem \ref{thm:NLBVPwithDir}} 
\end{align}
and
\begin{align}
   X^{[\nu]}_t := X^\bullet \circ V^{-1}_t, \quad \mu \leq 0 \quad \textrm{given in Theorem \ref{thm:NLBVP}} 
\end{align}
together with the hitting times 
\begin{align}
 \tau^{[\varepsilon]}_\ell := \inf\{ t\,:\, X^{[\varepsilon]}_t = \ell\} \quad \textrm{and} \quad \tau^{[\nu]}_J := \inf\{t\,:\, X^{[\nu]}_{t-}  \neq X^{[\nu]}_t \}.
 \label{taunuSecX}
\end{align}
We stress the fact that $\eta_\varepsilon$ and $\eta_\nu$ respectively control the holding times of $X^{[\varepsilon]}$ on vertices and $X^{[\nu]}$ on "edges". This motivates the notation we use here: vertices denoted by $\nu$ and edges denoted by $\varepsilon$ will be introduced in Section \ref{sec:MetricG} below together with metric graphs. Notice that $\tau^{[\varepsilon]}_\ell$ equals in law $\tau^\mu_\ell$ investigated in Corollary \ref{coro:tauLEVEL} and $\tau^{[\nu]}_J$ equals in law $\tau^\mu_0$ investigated in Theorem \ref{thm:hittingTime} for $\mu\leq 0$. We underline the fact that $\tau^{[\nu]}_0$ can be defined as the time at which a jump occurs. The process $X^\bullet$ never hits the boundary point $\{0\}$, it jumps immediately away. Indeed, the jumps of $X^\bullet$ is realized via the additive part $A$ and depends on the jumps of $H^\Upsilon$ from which the right-continuity is preserved. In particular, if $X^{[\nu]}_0=x \in (0, \infty)$, then $X^{[\nu]}$ leaves continuously the starting point $x$ whereas
\begin{align}
X^{[\nu]} \circ \tau^{[\nu]}_J = J \quad \textrm{with} \quad \mathbf{P}(J >0) = 1.
\label{jumpXnode} 
\end{align}
Thus, the process $X^{[\nu]}$ leaves the boundary point $\{0\}$ only by jumps and 
\begin{align}
X^{[\nu]} \circ \tau^\mu_0 = X^{[\nu]} \circ \tau^{[\nu]}_J .
\label{jumpOnlyXnode}
\end{align}
From the behaviour of $X^\bullet$ (we refer to \cite{BonColDovPag}) and the construction of the sticky effect via time change, we obtain some useful characterization for the energy accumulation process.\\ 

The sticky boundary condition introduces holding times on the boundary point $\{0\}$. The set of holding times is Lebesgue measurable and countable. In case $\Phi(\lambda)=\lambda$ (that is $D^\Phi_t u = \dot{u}$) we can identify a sequence of i.i.d. exponential random variables (with parameter $1/\eta_\varepsilon)$ for the time the process $X^{[\varepsilon]}$ spends on $\{0\}$ with each visit (see Theorem \ref{thm:HTrdBM}). For the symbol $\Phi$ of the tempered subordinator we still have a sequence of i.i.d. holding times.
\begin{corollary}
    The holding times at the boundary point $\{0\}$ of the process $X^{[\varepsilon]}$ are independent and identically distributed as $\mathcal{H}^\Phi \circ T_{1/\eta_\varepsilon}$.
    \label{coro:HTverteces}
\end{corollary}
\begin{proof}
An adaptation of the the proof of \cite[Theorem 4]{dovFBVPESfcaa} or \cite[Theorem 11]{BonDovfcaa}. 
\end{proof}

Recall that $X^\bullet$ is right-continuous, then the sticky condition at $\{0\}$ realizes holding times only after the jumps (away from $\{0\}$).

\begin{corollary}
    The holding times after the jumps of $X^{[\nu]}$ are independent and identically distributed as $\mathcal{H}^\Psi \circ T_{1/\eta_\nu}$.
    \label{coro:HTnodes}
\end{corollary}
\begin{proof}
    Also here, an adaptation of the the proof of \cite[Theorem 4]{dovFBVPESfcaa} or \cite[Theorem 11]{BonDovfcaa}. 
\end{proof}
For the process $X^{[\nu]}$ we summarize the following facts:
\begin{itemize}
    \item[i)]  A jump $J$ away from $\{0\}$ is realized according with $\Upsilon$;
    \item[ii)]  After a jump, the process $X^{[\nu]}$ is constant according with the holding times in Corollary \ref{coro:HTnodes}, then it starts afresh. 
\end{itemize}

We underline that
    \begin{align}
        \mathbf{E}_x[\tau^{[\nu]}_0] = \mathbf{E}_x [\tau^{\mu}_0 ] = \frac{x}{|\mu |}
    \end{align}
and 
    \begin{align}
        \mathbf{E}_x [V \circ (\tau^\mu_\ell \wedge \zeta^\mu)] = \mathbf{E}_x [V \circ \tau^\mu_\ell \wedge V \circ \zeta^\mu]
        \label{meanVcomposto1}
    \end{align}
    gives the mean sojourn time on $[0,\ell)$ of $X^{[\varepsilon]}$, $\mathbf{E}_x[\tau^{[\varepsilon]}_\ell]$. For $c=0$ ($\zeta^\mu=\infty$ almost surely), we simply write
\begin{align}
        \mathbf{E}_x[\tau^{[\varepsilon]}_\ell] = \mathbf{E}_x [V \circ \tau^\mu_\ell]= \mathbf{E}_x[\tau^\mu_\ell] + \eta_\varepsilon \Phi^\prime(0) \mathbf{E}_x[\gamma \circ \tau^\mu_\ell]
        \label{meanVcomposto2}
    \end{align}
for the time change $V_t = t + \mathcal{H} \circ \eta_\varepsilon \gamma_t$ where $\mathcal{H}$ is characterized by the symbol $\Phi$. Notice that 
\begin{align*}
\mathbf{E}_0[\mathcal{H} \circ e^\mu_0] = \mathbf{E}[e^\mu_0] \mathbf{E}_0[\mathcal{H}_1] =  \eta_\varepsilon \Phi^\prime(0).
\end{align*}
Our case is concerned with $\Phi$ in \eqref{symbPhi} and
    \begin{align}
        \mathbf{E}_x[\tau^{[\varepsilon]}_\ell] = \mathbf{E}_x[\tau^\mu_\ell] + \eta_\varepsilon \frac{1}{|\mu|}  \mathbf{E}_x[\gamma \circ \tau^\mu_\ell].
        \label{meanHTXvarepsilon}
    \end{align}

We observe that $\mathbf{E}[\mathcal{H}_t] = t \Phi^\prime(0)$ as we can deduce from $\mathbf{E}[e^{-\lambda \mathcal{H}_t}] = e^{-t \Phi(\lambda)}$. In particular
\begin{align*}
\Phi^\prime(0) = \lim_{\lambda \to 0} \frac{\Phi(\lambda)}{\lambda} < \infty
\end{align*}
for the symbol $\Phi$ as defined in the present work (see \eqref{symbPhi}). This is generally not true. Moreover, it provides information also in terms of \eqref{observeFiniteMean}.

\begin{corollary}
The mean exit time of the process $X^{[\varepsilon]}$ writes
    \begin{align}
        \mathbf{E}_x[\tau^{[\varepsilon]}_\ell] = \frac{\ell - x}{\mu}  - \frac{e^{-\mu x} - e^{-\mu \ell}}{\mu^2} \left( 1 - \eta_\varepsilon  \right) , \quad \eta_\varepsilon \geq 0, \quad x \in [0, \ell).
        \label{meanExitTime}
    \end{align}    
\end{corollary}
\begin{proof}
We used \eqref{meanHTXvarepsilon} and Theorem \ref{thm:meanStoppedSWcase} in case $\delta=1$ to handle the dynamic boundary condition in \eqref{NLBVPedges}.    
\end{proof}
 The process $X^{[\varepsilon]}$ has instantaneous reflections only in case $\eta_\varepsilon =0$. In this case, $X^{[\varepsilon]}$ equals in law $X^\mu$. We observe that, under the the second order boundary condition given by $\delta=0$ in Theorem \ref{thm:meanStoppedSWcase}, we get the mean exit time
    \begin{align}
        \frac{\ell - x}{\mu}  - \frac{e^{-\mu x} - e^{-\mu \ell}}{\mu^2} \left( 1 - \frac{\eta_\varepsilon}{\eta_\varepsilon \mu + 1} \right) , \quad \eta_\varepsilon \geq 0, \quad x \in [0, \ell)
    \end{align}
in place of \eqref{meanExitTime}.

\subsection{Non-singular kernels}
\label{sec:UpsilonExpJumps}
    We deal with the case in which $H^\Upsilon$ has finite activity with representation in terms of the compound process   
    \begin{align}
        \sum_{k=1}^{N_t} J_k \quad \textrm{with} \quad J_k \sim J\; \forall\, k
    \end{align}
    where $J$ is an exponential r.v. (with parameter $1/\eta_\varepsilon$) and $N_t$ is a Poisson process (with rate $1$). We have exponential jump intensity measure with
    \begin{align}
        \mathbf{E}_0\left[\exp \left( - \lambda \sum_{k=1}^{N_t} J_k \right) \right] = \exp \left(- t \int_0^\infty (1- e^{-\lambda y}) (1/\eta_\varepsilon) e^{-(1/\eta_\varepsilon)y} dy \right), \quad \lambda > 0.
    \end{align}
    The associated operator turns out to be 
    \begin{align}
        \mathbf{D}^\Upsilon_{x-} u(x) = \int_0^x u^\prime(x-y)\mathbf{P}(J>y)dy = \int_0^\infty (u(x) - u(x-y)) \kappa(dy)
    \end{align}
    where 
    $$\kappa(dy)=(1/\eta_\varepsilon) e^{-(1/\eta_\varepsilon) y} dy$$
    is the jump intensity measure. The additive term $A_t = H^\Upsilon \circ L^\Upsilon_t - t$ introduces (see \cite{BonColDovPag}) the formal adjoint to $\mathbf{D}^\Upsilon_{x-}$ written as
    \begin{align}
        \mathbf{D}^\Upsilon_{x+} = \int_0^\infty (u(x) - u(x+y)) \kappa(dy)
    \end{align}
    which is, in our notation, the operator $\mathbf{D}^\Upsilon_x$ with symbol 
    \begin{align}
        \Upsilon(\lambda) = \frac{\lambda}{\lambda + 1/\eta_\varepsilon}, \quad \lambda>0.
        \label{UpsilonExp}
    \end{align}
    As the local time $\gamma^0_t(X^\mu)$ increases, then $A \circ \gamma^0_t(X^\mu)$ jumps according with $J$.

\section{The processes $E$ and $W$}

We consider $N$ regions ($N$ star vertices). Let $(m_1, \ldots, m_N)$ be the vector of given magnitudes associated with the regions and therefore with the vertices. Let $(v_1, \ldots, v_N)$ be the vector of given velocities characterizing the seismic waves propagation for the regions. A seismic wave is characterized by the energy accumulation $E$ and the wave propagation $W$.

\begin{assumption}
For the region $r \in \{1,2, \ldots, N\}$:
\begin{itemize}
    \item $m_r \in (0,m_1]$ and $m_1 >0$; 
    \item $h_r=h(m_r)$ for a good (decreasing) function $h$;
    \item $v_r \geq 0$, $\forall\, r$.
\end{itemize}
\label{ASS2}
\end{assumption}

For the process $E$ and $W$ we define
\begin{align}
\zeta^E_{abs} = \inf\{ t\,:\, E_{t-} \neq E_t < h_*\} \quad \textrm{and} \quad \zeta^W_{reg} = \inf\{t\,:\, W_t =\ell\}.
\label{lifeABSandREG}
\end{align}
The absorption time $\zeta^E_{abs}$ is written in terms of jumping times and the threshold
\begin{align*}
h_* < \min_r \{h_r \}
\end{align*}
representing the low level of energy for which no events occur. The process $W$ sticks to $\{0\}$ for the time the process $E$ started from $\{h_r\}$ spends to reach $\{0\}$, then an event occurs and $W>0$. If $E$ jumps to a level lower or equal than $h_*$, then as $E$ hits zero no event occurs. Recall that our model deals with negatively drifted Brownian motion for $E$, thus the jump of $E$ gives the energy for a given location to be released in a random time, the hitting time of $\{0\}$. We observe that the time $\zeta^W_{reg}$ will be associated in Section \ref{sec:MetricG} with the exit time from a star graph with edges of length $\ell$. Thus, all the regions will be associated to the same star graph. We consider the same distance between regions. With no effort we can consider the realistic case in which 
\begin{align*}
\zeta^W_{reg} (r) = \inf \{t\,:\, W_t = \ell_r\}
\end{align*} 
for a given region $r$. There are different distances between regions.\\


It is worth underlining the meaning of $i$ and $r$ used below. Recall Assumption \ref{ASS1}.
\begin{notation}
A region is denoted by $r \in \{1,2,\ldots, N\}$. We use the index $i \in \mathbb{N}$ for the $i$-th excursion of the process ($E$ or $W$), that is for the $i$-th seismic wave. In particular:
\begin{itemize}
\item[-] $E$ may have many excursions on $(0, \infty)$ associated with each regions;
\item[-] $W$ may have many excursions on $(0, \ell)$ associated with each regions;
\item[-] $i=r \in \{1,2, \ldots, N\}$ iff one excursion of a process can be associated with each regions. 
\end{itemize} 
\end{notation}

In the next section we introduce $E$ and $W$ in detail. For example, the first seismic wave is characterized by the couple $(E_t, W_t)$ up to time $\tau_1$. The reader can have a clear picture of our construction by taking a glance at Figure \ref{fig:relations} at the end of the work.

\subsection{The energy accumulation process}
\label{sec:EnergyAccumulation}
Let $E = \{E_t\}_{t\geq 0}$ be the process such that:
\begin{itemize}
\item[i)] $E_0=h_0 > 0$; 
\item[ii)] $E$ can be associated with the negatively drifted process $X^{\mu}$ on $(0, \infty)$ as follows
\begin{align*}
E_t = \left\lbrace
\begin{array}{lrll}
X^{\mu_1}_t, & 0 < &t& < \tau_1^E,\\
h_1, & \tau_1^E \leq &t& <  \tau_1,\\
\ldots \\
X^{\mu_i}_t, & \tau_{i-1} < &t& < \tau_{i-1} + \tau^E_i,\\
h_i, & \tau_{i-1} + \tau^E_i \leq &t& < \tau_i,\\
\ldots
\end{array}
\right. , \quad i \in \{1,2, \ldots, \mathfrak{N}\}
\end{align*}
where
\begin{align}
    \tau_i := \sum_{j \leq i} ( \tau^E_j + \tau^W_j ), \quad i \in \mathbb{N}, \quad \tau_0=0,
    \label{sumT}
\end{align}
and $|\mu_i| \in (m_1, \ldots, m_N)$ with $N \leq \mathfrak{N} \in \mathbb{N}$;
\item[iii)] If $\mathfrak{N}$ equals the number $N$ of regions, then $E$ is a bijection between the set of regions and the family of drifted Brownian motions. Then, for the realization $i$ we have $\mu_i=\mu_r$ for the region $r \in \{1, \ldots, N\}$ ;
\item[iv)] $E$ is a switching process. Indeed, the process changes its drift according with the vector of magnitudes. The switching rule depends on $W$.
\end{itemize}

\begin{remark}
To have a clear picture, we underline that, in the region $r$, the $i$-th seismic wave is given by the couple $(E,W)$ where $E$ is a drifted Brownian motion with $\mu_i = - m_r$.  
\end{remark}

Under the previous characterization we say that
\begin{center}
$E$ describes the {\it earthquake energy accumulation}.
\end{center}

For the sequence $\tau^E = \{\tau^E_i\}_{i}$ we say that
\begin{center}
    $\tau^E$ describes the {\it accumulation times (release of the accumulated energy)} .
\end{center}
That is, the time needed to reach an appropriate level of energy for which in a given region (or edge) an event (or propagation) occurs. Otherwise, the process $E$ never reaches (jumps) a level greater that $h_*$ and then, the release of energy is not enough to create the seismic wave. In this case, the process $E$ still runs as time passes with lower and lower jumps below the threshold $h_*$ whereas (as described below) in the meanwhile the process $W$ sticks to $\{0\}$.

\subsection{The seismic waves propagation process}
\label{sec:SeismicWavePropagation}
Let $W = \{W_t\}_{t\geq 0}$ be the process on $[0, \ell)$ such that:
\begin{itemize}
\item[i)] $W_0 = 0$;  
\item[ii)] $W$ can be associated with the positively drifted process $X^{\mu}$ on $[0, \infty)$ as follows
\begin{align*}
W_t = \left\lbrace
\begin{array}{lrll}
0, & 0 \leq & t & < \tau^E_1  \wedge \zeta^E_{abs},\\
X^{\mu_1}_t, &  \tau^E_1 \leq & t & < \tau_1 \wedge \zeta^W_{reg} \wedge \zeta^W_{kil}, \\
\ldots \\
0, & \tau_{i-1} \leq & t & < ( \tau_{i-1} + \tau^E_{i} ) \wedge \zeta^E_{abs},\\
X^{\mu_i}_t, & \tau_{i-1} + \tau^E_i \leq & t & < \tau_i \wedge \zeta^W_{reg} \wedge \zeta^W_{kil},\\
\ldots
\end{array}
\right. \quad i \in \{ 1,2, \ldots, \mathfrak{N}\}
\end{align*}
where $\tau_i$ has been defined in \eqref{sumT} and $\mu_i \in (v_1, \ldots, v_N)$ with $N \leq \mathfrak{N} \in \mathbb{N}$. Moreover: 
\begin{itemize}
    \item[-] $\zeta^E_{abs}$ is an absorption time such that
\begin{align}
	W_t = 0, \quad t \geq \zeta^E_{abs},
\end{align}
    \item[-] $\zeta^W_{reg}$ is a switching time (between regions) such that
\begin{align}
      X^{v_{r}} \circ (\zeta^W_{reg} + t) = X^{v_{r+1}}_t
      \label{ruleZetaReg}
\end{align}
where $(v_1, \ldots, v_N)$ is the vector of velocities characterizing the propagation;
\item[-] $\zeta^W_{kil}$ is a killing time do not associate with $\zeta^E_{abs}$. However, we still say that
\begin{align}
W_t = 0, \quad t \geq \zeta^W_{kill}
\end{align}
meaning no wave after the kill;
\end{itemize}
\item[iii)] If $\mathfrak{N}$ equals the number $N$ of regions, then $W$ visits a region never to return; 
\item[iv)] $W$ is a switching process on $[0,\ell)$ and the sequence $\{\tau^E_i + \zeta^W_{reg}\}_{i}$ is a sequence of switching times for $W$ only if $\mathfrak{N}=N$. Indeed, the process changes its drift according with the vector of velocities.
\end{itemize}

The process $W$ starts from $\{0\}$ after a random time, then it runs until the first hitting time with the level $\ell$. Once it reaches $\ell$, after a resetting at $\{0\}$, it stars as a new process with different drift under the same rule for the holding time. The holding times for $W$ are given by the sequence $\tau^E$.

\begin{remark}
We underline that, in the region $r$, the $i$-th seismic wave is given by the couple $(E,W)$ where $W$ is a drifted Brownian motion with $\mu_i = v_r$.  
\end{remark}

Under the previous characterization we say that
\begin{center}
$W$  describes the {\it seismic waves propagation}.
\end{center}

For the sequence $\tau^W = \{\tau^W_i\}_i$ we say that
\begin{center}
    $\tau^W$ describes the {\it propagation times in a region},
\end{center}
that is, the time the seismic waves spend to propagate in some directions (or edges) of the same region (star graph) whereas 
\begin{center}
    $\zeta^W_{reg}$ describes the {\it propagation time between regions}  
\end{center}
that is, the time the seismic waves spend to arrive at the next region (next vertex). 


\begin{remark}
Recall that our aim is to consider $X^{[\varepsilon]}$ to describe $W$ and $X^{[\nu]}$ to describe $E$. The previous assumption can be associated with the stickiness parameter $\eta_\varepsilon$. If $\eta_\varepsilon = 0$, then the process $X^{[\varepsilon]}$ reflects instantaneously (no positive holding times) and there are no excursions for $X^{[\nu]}$ describing the energy accumulation $E$. 
\end{remark}

\begin{remark}
We observe that $\{v_i\}_i$ is a sequence of given constants. In the model this sequence says that the earthquake in a region will propagate according with a Brownian motion with a given velocity. A given velocity characterizes a given region and depends from ground and magnitude among other variables for that region.  
\end{remark}
        
\begin{remark}        
In place of the previous characterization we may also consider $v$ as a function of $E$, thus the process $W$, region by region, depends also from the random energy accumulated in the last visited region. 
\end{remark}

\section{The processes $E$ and $W$ via PDEs}

Recall \eqref{EtoRDBM} and \eqref{WtoRDBM} and focus on \eqref{sumT}. We exploit here the fact that, in case $\tau^W_i = 0$ $\forall\, i$, then $\{E_t , t < \tau^E\}$ can be identified with the excursion of $X^{[\nu]}$ as well as, in case $\tau^E_i=0$ $\forall\, i$, then $\{W_t, t<\tau^W\}$ can be identified with the excursion of $X^{[\varepsilon]}$. This is done in a weak sense, that is we have equivalence in law. From the Brownian motion $X^\mu$ we characterize the energy accumulation process $E$ and the wave propagation process $W$ in terms of the processes $X^{[\varepsilon]}$ and $X^{[\nu]}$ by considering the entire paths on $[0, \infty)$. We first discuss on the random times $\tau^E =\{\tau^E_i\}_i$ and $\tau^W=\{\tau^W_i\}_i$ in terms of the occupation time $\Gamma$ and the local time $\gamma$ of $X^{[\varepsilon]}$ on $[0, \ell)$. Let us define the set 
\begin{align*}
     J^\varepsilon := \{t\,:\, \Gamma^{-1}_{t-} \neq \Gamma^{-1}_t\}, \quad \Gamma_t = \Gamma^{[\varepsilon]}_t := \int_0^{t \wedge \tau^{[\varepsilon]}_\ell} \mathbf{1}_{(0, \ell)}(X^{[\varepsilon]}_s)ds 
\end{align*}
for which
\begin{align}
    \tau^E \subset \{ \Gamma^{-1}_t - \Gamma^{-1}_{t-} \,:\, t \in J^\varepsilon\} =:  J^\varepsilon_\Delta 
\label{Jedge}    
\end{align}
and the set
\begin{align*}
    J^\nu := \{t\,:\, \gamma^{-1}_{t-} \neq \gamma^{-1}_t\}, \quad \gamma_t = \gamma^{[\varepsilon]}_t := \int_0^{t \wedge \tau^{[\varepsilon]}_\ell} \mathbf{1}_{\{0\}}(X^{[\varepsilon]}_s)ds
\end{align*}
for which 
\begin{align}
    \tau^W \subset \{ \gamma^{-1}_t - \gamma^{-1}_{t-}\,:\, t \in J^\nu\}=: J^\nu_\Delta.
    \label{Jnodes}
\end{align}
Without abuse of notation we write the local time $\gamma_t$ as the integral above. Recall that $\gamma_t$ is defined as the continuous additive function increasing only as $X^{[\varepsilon]}$ hits $\{0\}$ up to time $t$. That integral is not trivially zero because the sticky condition for $X^{[\varepsilon]}$ says that $\{t\,:\, X^{[\varepsilon]}_t = 0\}$ is a set of positive Lebesgue measure. Observe that $J^\varepsilon_\Delta$ is a countable set of holding times at $\{0\}$ for $X^{[\varepsilon]}$ whereas, $J^\nu_\Delta$ is a countable set of holding times at the point $J \in (0, \infty)$ for $X^{[\nu]}$. Indeed, $X^{[\nu]}$ never hit the boundary point $\{0\}$ and immediately jumps at the random level $J$ according with the additive functional $A_t$ given in \eqref{jumpsL}, see also Figure \ref{fig:BmNLBVPjumpsANDplateaus}. Formulas \eqref{Jedge} and \eqref{Jnodes} must be meant in the sense that $\{\tau^E_i\}_i$ are i.i.d. random variables sharing the law of the holding times of $X^{[\varepsilon]}$ whereas, $\{\tau^W_i\}_i$ are i.i.d. random variables sharing the law of the holding times of $X^{[\nu]}$. For the integral
\begin{align*}
    \int_0^{\tau^{[\varepsilon]}_\ell}  \mathbf{1}_{\{0\}}(X^{[\varepsilon]}_t) \, dt = \gamma \circ \tau^{[\varepsilon]}_\ell = \int_{J^\varepsilon_\Delta \cap [0, \tau^{[\varepsilon]}_\ell]} d\gamma_t 
\end{align*}
we recall the result obtained in Theorem \ref{thm:meanStoppedSWcase} and the fact that the zero set of a sticky Brownian motion has positive Lebesgue measure.

\subsection{The energy accumulation process}
\label{sec:EII}

\begin{assumption}
For every $i \in \mathbb{N}$, the levels $h_i$ are independent exponential random variables such that $\mathbf{E}[h_i] = \eta_\varepsilon$.
For the region $r\in \{1,\ldots, N\}$ we have $\eta_\varepsilon = m_r/\sigma_r$ and write $h_r$ for the i.i.d random variables in that region.
\label{ASS3}
\end{assumption}

The previous assumption well agrees with Section \ref{sec:UpsilonExpJumps} where $\eta_\varepsilon = m_r/\sigma_r$, that is the stickiness coefficient for the region $\nu=\nu_r$ is given by the magnitude $m_r$ and the parameter $\sigma_r$. The random level $h_r$ plays the role of the random variable $J$ associated with the symbol $\Upsilon$ as described above. Moreover, we observe that $\mathbf{P}(h_r > x)$ decreases as $m_r/\sigma_r$ decreases. This is due to the realistic description of decreasing effects after the mainshock. 

\begin{theorem}
    Fix $i \in \mathbb{N}$ and $r \in \{1, \ldots, N\}$. Let $\tau^E_i$ be the hitting time of the excursion $i$ of $E$ associated with the region $r$. Under Assumption \ref{ASS3}, we have that $\tau^E_i \stackrel{d}{=} \mathcal{H} \circ h_r$. In particular, 
    \begin{align*}
\mathbf{P}_0(\tau^E_i \in dz) = \int_0^\infty \frac{h}{z} \frac{e^{-(h- m_r z)^2 /(4z)}}{\sqrt{4\pi z}} \frac{\sigma_r}{m_r} e^{-h(\sigma_r/m_r)} \,dh \,dz, \quad z>0.
\end{align*}
\label{thm:tauEwithH}
\end{theorem}
\begin{proof}
This follows directly from Corollary \ref{coro:HTverteces}. Indeed, $\tau^E$ is the running time for the process $E$ corresponding with the holding time of the process $W$. Our construction consider the description of the process $W$ to be given in terms of $X^{[\varepsilon]}$ which is a sticky Brownian motion with holding time at $\{0\}$ given by $\mathcal{H} \circ h_r$ where $h_r$ is an exponential r.v. with parameter $1/\eta_\varepsilon$ where $\eta_\varepsilon = m_r/\sigma_r$ being associated with a region $r \in \{1, \ldots, N\}$.
\end{proof}

Recall that the tempered subordinator $\{\mathcal{H}_t\}_{t\geq 0}$ has Bernstein symbol
\begin{align}
\Phi_r(\lambda) = \sqrt{\lambda + \theta_r} - \sqrt{\theta_r} \quad \textrm{with} \quad \theta_r = (m_r/2)^2, \quad \lambda>0
\label{SymbPhiE}
\end{align}
from which we write
\begin{align*}
\mathbf{E} [\mathbf{E}_0[e^{-\lambda \mathcal{H}_{h_r}} | h_r]] = \mathbf{E}[e^{- h_r \Phi_r(\lambda)}] = \frac{1/\eta_\varepsilon}{1/\eta_\varepsilon + \Phi_r(\lambda)}.
\end{align*}

\begin{theorem}
Let Assumption \ref{ASS3} holds true. For the energy accumulation we have
\begin{align}
\mathbf{E}_{h_r} [f(E_t)] = \mathbf{E}_{h_r}[f(X^{[\nu]}_t)],\quad t\geq 0, \quad f \in C_b([0, \infty)).
\label{formula1-thm:equivE}
\end{align}
for a suitable choice of $\Upsilon$ and $\Psi$.\\ 
In particular, for the $i$-th seismic wave in the region $r \in \{1, \ldots, N\}$,
\begin{align}
\mathbf{E}_x[f(E_t), \tau_{i-1} \leq t < \tau_{i}] = \mathbf{E}_x[f(X^{\mu_i} \circ V^{-1}_t + A \circ \gamma \circ V^{-1}_t),\, 0 \leq t < \mathcal{H}^\Phi_x + \mathcal{H}^\Psi_{e^{\mu_i}_0}] 
\label{formula2-thm:equivE}
\end{align}
is the solution to the problem \eqref{NLBVPnodes} with $\mu= \mu_i$ and $|\mu_i| \in (m_1, \ldots, m_N)$, $\Upsilon$ given in \eqref{UpsilonExp} and $\Psi$ the symbol of the independent subordinator $\mathcal{H}^\Psi$.
\label{thm:equivE}
\end{theorem}
Prof postponed, see Section \ref{proof-thm:equivE}.\\

We now introduce some connection with the star graphs hopping this make some advantage for the readers. The composition $\mathcal{H} \circ h_r$ will provide the running time for $E$ and the holding time for $W$ in the star graph $\mathsf{S}_{\nu_r}$ associated with the region $r$. Then, under Assumption \ref{ASS3}, for a region $r \in \{1,\ldots, N \}$ with associated level $h_r$, the time $\tau^E_i$ for the excursion $i \in \mathbb{N}$ of $E$ equals in law $\mathcal{H} \circ h_r$. In particular, $h_r$ is an exponential r.v. with parameter $1/\eta_\varepsilon$ independent from the excursion $i \in \mathbb{N}$ of $E$. Notice that we refer to the excursion of $E$ as the paths on $(0, \infty)$ of $E$. This is actually associated with an excursion if we recall that $X^{[\nu]}$ describing $E$ leaves the boundary point $\{0\}$ only by jumps. Our model is consistent with the following setting. 

For the region characterized by $\mathsf{S}=\mathsf{S}_\nu$:
\begin{align*}
    & E \textrm{ is the energy at the vertex $\nu$ of the star graph},\\
    & E \textrm{ can be described by $X^{[\nu]}$ driven by \eqref{NLBVPnodes} with $\mu=-m_r$ (negative magnitude)},\\
    & \eta_\varepsilon D^\Phi_t \textrm{ in \eqref{NLBVPedges} describes the running time of $E$ (time to release the initial level $h$ of energy)},  \\
    & \eta_\varepsilon = (m_r/\sigma_r) \textrm{ iff } \nu=\nu_r \textrm{ is a vertex associated with the region } r,\\
    & \eta_\nu D^\Psi_t \textrm{ in \eqref{NLBVPnodes} describes the holding time of $E$ after a jump (a level $h$ of energy)},\\
    & \eta_\nu = \eta_\nu(v_r) \textrm{ depends on $v_r$ iff $\nu=\nu_r$ is a vertex associated with the region r}.
\end{align*}

Summarizing, the process $X^{[\nu]}$ on $[0, \infty)$ will describe the energy accumulation $E$ on the vertex $\nu=\nu_r$ of a star graph in terms of the excursion $i$ ($i$ visits of $\{0\}$ or $i$ jumps away from $\{0\}$ will be associated with $i$ visits of the set of edges).  The time $\tau^{[\nu]}_{h_i} \stackrel{d}{=} \tau^E_i$ according with the right selection of $h_i$ for the region $r$. We underline that, under the previous setting and assumptions, for the region $r$,
\begin{align}
    \mathbf{E}_0[\mathcal{H} \circ h_i] = \frac{1}{\sigma_r}
    \label{meanAccumTime}
\end{align}
is the mean accumulation time in the region $r \in \{1,\ldots, N\}$ for every excursion $i \in \mathbb{N}$. In \eqref{Epdeh0} we wrote $h_0=x$ for the starting level of $E$, we underline that $h_i \sim h_r$ for $i \in \mathbb{N} \cup \{0\}$, that is we always have an exponential r.v. depending on the region $r$. The parameter $\sigma_r$ can be considered in order to have an extra characterization of the associated region $r$. For the process $E$ the parameter $\sigma_r$ provides information about the $r$-th epicenter.

\subsection{The seismic waves propagation process}
\label{sec:WIII}
It is well known that 
\begin{align}
    \mathbf{E}_x[f(X^\mu_t), t< \tau_\ell^\mu], \quad t>0,\; x \in [0, \ell), \qquad f \in C[0,\ell)
\end{align}
gives the Dirichlet-Neumann semigroup on $C[0,\ell)$ for a reflected drifted Brownian motion on $[0,\ell)$ killed at $\{\ell\}$. The characterization of $W$ is given in terms of $X^{[\varepsilon]}$ and the problem \eqref{NLBVPedges}. Thus, 
\begin{align}
    \mathbf{E}_x[f(X^\mu \circ V^{-1}_t), t < V \circ \tau^\mu_\ell], \quad t>0,\; x \in [0, \ell), \quad f \in C[0, \ell)
\end{align}
is the main object to focus on. Recall that, as $X^\mu$ is not in $\{0\}$, then $X^{\mu} \circ V^{-1}_t$ behaves like $X^\mu$. The first hitting time $V \circ \tau^\mu_\ell$ will be associated with the time the process spends in a given region (in a given star graph). The sequence $\{\tau^W_i\}_i$ are the occupation time of $(0, \ell)$ for the excursion $i$ and represent a very hard object to deal with. The standard Brownian motion can hit the boundary point $\{0\}$ infinitely many times. For positively drifted Brownian motion we can immediately check from Theorem \ref{thm:hittingTime} and formula \eqref{zeroHittingPosDrift} that the mean (first) hitting time of zero can be infinite. Despite of this, we are able to provide a simple representation for the process $W$. Indeed, the sequence $\{\tau^E_i \}_i$ can be regarded as a sequence of holding times for $W$. That is, $\forall\, i$,
\begin{align}
    \mathbf{P}_0(\tau^E_i > t \,| \, W \circ \tau^E_i >0) = \mathbf{P}(\textrm{$W$ started at zero is forced to stay at zero for a time} > t  ).
\end{align}
It is well known that the set of holding times for the sticky Brownian motion has positive Lebesgue measure. It is a countable set. In particular, there exists a sequence of holding times for the process $X^{[\varepsilon]}$ to be associated with $\tau^E$ as discussed above in Corollary \ref{coro:HTverteces}.

\begin{theorem}
Let Assumption \ref{ASS3} holds true. For the wave propagation we have
\begin{align}
\mathbf{E}_0[f(W_t), \, t < \zeta^E_{abs} \wedge \zeta^W_{kill}] = \mathbf{E}_0[f(X^{[\varepsilon]}_t), \, t < \zeta^{[\nu]}_{abs} \wedge \zeta^{[\varepsilon]}], \quad t\geq 0, \quad f \in C([0, \ell))
\label{formula1-thm:equivW}
\end{align}
where
\begin{align*}
\zeta^{[\nu]}_{abs} = \inf\{t\,:\, X^{[\nu]}_{t-} \neq X^{[\nu]}_t < h_*\}.
\end{align*}
In particular, for the $i$-th seismic wave in the region $r \in \{1, \ldots, N\}$, 
\begin{align}
\mathbf{E}_x[f(W_t),\, t < \zeta^W_{kil} \wedge \zeta^W_{reg}] = \mathbf{E}_x[f(X^{\mu_i} \circ V^{-1}_t),\, t < V \circ \zeta^{\mu_i} \wedge V \circ \tau^{\mu_i}_\ell]
\label{formula2-thm:equivW}
\end{align}
is the solution to the problem \eqref{NLBVPedges} with $\mu_i \in (v_1, \ldots, v_N)$ and $\Phi = \Phi_r$ given in \eqref{SymbPhiE}.
\label{thm:equivW}
\end{theorem}
Proof postponed, see Section \ref{proof-thm:equivW}\\

We completely neglect the excursions of $X^{[\varepsilon]}$ and we only focus on the holding times. The symbol $\Psi$ in \eqref{NLBVPnodes} together with the parameter $\eta_\nu$ are devoted to the characterization of the holding times for $X^{[\nu]}$ and therefore, of the excursion times for $X^{[\varepsilon]}$. Since there are infinite many subordinators for the symbol $\Psi$, we reasonably assume that $\exists (\eta_\nu, \Psi)$ to be considered in our construction. However, we proceed with a simplified version of the model which excludes the characterization of the couple $(\eta_\nu, \Psi)$ maintaining an accurate description of the earthquake evolution. Our model is consistent with the following setting. 

For the region characterized by $\mathsf{S}=\mathsf{S}_\nu$:
\begin{align*}
    & W \textrm{ is the seismic wave on the edge $\varepsilon$ of the star graph},\\
    & W \textrm{ can be described by $X^{[\varepsilon]}$ driven by \eqref{NLBVPedges} with $\mu=v_r$ (propagation velocity)},\\
    & M \circ V^{-1}_t \textrm{ gives the stopping time according with the time at $\{0\}$ (no propagation)},\\
    & V\circ \tau^{\mu_i}_\ell \textrm{ gives the stopping time according with the time at $\{\ell\}$ (propagation distance) },\\
    & \eta_\nu D^\Psi_t \textrm{ in \eqref{NLBVPnodes} describes the running times of $W$ up to the exit from $(0, \ell)$},\\
    & \eta_\nu = \eta_\nu(v_r) \textrm{ as above iff $\mathsf{S}=\mathsf{S}_{\nu_r}$ is the star graph associated with the region $r$},\\
    & \eta_\varepsilon D^\Phi_t \textrm{ in \eqref{NLBVPedges} describes the holding times of $W$ at $\{0\}$},\\
    & \eta_\varepsilon = (m_r/\sigma_r) \textrm{ iff $\mathsf{S}=\mathsf{S}_{\nu_r}$ is the star graph associated with the region $r$}.
\end{align*}
We observe that $\eta_\nu D^\Psi_t$ in \eqref{NLBVPnodes} gives the excursion times of $X^{[\varepsilon]}$, that is the return times at $\{0\}$ of a Brownian motion on $[0, \infty)$ up to the first hitting time with the level $\ell>0$. The identification of this subordinator is not significant for our purposes, we only underline that there exists a subordinator characterizing this random times.

Summarizing, the process $X^{[\varepsilon]}$ on $[0, \ell)$ will describe the seismic wave propagation on the region associated with the star graph $\mathsf{S}_{\nu_r}$ in terms of the excursion $i$ ($i$ visits of $\{0\}$ will be associated with $i$ visits of the set of nodes). The time $\tau^{[\varepsilon]}_\ell$ will give the time the wave propagates in the same region, that is on the edges of $\mathsf{S}_{\nu_r}$ until the first hitting time with an external vertex. We underline that, under the previous setting and assumptions,
\begin{align}
    \mathbf{E}_0[\tau^{[\varepsilon]}_\ell] =  \frac{\ell}{v_r}  - \frac{1 - e^{- \ell v_r}}{(v_r)^2} \left( 1 - \frac{m_r}{\sigma_r} \right)
\end{align}
is the mean propagation time on the region $r \in \{1,\ldots, N\}$ which takes the form
\begin{align}
    \frac{\ell}{v_r}  - \frac{1 - e^{- \ell v_r}}{(v_r)^2} \left( 1 - \frac{\frac{m_r}{\sigma_r}}{1 + \frac{m_r v_r}{\sigma_r}} \right)
\end{align}
according with $\delta=0$ in Theorem \ref{thm:meanStoppedSWcase}. For the process $W$ the parameter $\sigma_r$ provides information about the $r$-th region far from the epicenter.

\subsection{The processes $E$ and $W$ under scale transformation}

We observe that the energy accumulation process can be also obtained via some transformation of $X^{[\nu]}$, for example, for $h,k>0$ (see Figure \ref{fig:g-transf}),
\begin{align}
    E = g(X^{[\nu]}) = h\, \exp(- k X^{[\nu]}).
    \label{g-transf}
\end{align}
An example is given in Figure \ref{fig:g-transf}. Analogously, we may consider a transformation
\begin{align}
W = g(X^{[\varepsilon]})    
\end{align}
for some $g$ suitable with our data. The analysis of the associated functionals (hitting times, etc.) can be therefore obtained by considering different scales for the processes $X^{[\nu]}$ and $X^{[\varepsilon]}$. Notice that, under suitable transformations, we can also consider drifted Brownian motion on $\mathbb{R}$ in place of $X^\mu$ on $[0, \infty)$.

\section{Earthquake modeling via Brownian motions on metric graphs}
\label{sec:MetricG}
We now introduce the metric graph to be associated with a given geographical area, then we present our model based on the motions described in Section \ref{sec:EII} and Section \ref{sec:WIII}.

\subsection{Brownian motions on graphs}
\label{sec:BMonG}
We begin with the following
\begin{assumption}
The geographical area is characterized by regions. Regions can be represented by star graphs with the same distance $\ell>0$ between nodes and the star vertex.
\label{ASS5}
\end{assumption}
Recall that we are considering the level $\ell>0$ to be associated with the length of a given edge. With no efforts we are able to include different lengths, eventually random.

Under the Assumption \ref{ASS5} we will consider the motion $X^{[\varepsilon]}$ on $[0, \ell)$. This process is equivalent to a Brownian motion on the edge $\varepsilon$ according with the following construction.\\

Let $\mathcal{E} := \{ \varepsilon=[0,\ell)  \}$ with cardinality $|\mathcal{E}| = N$ be the collection of rays given by the bounded intervals of the real line. Thus, $\mathsf{S} \equiv S_N$ introduced above in Section \ref{sec:occurrence}. We define the star graph $\mathsf{S}$ as the quotient space $\mathsf{S} = \mathcal{E}/\sim$, i.e., we identify the starting points on all edges and in $\mathcal{E}$ the origin $0 \equiv (\cdot, 0)$ is the unique point that belongs to all the rays. Such a point is identified as the vertex $\nu$ of the the star graph $\mathsf{S}$ written also as $\mathsf{S}_\nu$. We identify a point $\mathsf{x} \in \mathsf{S}$ as $\mathsf{x}=(\varepsilon, x)$ for the edge $\varepsilon$ and the distance $x$ from the star vertex. On every edge, we have an Euclidean structure given by the Euclidean distance, and a measure structure induced by the Lebesgue measure. These structures are inherited by the space $\mathcal{E}$. Thus,  we have a metric space with the distance
\begin{align*}
d(\mathsf{x}, \mathsf{y}) = d((\varepsilon_j,x), (\varepsilon_k,y)) = |x - y| \mathbf{1}_{\varepsilon_j= \varepsilon_k} + (x + y) \mathbf{1}_{\varepsilon_j \not= \varepsilon_k}, \quad \varepsilon_j, \varepsilon_k \in \mathcal{E}
\end{align*}
and a measure space with respect to the direct sum measure induced by the Lebesgue 
measure on every edge. We focus on the motion on the star graph $\mathsf{S}_\nu$ with vertex $\nu$. Our model can be adapted to any structure (and therefore to any geographical area) by considering a collection of star graphs $\{\mathsf{S}_\nu, \, \nu \in \mathcal{V}\}$ and the motion on the graph
\begin{align*}
\mathsf{G} = \bigcup_{\nu \in \mathcal{V}} \mathsf{S}_\nu.
\end{align*} 
Thus, we provide a rigorous mathematical formulation only for the motion on $\mathsf{S}$, that is $\mathsf{S}_\nu$ associated to a given region with a given epicenter associated with the vertex $\nu$. 

An earthquake on a given geographical area is characterized by the couple $(E,W)$. If we focus on a given region and assume that $\mathsf{S}_\nu$ represents that region, then the earthquake is described by  $(X^{[\nu]}, \Theta^{[\nu]}, X^{[\Theta]})$ where $\Theta^{[\nu]} =\{\Theta^{[\nu]}_t\}_{t\geq 0}$ is given by $\Theta^{[\nu]}_t = U^{[\nu]} \circ V^{-1}_t$ and 
\begin{align}
    U^{[\nu]}_t = 
    \left\lbrace
    \begin{array}{ll}
          U, & \textrm{if } V_{t-} \neq V_{t} \\
          U^{[\nu]}_{t-}, & \textrm{otherwise} 
    \end{array} 
    \right . , \quad t\geq 0
\end{align}
with
\begin{align}
    \mathbf{P}(U = \varepsilon) = \rho_\varepsilon, \quad \varepsilon \in \mathcal{E}, \quad \textrm{and} \quad \rho_1 + \ldots +\rho_{|\mathcal{E}|} = 1.
\end{align} 
The step process $\Theta^{[\nu]}_t$ changes its value as $X^{[\varepsilon]}$ hits $\{0\}$. We observe that $U^{[\nu]}_t = U_t$ does not depend on $\nu$ in the present formulation. Then, we streamline the notation and write also $\Theta$ in place of $\Theta^{[\nu]}$.

According to \cite{BonDovfcaa}, the process $\mathsf{Q}$ on $\mathsf{S}$ can be defined under the equivalence (up to absorption) of the radial part with a diffusion on the half line. The process $\Theta$ plays the role of an edge selector. Once the process $\mathsf{Q}$ arrives at the star vertex, the angular part $\Theta$ gives the next edge to be visited. Then, the process $\mathsf{Q}$ moves on the selected edge according to the radial part $X^{[\varepsilon]}$ only after the process $X^{[\nu]}$ has reached the desired level of energy.  The motion on $\mathsf{G}$ inherits such behaviour, and the process stops for a random time with each visit of a vertex connecting star graphs, then it starts afresh on a selected incident edge according with the setting of the new star graph ($m_r$, $\sigma_r$, $v_r$ of the associated new region). The probability $\rho_\varepsilon$ provides further characterization of the propagation of the earthquake. However, in order to simplify our model, we proceed with the following.

\begin{assumption}
The seismic wave propagation in a given direction is completely characterized by the propagation velocity for that region.
\label{ASS6}
\end{assumption}

For simplicity, also due to the fact that $U^{[\nu]}_t = U_t$,  we set 
$$\rho_\varepsilon = 1/|\mathcal{E}|$$ 
and $U$ is uniformly distributed over $\mathcal{E}$. Recall that $\mathcal{E}=\mathcal{E}_\nu$ is the set of edges incident the star vertex $\nu$ of the star graph $\mathsf{S}_\nu$. Moreover, for simplicity, we set $|\mathcal{E}_\nu| =N$ for every $\nu \in \mathcal{V}$.

\begin{notation}
Let $Q$ be the earthquake on a given geographical area and $(E,W)$ be the associated energy and wave. Let $\mathsf{G}$ be the network characterizing that area. The process $\mathsf{Q}=\{\mathsf{Q}_t\}_{t\geq 0}$ on $\mathsf{G}$ describes the earthquake $Q$. 
\end{notation}

We observe that $\mathsf{Q}$ on $\mathsf{G}$ can be studied under equivalence with
$$(X^{[\nu]}, \Theta, X^{[\Theta]}), \, \nu \in \mathcal{V}.$$
However, we stress the fact that $X^{[\nu]}$ provides the stopping time (see $\zeta^W_{abs}$) for the process $X^{[\Theta]}$ and the seismic wave on $\mathsf{G}$ can be therefore described in terms of
\begin{align*}
(\Theta_t, X^{[\Theta_t]}_t), \, t < \zeta_{abs}(\mathsf{G})
\end{align*} 
where 
\begin{align}
\zeta_{abs}(\mathsf{G}) = \inf\{t\,:\, X^{[\nu]}_{t-} \neq X^{[\nu]}_t < h_*\}=: \zeta^{[\nu]}_{abs}
\label{lifeGabs}
\end{align}
equals in law $\zeta^E_{abs}$. Further on we write $X_t^{[\Theta]}$ in place of $X_t^{[\Theta_t]}$.\\

We now focus on the NLBVP on $\mathsf{S}$ for the process $\mathsf{Q}$. Let us recall that
\begin{align}
 u(t,\mathsf{x}) = u_\varepsilon(t,x), \quad t>0,\, x \in [0,\ell)\; \textrm{for}\;  (\varepsilon, x) = \mathsf{x} \in \mathsf{S}  
\end{align}
can be written in terms of the projection $u_\varepsilon$ of $u$ along the edge $\varepsilon \in \mathcal{E}$. Thus, for $t>0$, $\mathsf{x} \in \mathsf{S}$ with $\mathsf{x}=(\varepsilon,x)$ and $\varepsilon \in \mathcal{E}$, $x \in [0,\ell)$, we are able to define
\begin{align*}
u^\prime(t, \mathsf{x}) := u^\prime_\varepsilon (t, x) = \frac{d}{dx} u_\varepsilon(t, x), 
\end{align*}
and
\begin{align*}
    u^{\prime \prime}(t, \mathsf{x}) := u^{\prime \prime}_\varepsilon(t,x) = \frac{d^2}{d x^2} u_\varepsilon(t,\mathsf{x}).
\end{align*}
Accordingly, for the Brownian motion on $\mathsf{S}$ with drift $\mu$, we introduce the operator $\mathsf{G}_\mu u= \mu u^\prime + u^{\prime \prime}$ on the space of continuous functions that are twice continuously differentiable on each open ray $\varepsilon$ and such that, for the vertex $\mathsf{v}$ of $\mathsf{S}$,
\begin{align}
    \mathsf{G}_\mu u(t,\mathsf{v}) := \mathsf{G}_\mu u(t,0) = \lim_{x\to 0} \mu \frac{d}{dx} u_\varepsilon(t, x) + \frac{d^2}{d x^2} u_\varepsilon(t, x), \quad t>0.
\end{align}

Focus now on the definition given in Section \ref{app:NLOs} for functions on $(0, \infty) \times \mathsf{S}$. The operator
\begin{align*}
D^\Phi_t u(t, \mathsf{x}) = \int_0^t \frac{\partial u}{\partial s}(s, \mathsf{x}) \phi(t-s)ds, \quad t>0,\; \mathsf{x} \in \mathsf{S}
\end{align*}
is well-defined if, $\forall\, \mathsf{x} \in \mathsf{S}$, 
\begin{align*}
t \mapsto u(t, x) \textrm{ belongs to the set } W^{1, \infty}(0, \infty) .
\end{align*}
By following the arguments as in Section \ref{app:NLOs}, we may also ask for the following condition
\begin{align}
\label{condME}
\exists\, M_{\mathsf{S}}>0\,:\, \bigg| \frac{\partial u}{\partial s}(s,\mathsf{x}) \bigg| \leq  M_{\mathsf{S}}\, \frac{\kappa(ds)}{ds}.
\end{align}

We are ready to study the main problem of the work. Let us write the problem \eqref{NLBVPGraphINTRO} as 
\begin{equation}
\left\lbrace
\begin{array}{ll}
\displaystyle \dot{u}(t, \mathsf{x}) = \mathsf{G}_\mu u(t, \mathsf{x}), & t>0,\, \mathsf{x} \in \mathsf{S} \setminus \{\mathsf{v}\}\\
\\
\displaystyle m_r D^\Phi_t u(t, \mathsf{v}) = \sigma_r \sum_{\varepsilon \in \mathcal{E}} \frac{1}{|\mathcal{E}|}\, u^\prime_\varepsilon(t, 0) - c\, u(t, \mathsf{v}), & t>0, \quad c \geq 0,\\  
\displaystyle u_\varepsilon(t, \ell)=0, & t>0, \quad \ell >0, \quad \varepsilon \in \mathcal{E},\\ 
\\
\displaystyle u(0, \mathsf{x}) = f(\mathsf{x}), & \mathsf{x} \in \mathsf{S}, \quad f \in C(\mathsf{S})
\end{array}
\right.
\label{NLBVPGraph}
\end{equation}
where we used the following notation: $\mathsf{S}=\mathsf{S}_{\nu_r}$ is the star graph with vertex $\nu_r=\mathsf{v} \in \mathcal{V}$ characterizes the region $r$; $\mu = v_r$ is the velocity of propagation for the region $r$;  $\Phi=\Phi_r$ is a symbol depending on the vertex $\mathsf{v}$ and given in \eqref{SymbPhiE}; $\eta= m_r/\sigma_r$ characterizes the holding time at $\mathsf{v}$ of $\mathsf{Q}$; $\rho_\varepsilon = 1/|\mathcal{E}|$ assigns the same rate of propagation to all the edges $\mathcal{E}$ of $\mathsf{S}$, that is the possible directions in the region $r$. 

\begin{remark}
We consider $c\geq 0$ for the sake of completeness. However, $c=0$ can be considered as the case of interest in the present paper.
\end{remark}

Let us recall that $\varphi : \mathsf{S} \to \mathbb{R}$ can be written as
\begin{align}
    \varphi(\mathsf{x}) = \varphi(\varepsilon, x) = \varphi_\varepsilon(x), \quad \varepsilon \in \mathcal{E},\, x \in [0,\ell),\, \mathsf{S} \ni \mathsf{x}=(\varepsilon, x)
\end{align}
where $\varphi_\varepsilon$ is the projection of $\varphi$ on the edge $\varepsilon$. Thus, for the star vertex $\mathsf{v}$, we write $\varphi(\mathsf{v}) = \varphi(\cdot, 0)$ which means that $\varphi(\cdot, 0) = \varphi(\varepsilon, 0)$, $\forall\, \varepsilon \in \mathcal{E}$. We introduce the spaces
\begin{align*}
    \mathsf{D}_\phi = \{ \varphi : (0, \infty) \times \mathsf{S} \to \mathbb{R}\, \textrm{with } \varrho = \varphi|_{\mathsf{x=0}}\, \textrm{ s.t. } \,  \dot{\varrho}(s) \phi(t-s) \in L^1(0,t), \, 0<s<t  \}
\end{align*}
as the analogue of $D_\phi$ and 
\begin{align}
    \mathsf{K} = \{ \varphi : (0, \infty) \times \mathsf{S} \to \mathbb{R}\, \textrm{ s.t. } \, \varphi|_{x=\ell} = 0 \}
\end{align}
for the killing condition at the external vertices of $\mathsf{S}$. We now write $\widetilde{\varphi} = \int e^{-\lambda t} \varphi \,dt$ for $\varphi= \varphi(t, \mathsf{x})$ and introduce the condition 
\begin{align}
m_r \frac{\Phi_r(\lambda)}{\lambda} \big( \lambda \widetilde{\varphi}(\lambda, \cdot, 0) - \varphi(0, \cdot, 0) \big) = \sigma_r \sum_{\varepsilon \in \mathcal{E}} \frac{1}{|\mathcal{E}|} \widetilde{\varphi}^{\prime} (\lambda, \varepsilon, 0) - c \, \widetilde{\varphi}(\lambda, \cdot, 0) 
\label{condSpaceU}
\end{align}
with the collection of spaces
\begin{align*}
 \mathsf{U}^r_\phi = \left\{ \varphi \in \mathsf{D}_\phi \cap \mathsf{K} \, \textrm{such that \eqref{condSpaceU} holds true} \right\}
\end{align*}
with $r \in \{1, \ldots, N\}$. Observe that $\{\mathsf{U}^r_\phi \}_r \subseteq Dom(\mathsf{G}_\mu) \subset \{\varphi \in C(\mathsf{S})\,:\, \varphi^\prime \in C(\mathsf{S} \setminus \{\mathsf{v}\})\}=:\mathcal{C}^1(\mathsf{S})$. We write $\varphi(t, \cdot, 0)$ meaning $\varphi(t, \mathsf{v})$, namely $\varphi$ reaches continuously the vertex $\{\mathsf{v}\}$. Observe that $G_\mu u$ is continuous up to the boundary point $\{0\}$. Since $\mathsf{G}_\mu$ acts on $\mathcal{C}^1(\mathsf{S})$ we focus on the spaces $\mathsf{U}^r_\phi$ 
 introduced above.\\

\begin{theorem}
    ({\it Equivalence}) Let $\mathsf{T}_\mathsf{v} := \inf\{ t\,:\, \mathsf{Q}_t = \mathsf{v}\}$ be the first hitting time of the vertex $\mathsf{v}$, that is, the point $(\varepsilon, 0) \equiv 0 \in \mathsf{S}$ and write $\mathsf{T}_0 = \mathsf{T}_\mathsf{v}$. The process $\mathsf{Q} \circ (t \wedge \mathsf{T}_0)$ started at $\mathsf{x} = (\varepsilon, x)$ is equivalent in law to the process $X^{[\varepsilon]} \circ (t \wedge \tau^{[\varepsilon]}_0)$ started at $x$ for any $x \neq 0$ and $\varepsilon \in \mathcal{E}$.
    \label{thm:equiv}
\end{theorem}
\begin{proof}
    The proof follows from the same arguments as in Theorem 4 in \cite{BonDovfcaa}. Notice that both processes behave like a Brownian motion away from $0$.
\end{proof}

We introduce the elastic lifetime $\zeta(\mathsf{S}_\nu)=\zeta_{el}(\mathsf{S}_\nu)$ and the exit time $\tau(\mathsf{S}_\nu)$ of $\mathsf{Q}$ on $\mathsf{S} = \mathsf{S}_\nu$. In particular, 
\begin{align}
    \mathbf{E}_x[f(\mathsf{Q}_t)] = \mathbf{E}_x[f(\mathsf{Q}_t), t < \tau(\mathsf{S}_\nu) \wedge \zeta(\mathsf{S}_\nu)] 
\end{align}
on the star graph $\mathsf{S}=\mathsf{S}_\nu$. The lifetime $\zeta(\mathsf{S}_\nu)$ depends on the elastic condition (with coefficient $c\geq 0$) and $\tau(\mathsf{S}_\nu)$ depends on the killing condition (at the external nodes). From Theorem \ref{thm:equiv}, the motion of $\mathsf{Q}$ on $\mathsf{S}_\nu$ ie equivalent to the motion $X^{[\varepsilon]}$ on $[0, \ell)$, that is
\begin{align}
    \mathbf{E}_{(\varepsilon, x)}[f(\mathsf{Q}_t), t < \tau(\mathsf{S}_\nu) \wedge \zeta(\mathsf{S}_\nu) ] = \mathbf{E}_x [f(X^{[\varepsilon]}_t), t < \tau^{[\varepsilon]} \wedge \zeta^{[\varepsilon]}]
\end{align}
where $\zeta^{[\varepsilon]}$ is the elastic lifetime of $X^{[\varepsilon]}$ on $[0, \ell)$.

\begin{theorem}
For the process $\mathsf{Q}$ on $\mathsf{S}$ started at $\mathsf{x}=(\varepsilon, x)$ and the process $X^{[\varepsilon]}$ on $[0, \ell)$ started at $x$, consider the local times:
\begin{itemize}
\item $\gamma_t(\mathsf{Q})$ which increases only as $Q$ hits the star vertex $\mathsf{v} \equiv (\cdot, 0)$;
\item $\gamma_t(X^{[\varepsilon]})$ which increases only as $X^{[\varepsilon]}$ hits the boundary point $\{0\}$.
\end{itemize}
We have $\gamma_t(\mathsf{Q}) = \gamma_t(X^{[\varepsilon]})$, $t\geq 0$.
\end{theorem}
\begin{proof}
By construction, the path of $\mathsf{Q}$ on a given edge $\varepsilon$ can be regarded as an excursion of $X^{[\varepsilon]}$ in $(0, \ell)$. The process on the edge hits the vertex and starts again, there is no need to collect the selected edges. Thus, the motion on different edges up to the (elastic) killing time can be associated with a Brownian motion on $[0, \ell)$ with reflection at $\{0\}$. 
\end{proof}

\begin{theorem}
The solution $u \in C((0, \infty) \times \mathsf{S}) \cap \mathsf{U}^r_\phi$ to the problem \eqref{NLBVPGraph} with $c=0$ has the probabilistic representation
\begin{align}
u(t,\mathsf{x}) 
& = \mathbf{E}_\mathsf{x}[f(\mathsf{Q}_t)], \quad t\geq 0,\; \mathsf{x} \in \mathsf{S},\, f \in C(\mathsf{S}) \notag \\
& = \mathbf{E}_{(\varepsilon, x)}[f(\Theta_t, X^{[\Theta]}_t), \, t < \tau^{[\Theta]}_\ell \wedge \zeta^{[\Theta]} ], \quad x \in [0,\ell),\, t\geq 0, \, f_\varepsilon \in C[0,\ell), \varepsilon \in \mathcal{E} \label{repSolGrph}
\end{align}
for a given $r \in \{1, \ldots, N\}$. In particular, \eqref{repSolGrph} holds on $\mathsf{S}= \mathsf{S}_{\nu_r}$ $\forall\, r \in \{1, \ldots, N\}$.
\label{thm:NLBVPfinalG}
\end{theorem}
Proof postponed, see Section \ref{proof-thm:NLBVPfinalG}. \\   

For $c>0$ we introduce the elastic lifetime $\zeta(\mathsf{S}_\nu)$. Thus, $\zeta(\mathsf{S}_\nu)$ is equivalent to a killing time for $W$ on $[0, \ell)$, that is $\zeta^W_{kil}$, which differs from $\zeta^E_{abs}$ and $\zeta^W_{reg}$ defined in \eqref{lifeABSandREG}. We have that
\begin{align}
    \mathbf{E}_{(\varepsilon, x)} [\tau(\mathsf{S}_\nu) \wedge \zeta(\mathsf{S}_\nu) ] =  \mathbf{E}_x[\tau^{[\varepsilon]}_\ell \wedge \zeta^{[\varepsilon]}], \quad \varepsilon \in \mathcal{E},\; x \in [0, \ell)
\end{align}
gives the mean value of $\zeta^W_{reg} \wedge \zeta^W_{kil}$. Indeed, $\zeta^W_{reg} \stackrel{d}{=} \tau^{[\varepsilon]}_\ell$ gives the exit time $\tau(\mathsf{S}_\nu)$ of $\mathsf{Q}$ from the star graph $\mathsf{S}_\nu$. For the (elastic) lifetime $\zeta(\mathsf{S}_\nu)$ we recall that, for $\mathsf{Q}_0=\mathsf{x} \in \mathsf{S}_\nu$, for every $t\geq 0$,
\begin{align*}
\mathbf{P}_\mathsf{x}(\zeta(\mathsf{S}_\nu) > t) \to 1 \quad \textrm{as} \quad c \to 0.
\end{align*} 
With \eqref{meanMinimum} and \eqref{meanVcomposto1}-\eqref{meanVcomposto2} at hand, from Theorem \ref{thm:tauLEVELgen} and Theorem \ref{thm:meanStoppedSWcase} we known that
\begin{align}
    \mathbf{E}_0[\tau^{[\varepsilon]}_\ell \wedge \zeta^{[\varepsilon]}] = \frac{\ell}{\mu} + \frac{1+c\ell}{\mu} \frac{e^{-\ell \mu} - 1}{(c + \mu) e^{\ell \mu} - c} + \eta \frac{1}{\mu} \mathbf{E}_0[\gamma \circ (\tau^{\mu}_\ell \wedge \zeta^\mu)] 
\end{align}
For $c=0$, $\zeta^{[\varepsilon]}=\infty$ a.s. and 
\begin{align}
    \mathbf{E}_{(\varepsilon, x)} [\tau(\mathsf{S}_\nu)] = \mathbf{E}_x[\tau^{[\varepsilon]}_\ell], \quad \varepsilon \in \mathcal{E},\; x \in [0, \ell).
\end{align}
According with \eqref{lifeABSandREG} and the definition of $(E,W)$, 
\begin{align}
E \circ \zeta^E_{abs} < h_*\, \Rightarrow \,  W_t = 0, \; t \geq \zeta^E_{abs},
\label{ABScoorsp1}
\end{align}
that is, there is no seismic wave propagation if the accumulation (jump) of energy to be releases is less than the threshold $h_*$. According with \eqref{taunuSecX}, \eqref{jumpXnode}, \eqref{jumpOnlyXnode} we have
\begin{align*}
X^{[\nu]} \circ \tau^{[\nu]}_J = J < h_*
\end{align*}
which implies $X^{[\varepsilon]}_t = 0$ for $t > \tau^{\nu}_J$. That is, as the earthquake reaches a new region, if the accumulation of energy is not enough, then there is no seismic wave propagation in the new region. Thus, recalling \eqref{lifeGabs}, 
\begin{align}
\mathsf{Q}_t = \mathsf{v}_{last}, \quad  t \geq \zeta_{abs}(\mathsf{G})
\label{ABScoorsp2}
\end{align}
where $\mathsf{v}_{last}$ is the \emph{last visited (external) vertex}. This corresponds to \eqref{ABScoorsp1}.

\begin{remark}
We underline that \eqref{ABScoorsp1} and \eqref{ABScoorsp2} correspond to $\{\sigma_r >0, \, 1 \leq r \leq r_*\}$ and $\sigma_{r} = 0$ for $r >r_*$ where $r_*$ is the number of events (visited regions, visited star graphs). Indeed, assume $c=0$, as $\sigma_r=0$ in \eqref{NLBVPGraph} we get absorption instead of partial reflection. This is the case of pure sticky condition. 
\end{remark}

\begin{remark}
We observe that Theorem \ref{thm:NLBVPfinalG} deal with $\mathsf{Q}$ on $\mathsf{S}$. By using the previous arguments we can write, for $t\geq 0,\; \mathsf{x} \in \mathsf{G},\, f \in C(\mathsf{G})$
\begin{align*}
\mathbf{E}_\mathsf{x}[f(\mathsf{Q}_t)] = \mathbf{E}_{(\varepsilon, x)}[f(\Theta_t, X^{[\Theta]}_t), \, t < \zeta_{el}(\mathsf{G}) \wedge \zeta_{abs}(\mathsf{G}) ], 
\end{align*}
with $t\geq 0, \, x \in [0,\ell),\, \varepsilon \in \mathcal{E}$ such that $\mathsf{x}=(\varepsilon, x)$. We recall that $\zeta_{el}(\mathsf{G})$ depends on the elastic coefficient $c\geq 0$ and $\zeta_{abs}(\mathsf{G})$ is that defined in \eqref{lifeGabs}.
\end{remark}

\subsection{Statistical characterization of earthquakes}
\label{sec:Stat}
We consider drifted Brownian motions subjected to non-local boundary value problems in order to describe the behaviour of $E$ and $W$. For an earthquake given by $(E,W)$ on a given area we consider the characterization of the earthquake in terms of the process $\mathsf{Q}$ on the network $\mathsf{G}$ where $\mathsf{G}$ characterizes that geographical area. Thus, we conclude our discussion with a characterization of the energy accumulation $E$ and the wave propagation $W$ in terms of useful statistics obtained from the drifted Brownian motions $X^{[\nu]}$ and $X^{[\varepsilon]}$ together with the angular process $\Theta = \Theta^{[\nu]}$. 

Our model can be identified by the following set of parameters:
\begin{itemize}
\item $m_r$ is the magnitude to be associated with the region $r$,
\item $v_r$ is the velocity of propagation to be associated with the region $r$
\item $\varrho_\varepsilon$ is the rate of propagation along the direction $\varepsilon$,
\item $\sigma_r$ is an auxiliary (delay) parameter (for example due to the distance between the seismic station and the region $r$ or the fact that sensors are not close enough),
\item $c$ is an auxiliary (cessation) parameter (for example due to increased friction, different type of material or more properly, due to hypocenter, how deep down in the earth a quake arises, causes like those arising from induced seismicity). Also $c=c_r$ depending on the region $r$ can be considered. However, we mainly focus, in the present paper, to $c=0$ for the sake of simplicity.
\end{itemize}

Moreover, we have that:
\begin{itemize}
\item $\mathbf{P}(J > h_*)= e^{-(\sigma_r/m_r)h_*}$ is the probability that a new event occurs in the region $r$ given $h_*$ as the minimum level of needed energy,
\item $\zeta^E_{abs}$ depends on $\{(\sigma_s/m_s),\, 0\leq s \leq r\}$. It represents the lifetime of the earthquake until the lack of energy or the fault ending. It is related with the Gutenberg-Richter law (see below),
\item $\zeta^W_{reg}$ has mean value \eqref{meanExitTime} with $\mu=v_r$, $\eta_\varepsilon = \sigma_r/ m_r$. It is the mean time between events or the time the seismic wave spend to reach a new region,
\item $\mathbf{P}(\zeta^W_{kil} > t) = \mathbf{E}[\exp(-c\, \gamma_t(X^{[\varepsilon]}))]$ is the probability the earthquake stops after time $t$ for hidden causes keeping out the case in which the fault simply ends. If $c=0$, then $\mathbf{P}(\zeta^W_{kil} > t) =1$, that is $\zeta^W_{kill} = \infty$ almost surely.
\end{itemize}

For the process $\mathsf{Q}$ on $\mathsf{G}$ we list the following quantities:
\begin{itemize}
\item $\tau(\mathsf{S}_\nu)$ is equivalent to $\zeta^W_{reg}$ and gives the occupation time of $\mathsf{S}_\nu$, that is the time between two subsequent events,
\item $\zeta(\mathsf{G}) = \zeta_{el}(\mathsf{G})$ is equivalent to $\zeta^W_{kill}$, that is the elastic kill for $c\neq 0$, 
\item $\zeta_{abs}(\mathsf{G})$ is equivalent to $\zeta^E_{abs}$ and gives the number of (events) involved regions or star graphs.
\end{itemize}

The fluctuations of $W$ in Figure \ref{fig:relations} represent the motion $\mathsf{Q}$ on $\mathsf{S}$ via equivalence with $(\Theta, X^{[\Theta]})$. An excursion of $X^{[\Theta]}$ on $(0, \ell)$ describes $W$ and represents the motion $\mathsf{Q}$ on a randomly chosen edge. The selection of the edge is done by $\Theta$ as $\mathsf{Q}$ hits a star vertex. 

The fluctuations of $E$ in Figure \ref{fig:relations} represent the accumulations of energy or equivalently, the releases of the accumulated energy. Such fluctuations are equivalent with $X^{[\nu]}$ and provides the time with no events, the time the stored energy overcomes the fault's resistance. In that time, $\mathsf{Q}$ stops on a vertex and equivalently the propagation $W$ equals $0$.

In Figure \ref{fig:relations} we have the $i$-th and $(i+1)$-th seismic waves on a region $r$. Notice that $X^{[\varepsilon]}$ (and $W$) can have in general more than two excursions on $(0, \ell)$. The seismic waves in Figure \ref{fig:relations} are given by $E$ and $W$ with
\begin{align*}
E_0= h_r \quad \textrm{and} \quad W_0=0
\end{align*}
where $h_r$ is an exponential r.v. (see Assumption \ref{ASS3}), that is the jump $J$ for the region $r$. 
\begin{center}
\begin{tabular}{|c|c|}
\hline
I & $\tau^E_i \stackrel{d}{=} (\tau^{[\nu]}_0 | X^{[\nu]}_0=h_r) \stackrel{d}{=} \mathcal{H}^\Phi \circ h_r$ \\
II & $\tau^W_i \stackrel{d}{=} \mathcal{H}^{\Psi} \circ e^{-m_r}_0$ \\
I + II & $\tau_i$ \\
III & $\tau^E_{i+1} \stackrel{d}{=} (\tau^{[\nu]}_0 | X^{[\nu]}_0=h_r) \stackrel{d}{=} \mathcal{H}^\Phi \circ h_r$ \\
IV & $\neq \tau^W_{i+1}$\\
III + IV & $\neq \tau_{i+1}$ \\
I+II+III+IV & $\zeta^W_{reg} \stackrel{d}{=} \tau^{[\varepsilon]}_\ell$ \\
\hline
\end{tabular}
\end{center}
\vspace{.5cm}
\begin{center}
\begin{tabular}{|c|c|c|}
\hline
I & $W=0$, $E \stackrel{d}{=} X^{[\nu]} = X^{\mu_i}$ & $\mu_i = -m_r$\\
II & $E=h_J \sim h_r$, $W \stackrel{d}{=} X^{[\varepsilon]} = X^{\mu_i}$ & $\mu_i=v_r$ \\
I + II & $i$-th seismic wave & \\
III & $W=0$, $E=X^{[\nu]} = X^{\mu_{i+1}}$ & $\mu_{i+1} = - m_r$\\
IV & $E=h_J\sim h_r$, $W \stackrel{d}{=} X^{[\varepsilon]} = X^{\mu_{i+1}}$ & $\mu_{i+1} = v_r$\\
III + IV & $(i+1)$-th seismic wave & \\
I+II+III+IV & $E$, $W$ & \\
\hline
\end{tabular}
\end{center}

%
%
%
\vspace{.5cm}

We stress the fact that our model provides a rigorous basis for simulations. In particular we do not simulate the quantities in the tables above, that quantities are obtained by simulation of suitably combined Brownian motions. Thus, we are able to simulate the hidden mechanism leading to the quantities in the tables above. Such quantities can be confirmed from empirical studies. Inverse Gaussian distributions for example, are commonly used to model inter-earthquake intervals. Fractional calculus and special functions have also recently been considered as useful tools, see for example \cite{CGSS23} and the references therein.

\subsection{Gutenberg-Richter law}
Gutenberg-Richter distribution describes how many earthquakes of a given magnitude will occur in a given region during a given time period. Here we provide a generalized version for the relation between number of earthquakes and magnitude. First we write
\begin{align*}
\mathbf{P}(\sharp\{events\ with\ magnitude\ greter\ than\ h_*\} \geq n)=\mathbf{P}(\sharp\{events\} \geq n)
\end{align*}
in order to streamline the notation. For the network we observe that
\begin{align*}
\sharp\{events\} = \sharp\{visited\ star\ graphs \}.
\end{align*}
Recall that
\begin{align*}
h_* < \min_r \{h_r \} \leq \max_{r} \{h_r\} = h_1 =: h(m_1)
\end{align*}
and $h$ may in general include further arguments, not only the magnitude. According with $\zeta^E_{abs}$ defined in \eqref{lifeABSandREG}, we observe that
\begin{align}
\mathbf{P}(h_* < \min_{1\leq r \leq n} \{h_r \}) = \exp \left( -h_* \sum_{r=1}^n \frac{\sigma_r}{m_r} \right) = \mathbf{P}(\sharp\{events\} \geq n).
\label{hMIN}
\end{align}
If the propagation velocity $v_r$ is large enough, then $W$ starts from zero never to return. Then $i=r$, that is for a given region, $W$ ($X^{[\varepsilon]}$) may have only one excursion on $(0, \ell)$ before to be killed (at $\ell$). In terms of $X^{[\varepsilon]}$ we have that $\tau^\mu_\ell = \tau_1$. In addition, 
\begin{align*}
\eta_\varepsilon = 1/m_r \; \textrm{ that is } \; \sigma_r=m_r^2\, \Rightarrow \,  \mathbf{P}(\sharp\{events\} \geq n) = e^{- h_* n \bar{m}}, 
\end{align*}
\begin{align*}
h_* = \ln 10\, \Rightarrow \, \mathbf{P}(\sharp\{events\} \geq n) = 10^{-n \bar{m}}
\end{align*}
where 
\begin{align*}
\bar{m} = \frac{1}{n} \sum_{r=1}^n m_r
\end{align*}
denotes the (sample) mean magnitude. With no restriction on $v_r$, $W$ is described by $X^{[\varepsilon]}$ which can return to $\{0\}$ infinitely many times. However, \eqref{hMIN} still gives the number of events (or visited graphs/regions). Given a geographical area with $r$ regions characterized by the known values of magnitude $\{m_r,\, r=1,2, \ldots, N\}$,  
\begin{align*}
\mathbf{P}(at\ least\ n\ events\ of\ magnitude\ greater\ than\ \ln a^\beta) = \left( a^{-\beta\, \bar{m}} \right)^n. 
\end{align*} 



\begin{appendix}

\section{Proofs}

\subsection{Proof of Theorem \ref{thm:localTime}}
\label{proof-thm:localTime}
Consider
\begin{align*}
\mathbf{E}_0[\mathbf{1}(X^\mu) M^\mu_t] = \int_0^\infty p(t,0,y) dy
\end{align*}
where, after some calculation, 
\begin{align*}
p(t,0,y) = e^{-\frac{\mu^2}{4}t} \int_0^\infty e^{-(c + \frac{\mu}{2}) w} e^{\frac{\mu}{2}(y-w)} \frac{w+y}{t} g(t, w+y) dw. 
\end{align*}
Thus, we write
\begin{align*}
\mathbf{E}_0[e^{-c\, \gamma^0_t(X^\mu))}] 
= & \int_0^\infty e^{-(c + \frac{\mu}{2})w} e^{-\frac{\mu^2}{4}t}  \int_0^\infty e^{\frac{\mu}{2}(y-w)} \frac{w+y}{t} g(t, w+y)\, dy\, dw\\
= & \int_0^\infty e^{-(c + \frac{\mu}{2}) w} \mathbf{P}_0(\gamma^0_t(X^\mu) \in dw).
\end{align*}
Since 
\begin{align*}
\int_0^\infty e^{-\lambda t} \frac{w+y}{t} g(t,w+y) dt = e^{-(w+y) \sqrt{\lambda}}
\end{align*}
we get
\begin{align*}
\int_0^\infty e^{-\lambda t} \mathbf{P}_0(\gamma^0_t(X^\mu) \in dw) dt 
= & \int_0^\infty e^{\frac{\mu}{2}(y-w)} e^{-(w+y) \sqrt{\lambda + \mu^2/4}} dy\, dw\\
= & \frac{1}{\sqrt{\lambda + \mu^2/4} - \mu/2} e^{-w (\sqrt{\lambda + \mu^2/4} + \mu/2)}\, dw\\
= & \frac{\sqrt{\lambda + \mu^2/4} + \mu/2}{\lambda} e^{-w (\sqrt{\lambda + \mu^2/4} + \mu/2)}\, dw.
\end{align*}
Thus,
\begin{itemize}
\item $\mu<0$ implies that
\begin{align*}
\int_0^\infty e^{-\lambda t} \mathbf{P}_0(\gamma^0_t(X^\mu) \in dw) dt = \frac{\Phi(\lambda)}{\lambda} e^{- w \Phi(\lambda)} dw
\end{align*} 
where
\begin{align*}
\Phi(\lambda) = \sqrt{\lambda + \mu^2/4} - |\mu|/2
\end{align*}
is the symbol of a tempered subordinator of order $1/2$ and tempering parameter $|\mu|/2$;
\item $\mu=0$ comes from the previous case in which $\Phi(\lambda) = \sqrt{\lambda}$ is the symbol of subordinator of order $1/2$. It is well known that the local time of a reflecting Brownian motion equals in law the inverse to a stable subordinator of order $1/2$;
\item $\mu>0$ implies that
\begin{align*}
\frac{\sqrt{\lambda + \mu^2/4} + \mu/2}{\lambda} e^{-w (\sqrt{\lambda + \mu^2/4} + \mu/2)} 
= & - \frac{d}{dw} \frac{1}{\lambda} e^{-w (\sqrt{\lambda + \mu^2/4} + \mu/2)}\\
= & - \frac{d}{dw} \frac{1}{\lambda} e^{-w \Phi(\lambda)} e^{-\mu w}\\
= & - \frac{d}{dw} \frac{1}{\lambda} \mathbf{E}_0[e^{-\lambda \mathcal{H}_w}] \mathbf{E}[\mathbf{1}_{(w < T_\mu)}].
\end{align*}
Write
\begin{align*}
\mathbf{E}_0[e^{-\lambda \mathcal{H}_w}] = 1 - \lambda \int_0^\infty e^{-\lambda t} \mathbf{P}_0(\mathcal{H}_w >t) dt
\end{align*}
and recall that $\mathbf{P}_0(\mathcal{H}_w >t) = \mathbf{P}_0(w > \mathcal{L}_t)$ by definition of inverse process. Then, 
\begin{align*}
\mathbf{E}_0[e^{-\lambda \mathcal{H}_w}] 
= & 1 - \lambda \int_0^\infty e^{-\lambda t} \mathbf{P}_0(w >\mathcal{L}_t) dt\\
= & \lambda \int_0^\infty e^{-\lambda t} \mathbf{P}_0(w \leq \mathcal{L}_t) dt
\end{align*}
and
\begin{align*}
\frac{1}{\lambda} \mathbf{E}_0[e^{-\lambda \mathcal{H}_w}] \mathbf{E}[\mathbf{1}_{(w < T_\mu)}] 
= & \int_0^\infty e^{-\lambda t} \bigg( \mathbf{P}_0(w \leq \mathcal{L}_t) \, \mathbf{P}(w \leq T_\mu) \bigg) dt\\
= & \int_0^\infty e^{-\lambda t} \bigg( \mathbf{P}_0(w \leq \mathcal{L}_t \wedge T_\mu) \bigg) dt.
\end{align*}
This concludes the proof.
\end{itemize}

\subsection{Proof of Theorem \ref{thm:hittingTime}}
\label{proof-thm:hittingTime}
We now consider $(G^\dagger_\mu, D(G^\dagger_\mu))$ where $G^\dagger_\mu \varphi = \varphi^{\prime \prime} + \mu \varphi^{\prime}$ with
\begin{align*}
	D(G^\dagger_\mu) = \left\lbrace \varphi, G^\dagger_\mu \varphi \in C_b((0, \infty))\,:\, \varphi(0^+)=0 \right\rbrace.
\end{align*}
The Dirichlet boundary condition can be obtained as the case $c\to \infty$ for the generator $(G_\mu, D(G_\mu))$ previously introduced. We have that, for $x \in (0, \infty)$, $\mathbf{P}_x(M_t^\mu=1, t < \tau_0^\mu)=1$ and
\begin{align*}
\mathbf{P}_x(\tau_0^\mu > t) 
= & \mathbf{E}_x[\mathbf{1}(\widetilde{X}^\mu_t), t< \tau_0^\mu] = \mathbf{E}_x[\mathbf{1}(X^\mu_t) M^\mu_t, t < \tau_0^\mu] \\
= & \int_0^\infty e^{-\frac{\mu^2}{4}t} e^{\frac{\mu}{2}(y-x)} \left[ g(t,x-y) - g(t,x+y) \right]dy.
\end{align*}
We recall that for $\mu=0$,
\begin{align*}
\int_0^\infty e^{-\lambda t} \mathbf{P}_x(\tau_0^0 > t) dt = \frac{1-e^{-x\sqrt{\lambda}}}{\lambda}
\end{align*}
and $\tau^0_0 | \widetilde{X}^0_0=x$ equals in law $H_x | H_0=0$. For $\mu \neq 0$,
\begin{align*}
& \int_0^\infty e^{-\lambda t}  \mathbf{E}_x[\mathbf{1}(\widetilde{X}^\mu_t), t< \tau_0^\mu]\, dt\\
= & \frac{1}{2} \int_0^\infty \frac{e^{(x-y)\sqrt{\lambda + \mu^2/4}}}{\sqrt{\lambda + \mu^2/4}} e^{\frac{\mu}{2}(y-x)}\, dy - \frac{1}{2} \int_0^\infty \frac{e^{-(x+y)\sqrt{\lambda + \mu^2/4}}}{\sqrt{\lambda + \mu^2/4}} e^{\frac{\mu}{2}(y-x)}\, dy \\
 - & \frac{1}{2} \int_0^x \left( \frac{e^{(x-y)\sqrt{\lambda + \mu^2/4}}}{\sqrt{\lambda + \mu^2/4}} - \frac{e^{-(x-y)\sqrt{\lambda + \mu^2/4}}}{\sqrt{\lambda + \mu^2/4}} \right) e^{\frac{\mu}{2}(y-x)}\, dy.
\end{align*}
Write $\lambda^\prime = \lambda + \mu^2/4$ with $\lambda>0$. Then $\sqrt{\lambda^\prime} > |\mu|/2$ and we get
\begin{align*}
& \int_0^\infty e^{-\lambda t}  \mathbf{E}_x[\mathbf{1}(\widetilde{X}^\mu_t), t< \tau_0^\mu]\, dt\\
= & \frac{1}{2} \frac{e^{-x (\frac{\mu}{2}-\sqrt{\lambda^\prime})} }{\sqrt{\lambda^\prime}} \frac{1}{\sqrt{\lambda^\prime} - \mu/2} - \frac{1}{2} \frac{e^{-x (\frac{\mu}{2} + \sqrt{\lambda^\prime})}}{\sqrt{\lambda^\prime}} \frac{1}{\sqrt{\lambda^\prime} - \mu/2}\\ 
& - \frac{1}{2} \left( \frac{e^{-x(\mu/2 - \sqrt{\lambda^\prime})}}{\sqrt{\lambda^\prime}} \frac{1-e^{-x(\sqrt{\lambda^\prime}-\mu/2)}}{\sqrt{\lambda^\prime}-\mu/2} - \frac{e^{-x(\sqrt{\lambda^\prime}+\mu/2)}}{\sqrt{\lambda^\prime}} \frac{e^{x(\sqrt{\lambda^\prime}+\mu/2)}-1}{\sqrt{\lambda^\prime}+\mu/2} \right)\\
= & \frac{1}{2} \frac{e^{-x (\frac{\mu}{2}-\sqrt{\lambda^\prime})} }{\sqrt{\lambda^\prime}} \frac{1}{\sqrt{\lambda^\prime} - \mu/2} - \frac{1}{2} \frac{e^{-x (\frac{\mu}{2} + \sqrt{\lambda^\prime})}}{\sqrt{\lambda^\prime}} \frac{1}{\sqrt{\lambda^\prime} - \mu/2}\\ 
& - \frac{1}{2} \left( \frac{1}{\sqrt{\lambda^\prime}} \frac{e^{-x(\mu/2 - \sqrt{\lambda^\prime})}-1}{\sqrt{\lambda^\prime}-\mu/2} - \frac{1}{\sqrt{\lambda^\prime}} \frac{1 - e^{-x(\sqrt{\lambda^\prime}+\mu/2)}}{\sqrt{\lambda^\prime}+\mu/2} \right)\\
= & \frac{1}{2\sqrt{\lambda^\prime} (\sqrt{\lambda^\prime} - \mu/2)} + \frac{1}{2\sqrt{\lambda^\prime}( \sqrt{\lambda^\prime}+\mu/2)} - \frac{e^{-x(\sqrt{\lambda^\prime} + \mu/2)}}{2\sqrt{\lambda^\prime}( \sqrt{\lambda^\prime}+\mu/2)} - \frac{e^{-x (\frac{\mu}{2} + \sqrt{\lambda^\prime})}}{2\sqrt{\lambda^\prime}(\sqrt{\lambda^\prime} - \mu/2)} \\
= & \frac{1}{2\sqrt{\lambda^\prime} (\sqrt{\lambda^\prime} - \mu/2)} + \frac{1}{2\sqrt{\lambda^\prime}( \sqrt{\lambda^\prime}+\mu/2)} - \frac{e^{-x (\frac{\mu}{2} + \sqrt{\lambda^\prime})}}{\lambda^\prime - \mu^2/4}\\
= & \frac{1 - e^{-x (\frac{\mu}{2} + \sqrt{\lambda^\prime})}}{\lambda^\prime - \mu^2/4}.
\end{align*}
That is,
\begin{align*}
\int_0^\infty e^{-\lambda t} \mathbf{E}_x[\mathbf{1}(\widetilde{X}^\mu_t), t< \tau_0^\mu]\, dt = \frac{1 - e^{-x (\frac{\mu}{2} + \sqrt{\lambda + \mu^2/4})}}{\lambda}
\end{align*}
We observe that, as $\lambda \to 0$, the previous formula says that $\mathbf{E}_x[\tau^\mu_0] < \infty$ only in case $\mu<0$. Since 
\begin{align*}
\mathbf{E}_x[e^{-\lambda \tau_0^\mu}] 
= & 1 - \lambda \int_0^\infty e^{-\lambda t} \mathbf{P}_x(\tau_0^\mu >t) dt \\
= & e^{-x (\frac{\mu}{2} + \sqrt{\lambda + \mu^2/4})}\\
= & \left\lbrace
\begin{array}{ll}
\displaystyle e^{-x (\sqrt{\lambda + \mu^2/4}- |\mu|/2)} = e^{-x\Phi(\lambda)}, & \mu\leq 0\\
\\
\displaystyle e^{-x (\sqrt{\lambda + \mu^2/4}- \frac{\mu}{2})} e^{- \mu x}, & \mu>0
\end{array}
\right.
\end{align*}
we conclude that
\begin{align*}
\mathbf{E}_x[e^{-\lambda \tau_0^\mu}] = \mathbf{E}_0[e^{-\lambda \mathcal{H}_x}], \quad \textrm{if} \quad \mu\leq 0
\end{align*}
and
\begin{align*}
\mathbf{E}_x[e^{-\lambda \tau_0^\mu}] = \mathbf{E}_0[e^{-\lambda \mathcal{H}_x} \mathbf{1}_{(x < T_\mu)}] = \mathbf{E}_0[e^{-\lambda \mathcal{H}^\dagger_x}], \quad \textrm{if} \quad \mu > 0
\end{align*}
where:
\begin{itemize}
\item $\mathcal{H}$ is a tempered subordinator with symbol 
\begin{align*}
\Phi(\lambda) = \sqrt{\lambda + \theta} - \sqrt{\theta}, \quad \theta=(\mu/2)^2; 
\end{align*}
\item $T_\mu$ is the exponential r.v. such that $\mathbf{P}(T_\mu>t)=e^{-\mu t}$; 
\item $\mathcal{H}^\dagger$ is a killed tempered subordinator with symbol
\begin{align*}
\Phi^\dagger(\lambda) = \mu + \Phi(\lambda).
\end{align*}
\end{itemize}
In particular, for positive drift, 
\begin{align}
    \big(\tau^\mu_0 | X^\mu_0=x \big) = \left\lbrace
        \begin{array}{ll}
             \mathcal{H}_x & x < T_\mu  \\
             +\infty & x\geq T_\mu 
        \end{array}
        \right .
\end{align}
This concludes the proof.

\subsection{Proof of Theorem \ref{thm:meanStoppedSWcase}}
\label{proof-thm:meanStoppedSWcase}

The result can be obtained by considering the associated elliptic problem. However, we provide further details for the interested readers.

    The mean value $\mathbf{E}_x[\tau^\mu_\ell]$ is known from Corollary \ref{coro:tauLEVEL}. The mean value $\mathbf{E}_x[\gamma \circ \tau^\mu_\ell]$ can be obtained by considering the mean difference between 
    \begin{align}
        \tau_\ell^{\delta, \mu} := \inf\{t\,:\, X^{\delta, \mu}_t =\ell\} \quad \textrm{and} \quad \tau^\mu_\ell := \inf\{t\,:\, X^\mu_t = \ell\}
    \end{align}
    where $X^{\delta, \mu}$ is driven by \eqref{PDEdelta} in which the elastic coefficient equals zero (observe that $c=0$ without loss of generality). For $\delta \in \{0,1\}$, the solution to \eqref{PDEdelta} has the representation
    \begin{align}
        u(t,x) = \mathbf{E}_x[f(X^\mu \circ V^{-1}_t), V^{-1}_t < \tau^\mu_\ell] = \mathbf{E}_x[f(X^\mu \circ V^{-1}_t), t < V \circ \tau^\mu_\ell]
    \end{align}
    where $V_t = t + \eta(\delta) \gamma_t$ produces the sticky behaviour on the boundary point and 
    \begin{align}
        \mathbf{P}_x(t < \tau^{\delta, \mu}_\ell) = \mathbf{P}_x(t < V \circ \tau^\mu_\ell), \quad t \geq 0.
    \end{align}
    We are only interested on mean hitting times, then we restrict our analysis on 
    \begin{align}
        v(x) = \int_0^\infty u(t,x) dt = \mathbf{E}_x\left[ \int_0^{V \circ \tau^\mu_\ell} f(X^\mu \circ V^{-1}_t) dt\right], \quad x \in [0, \ell).
        \label{vProbRep}
    \end{align}
    In particular, $\mathbf{E}_x[\tau^{\delta, \mu}_\ell]= \mathbf{E}_x[V \circ \tau^\mu_\ell]$ where
    \begin{align}
        \mathbf{E}_x[V \circ \tau^\mu_\ell] = \mathbf{E}_x[\tau^\mu_\ell] + \eta(\delta) \mathbf{E}_x[\gamma \circ \tau^\mu_\ell]. 
        \label{meanExtraProof}
    \end{align}
    
    In case $\delta=0$ we pass to the elliptic problem
    \begin{align*}
        \left\lbrace
        \begin{array}{ll}
             v^{\prime \prime}(x) +\mu v^\prime (x) = - 1, & x\in (0, \ell),\; \mu >0,  \\
            \eta v^{\prime \prime} (0) = v^\prime (0), &  \eta=\eta(0)>0,\\
            v(\ell) = 0 & 
        \end{array}
        \right .
    \end{align*}
    with solution 
    \begin{align}
        v(x) = \frac{\ell - x}{\mu} - \frac{e^{-\mu x} - e^{-\mu \ell}}{\mu^2} \left( 1 - \frac{\eta \mu}{1+\eta \mu} \right).
    \end{align}
    Thus, we obtain
    \begin{align}
        \eta(0) \mathbf{E}_x[\gamma \circ \tau^\mu_\ell] = \frac{\eta}{1 + \eta \mu} \frac{e^{-\mu x} - e^{-\mu \ell}}{\mu} 
    \end{align}
    for the extra time due to the second order boundary condition. \\

    In case $\delta=1$, we focus on the problem
    \begin{align*}
        \left\lbrace
        \begin{array}{ll}
             v^{\prime \prime}(x) +\mu v^\prime (x) = - 1, & x\in (0, \ell),\; \mu >0,  \\
            \eta \big (v^{\prime \prime} (0) +\mu u^\prime(0) \big) = v^\prime (0), &  \eta=\eta(1) \geq 0,\\
            v(\ell) = 0. & 
        \end{array}
        \right .
    \end{align*}
Here we get 
\begin{align}
    v(x) = \frac{\ell - x}{\mu} - \frac{e^{-\mu x} - e^{-\mu \ell}}{\mu^2} (1 - \eta \mu)
    \label{vUniqueSol}
\end{align}
with mean extra time 
    \begin{align}
        \eta(1) \mathbf{E}_x[\gamma \circ \tau^\mu_\ell] = \eta \frac{e^{-\mu x} - e^{-\mu \ell}}{\mu}.
    \end{align}
    Observe that $\eta(\delta) = \eta/(1+\eta (1-\delta) \mu) \geq 0$ gives the connection between problems in terms of $\delta \in \{0,1\}$. 
    Moreover, $V_t \geq t$ a.s. implies that
\begin{align*}
\eta(\delta) \mathbf{E}[\gamma \circ \tau^\mu_\ell] = \mathbf{E}[V \circ \tau^\mu_\ell - \tau^\mu_\ell]
\end{align*} 
is non negative as expected.

\subsection{Proof of Theorem \ref{thm:HTrdBM}} 
\label{proof-thm:HTrdBM}
Let us consider the natural filtration 
$$\mathcal{F}_t = \sigma\{X^{\delta, \mu}_s, \, 0\leq s < t\}.$$ 
The hitting time 
$$\tau_{(0, \ell)} = \inf\{t\,:\, X^{\delta, \mu}_t \in (0,\ell)\}$$ 
is an $\{\mathcal{F}_t\}$-stopping time and $\{\tau_{(0,\ell)} = 0 \} \subset \mathcal{F}_0$. Every $x\in [0, \ell)$ is regular for $(0,\ell)$ and $\mathbf{P}_x(\tau_{(0, \ell)} = 0) = 1$ if $x \in (0, \ell)$. For $X^{\delta, \mu}_0=0$, $\tau_{(0, \ell)} = e^\mu_0$ implies 
$$\mathbf{P}_0(X^{\delta, \mu}_{e^\mu_0}\in (0, \ell)) = 1$$ 
and 
$$\mathbf{P}_0(\tau_{(0,\ell)}  > 0) = 1,$$ 
that is 
\begin{align*}
\lim_{t \to 0} \mathbf{P}_0(e^\mu_0 >t) = \lim_{t \to 0} \mathbf{P}_0(X^{\delta, \mu}_t = 0)  = 1.
\end{align*} 
By definition $e^\mu_0$ is a holding time if $\mathbf{P}_0(X^{\delta, \mu}_{e^\mu_0} \in (0, \ell)) = 1$. On the other hand, $X^{\delta, \mu}$ is a Markov process on $[0, \ell)$, then
\begin{align*}
\mathbf{P}_0(e^\mu_0 > t, X^{\delta, \mu}_{e^\mu_0} \in (0, \ell)) = \mathbf{P}_0(e^\mu_0 > t), \quad t \geq 0
\end{align*}
which is \eqref{indHTfromX}. The process starts afresh after the holding time. 

By exploiting once again the Markovian nature of $X^{\delta, \mu}$ we have that
\begin{align*}
\mathbf{P}_0(X^{\delta, \mu}_{t+s} =0) = \mathbf{E}_0[\mathbf{E}_{X^{\delta, \mu}_s}[\mathbf{1}_{\{0\}}(X^{\delta, \mu}_t)]]
\end{align*}
and, by definition of $e^\mu_0$,
\begin{align*}
\mathbf{P}_0(X^{\delta, \mu}_{t+s} =0) = \mathbf{P}_0(e^\mu_0 > t+s).
\end{align*}
We get
\begin{align*}
\mathbf{E}_0[\mathbf{E}_{X^{\delta, \mu}_s}[\mathbf{1}_{\{0\}}(X^{\delta, \mu}_t)]] = \mathbf{E}_0[\mathbf{P}_{X^{\delta, \mu}_s}(e^\mu_0>t) \mathbf{1}_{\{0\}} (X^{\delta, \mu}_s)].
\end{align*}
with
\begin{align*}
\mathbf{P}_x(e^\mu_0>t) = \mathbf{E}_x[\mathbf{1}_{\{0\}} (X^{\delta, \mu}_t)]\, \mathbf{1}_{\{0\}} (x).
\end{align*}
That is, 
\begin{align*}
\mathbf{E}_0[\mathbf{E}_{X^{\delta, \mu}_s}[\mathbf{1}_{\{0\}}(X^{\delta, \mu}_t)]] = \mathbf{P}_{0}(e^\mu_0>t) \, \mathbf{E}_0[\mathbf{1}_{\{0\}} (X^{\delta, \mu}_s)] = \mathbf{P}_{0}(e^\mu_0>t)\, \mathbf{P}_{0}(e^\mu_0>s).
\end{align*}
This means that
\begin{align*}
\mathbf{P}_0(e^\mu_0 >t+s) = \mathbf{P}_0(e^\mu_0 >s) \mathbf{P}_0(e^\mu_0 >t)
\end{align*}
and $e^\mu_0$ is an exponential random variable. Moreover, the law of $e^\mu_0$ does not depend on $X^{\delta, \mu}_0 = 0$. Since the absorption (for $\eta \to \infty$) and the instantaneous reflection (for $\eta \to 0$) must  respectively imply, for all $t \geq 0$,
\begin{align*}
\mathbf{P}(e^\mu_0 >t) = 1 \quad \textrm{and} \quad \mathbf{P}(e^\mu_0 > t) = 0,
\end{align*} 
we must have $\mathbf{P}(e^\mu_0 > t) = e^{-(1/\eta) t}$, $t \geq 0$. Indeed, for $V_t = t + \eta \gamma_t$, we know that
\begin{align*}
\mathbf{E}[V_t - t] = \eta \mathbf{E}[\gamma_t] 
\end{align*} 
that is, the extra time at $\{0\}$ can be written as
\begin{align*}
\mathbf{E}[V_t - t] = \mathbf{E} \left[ \int_0^t e_s \, d\gamma_s \right]
\end{align*}
where 
\begin{equation}
e_s = 
\left\lbrace
\begin{array}{ll}
\chi, & s \in \{t\,:\, \gamma^{-1}_{t-} \neq \gamma^{-1}_t \}\\
e_{s-}, & \textrm{otherwise}
\end{array}
\right . \quad e_0=e^\mu_0 \stackrel{d}{=} \chi
\end{equation}
and $\chi$ is an exponential r.v. independent from $\gamma_t$. Thus $e_s$ is a step function jumping to an exponential value according with the jump of $\gamma^{-1}$, the inverse of the local time $\gamma$. These exponential values define the exponential holding times. We therefore obtain
\begin{align*}
\mathbf{E}[V_t - t] = \mathbf{E}[\chi] \mathbf{E} \left[ \int_0^t  d\gamma_s \right] = \mathbf{E}[\chi] \, \mathbf{E}[\gamma_t]
\end{align*}
with $\mathbf{E}[\chi] = \eta$, that is $\mathbf{P}(\chi > t) = e^{-(1/\eta) t}$. This also implies that $\{e^\mu_i \stackrel{d}{=} \chi\, ,\, i \in \mathbb{N}\}$ is the sequence of holding times. Indeed, $\{t\,:\, \gamma^{-1}_{t-} \neq \gamma^{-1}_t \}$ is a countable set.

\subsection{Proof of Theorem \ref{thm:NLBVPwithDir}}
\label{proof-thm:NLBVPwithDir}
The time change $V_t = t + \mathcal{H} \circ \eta_\varepsilon \gamma_t $ is right-continuous (and increasing) with continuous inverse $V^{-1}_t$ such that $V^{-1} \circ V_t = t$. We also use the fact that $\mathcal{H} \perp X^\mu$ and
\begin{align}
    \mathbf{E}_x[e^{-\lambda V_t} | X^\mu] = e^{-\lambda t - \eta_\varepsilon \Phi(\lambda) \gamma_t}, \quad \lambda>0.
\end{align}
We have that
\begin{align*}
& \mathbf{E}_x\left[\int_0^{V \circ \tau^\mu_\ell} e^{-\lambda t} f(X^\mu \circ V^{-1}_t) \, M^\mu \circ V^{-1}_t \, dt \right] \\
    = & \mathbf{E}_x\left[\int_0^{\tau^\mu_\ell} e^{-\lambda V_t} f(X^\mu_t) \, M^\mu_t \, dV_t \right]\\
    = & - \frac{1}{\lambda} \mathbf{E}_x\left[\int_0^{\tau^\mu_\ell} f(X^\mu_t) \, M^\mu_t \, de^{-\lambda V_t}  \right]\\
    = & - \frac{1}{\lambda} \mathbf{E}_x\left[\int_0^{\tau^\mu_\ell} f(X^\mu_t) \, M^\mu_t \, de^{-\lambda t - \eta_\varepsilon \Phi(\lambda) \gamma_t}  \right]\\
    = & \mathbf{E}_x\left[\int_0^{\tau^\mu_\ell} f(X^\mu_t) \, M^\mu_t \, e^{-\lambda t - \eta_\varepsilon \Phi(\lambda) \gamma_t} (dt + \eta_\varepsilon \frac{\Phi(\lambda)}{\lambda} d\gamma_t) \right]
\end{align*}
Now set
\begin{align}
    V_{\lambda, t} := t + \eta_\varepsilon \frac{\Phi(\lambda)}{\lambda} \gamma_t 
    \label{newV}
\end{align}
and write
\begin{align*}
& \mathbf{E}_x\left[\int_0^{V \circ \tau^\mu_\ell} e^{-\lambda t} f(X^\mu \circ V^{-1}_t) \, M^\mu \circ V^{-1}_t \, dt \right] \\
    = & \mathbf{E}_x\left[\int_0^{\tau^\mu_\ell} f(X^\mu_t) \, M^\mu_t \, e^{-\lambda V_{\lambda, t}} dV_{\lambda, t} \right]\\
    = & \mathbf{E}_x\left[\int_0^{V_\lambda \circ \tau^\mu_\ell} e^{-\lambda t} f(X^\mu \circ V^{-1}_{\lambda, t}) \, M^\mu \circ V^{-1}_{\lambda, t} \, dt \right]\\
    = & : R_\lambda f(x), \quad x \in [0, \ell), \quad \lambda > 0,
\end{align*}
that is, $R_\lambda f = \int e^{-\lambda t} u\, dt$. Observe that the new time change
\begin{align}
    V_{\lambda, t} = \int_0^t \gamma^z_t m_\lambda(dz), \quad m_\lambda(dz) = dz + \eta_\varepsilon \frac{\Phi(\lambda)}{\lambda} \delta_0(dz)
\end{align}
introduces the (Dirac) measure on the boundary point $\{0\}$ for the sticky effect. Let us write
\begin{align}
    V_{\beta, t} = \int_0^t \gamma^z_t m_\beta(dz), \quad m_\beta(dz) = dz + \eta_\varepsilon \frac{\Phi(\beta)}{\beta} \delta_0(dz), \quad \beta >0.
\end{align}
Since $X^\mu \circ V^{-1}_{\beta, t}$ is a sticky Brownian motion, $\forall \, \beta>0$, we have that
\begin{align}
    Q^{\mu, \beta}_t f(x) := \mathbf{E}_x\left[ f(X^\mu \circ V^{-1}_{\beta, t}) \, M^\mu \circ V^{-1}_{\beta, t}, \, t < V_\beta \circ \tau^\mu_\ell \right], \quad t\geq 0,\; x \in [0, \ell)]
    \label{RepL2}
\end{align}
is a $C_0$-semigroup on $L^2(m_\beta)$. Moreover, there exists a continuous kernel, say $p_\beta$, for which \eqref{RepL2} has the representation
\begin{align*}
    \int_{[0, \infty)} f(y) \, p_\beta (t,x, y) m_\beta(dy), \quad f \in C[0, \infty)]
\end{align*}
which provides the $C_0$-semigroup on $C[0, \infty)$. In particular, for every $\beta>0$, the semigroup \eqref{RepL2} solves the problem
\begin{equation}
\left\lbrace
\begin{array}{ll}
     \displaystyle \frac{\partial Q^{\mu, \beta}_t f}{\partial t} = G_\mu \,Q^{\mu, \beta}_t f & \textrm{in } (0, \infty) \times (0, \ell) \\
     \\
     \displaystyle \eta_\varepsilon \frac{\Phi(\beta)}{\beta} \, G_\mu Q^{\mu, \beta}_t f = (Q^{\mu, \beta}_t f)^\prime - c\, Q^{\mu, \beta}_t f & \textrm{in } (0, \infty) \times \{0\}  \\
     \\
     \displaystyle Q^{\mu, \beta}_t f = 0 & \textrm{in } (0, \infty) \times \{\ell\}\\
     \\
     \displaystyle Q^{\mu, \beta}_0 f = f & f \in C[0, \ell)].
\end{array}
\right .
\label{eqQsemig}
\end{equation}
We only add a Dirichlet stopping time to the problem \eqref{SWpde}. Moreover, for 
\begin{align*}
R^\beta_\lambda f(x) := \mathbf{E}_x \left[ \int_0^\infty e^{-\lambda t} Q^{\mu, \beta}_t f(x)\, dt \right] 
\end{align*}
we have
\begin{align*}
R_\lambda f = R^\beta_\lambda f \quad \textrm{in case} \quad \beta=\lambda.
\end{align*}
Thus, $R_\lambda f \in D(G^{1, c}_\mu)$  and
\begin{align*}
    R_\lambda f, G_\mu R_\lambda f \in C[0, \ell) \; : \; \eta_\varepsilon \frac{\Phi(\lambda)}{\lambda} G_\mu R_\lambda f |_{x=0} = (R_\lambda f)^\prime |_{x=0} - c R_\lambda f |_{x=0}.
\end{align*}
The fact that $u \in D_\phi$ ensures the existence of $D^\Phi_t u$ and
\begin{align}
    \int_0^\infty e^{-\lambda t} \, \eta_\varepsilon\,  D^\Phi_t u(t,0) \, dt = \eta_\varepsilon \frac{\Phi(\lambda)}{\lambda} \Big( \lambda R_\lambda f - f \Big) \Big|_{x=0}, \quad \lambda>0.
\end{align}
Thus, from \eqref{RepL2} we write 
\begin{align*}
    R_\lambda f \in C[0, \ell) \; : \; \eta_\varepsilon \frac{\Phi(\lambda)}{\lambda} \Big( \lambda R_\lambda f - f \Big) |_{x=0} =  (R_\lambda f)^\prime |_{x=0} - c R_\lambda f |_{x=0}
\end{align*}
and the dynamic boundary condition 
\begin{align}
    \eta_\varepsilon\,  D^\Phi_t u(t,0) = u^\prime(t, 0) - u(t, 0), \quad t>0
\end{align}
appears. Since $u(\cdot, x) \in C(0, \infty)$ $\forall\, x \in [0, \ell)$, then $R_\lambda f$ has a unique inverse.

\subsection{Proof of Theorem \ref{thm:NLBVP}}
\label{proof-thm:NLBVP}

Assume $X^\bullet$ is a Markov process for which $u(t,x) = \mathbf{E}_x[f(X^\bullet_t)]$ solves the problem
\begin{equation}
\left\lbrace
\begin{array}{ll}
\displaystyle \dot{u}(t,x) = G_\mu u(t,x), & t>0,\, x \in (0, \infty), \quad \mu<0,\\ 
\\
\displaystyle \mathbf{D}^\Upsilon_x u(t, x) |_{x=0} = 0, & t>0,\\
\\
\displaystyle u(0, x) = f(x), & x \in [0, \infty), \quad f \in C_b[0, \infty).
\end{array}
\right .
\label{app-thm:NLBVP}
\end{equation}
Then, from Theorem \ref{thm:NLBVPwithDir}, $X^\bullet \circ V^{-1}_t$ solves the problem \eqref{NLBVPnodes}. 

Now we argue on \eqref{app-thm:NLBVP} by considering the likelihood ratio $\mathbb{L}$ such that
\begin{align*}
\mathbf{E}^{\mathbf{P}} [f(X^\mu_t)] = \mathbf{E}^{\mathbb{P}} [\mathbb{L}(Y_t) f(Y_t)].
\end{align*}
In particular, 
\begin{align*}
u(t, x) 
= & \mathbf{E}_x[f(X^\mu_t)] = e^{- \frac{\mu}{2} x} \mathbf{E}_x[e^{\frac{\mu}{2} Y_t -\frac{\mu^2}{4} t} f(Y_t)] = e^{- \frac{\mu}{2} x} e^{-\frac{\mu^2}{4} t} v(t,x)
\end{align*}
where $Y=\{Y_t\}_{t\geq 0}$ is an elastic Brownian motion driven by
\begin{equation}
\left\lbrace
\begin{array}{ll}
\displaystyle \dot{v} = v^{\prime \prime}, & (0, \infty) \times [0, \infty),\\
\displaystyle v^\prime = (c + \frac{\mu}{2}) v, & (0, \infty) \times \{0\},\\ 
\displaystyle v(0,x) = f(x), & f \in C_b([0, \infty)).
\end{array}
\right .
\end{equation}
We get $\dot{u} = u^{\prime \prime} - \mu u^{\prime}$ and $\dot{v}=v^{\prime \prime}$, moreover we obtain $v^\prime = 0$ by setting $c=\frac{|\mu|}{2}$, recall that $\mu<0$. Thus, $Y$ is a reflected (elastic) Brownian motion. Let us define new functions $u$ and $v$:
\begin{align*}
u(t,x) = e^{- \frac{\mu}{2} x} \mathbf{E}_x[e^{\frac{\mu}{2} Y_t + \frac{\mu}{2} A\circ \gamma_t(Y) -\frac{\mu^2}{4} t} f(Y_t + A\circ \gamma_t(Y))] = e^{-\frac{\mu}{2} x} e^{-\frac{\mu^2}{4} t} v(t,x).
\end{align*} 
Then, from \cite{BonColDovPag}, the new function $v$ solves
\begin{equation}
\left\lbrace
\begin{array}{ll}
\displaystyle \dot{v} = v^{\prime \prime}, & (0, \infty) \times [0, \infty),\\
\displaystyle {\bf D}^\Upsilon_x v = (c + \frac{\mu}{2}) v, & (0, \infty) \times \{0\},\\ 
\displaystyle v(0,x) = f(x), & f \in C_b([0, \infty)).
\end{array}
\right .
\end{equation}
Assume $t=0$ and observe that
\begin{align*}
& \lim_{x\to 0} {\bf D}^\Upsilon_x u(x) = \lim_{x\to 0} {\bf D}^\Upsilon_x e^{-\frac{\mu}{2} x} v(x)\\
= & \lim_{x \to 0} \int_0^\infty \left( e^{-\frac{\mu}{2} x} v(x) - e^{-\frac{\mu}{2} (x+z)} v(x+z) \right) \mathbf{P}(J > z) dz \\
= & \lim_{x \to 0} e^{-\frac{\mu}{2} x + \frac{1}{\eta_\varepsilon} x} \int_x^\infty \left( v(x) - v(y) \right) e^{-\frac{1}{\eta_\varepsilon} y} dy\\
= & \lim_{x \to 0} \frac{e^{-\frac{\mu}{2} x}}{\mathbf{P}(J>x)} \int_x^\infty \left( v(x) - v(y) \right) \mathbf{P}(J >y) dy\\
= & \int_0^\infty \left( v(0) - v(y) \right) \mathbf{P}(J >y) dy\\
= & {\bf D}^\Upsilon_x v(x) \big|_{x=0}.
\end{align*}
Thus, for $t\geq 0$, we get
\begin{equation}
\left\lbrace
\begin{array}{ll}
\displaystyle \dot{u} = u^{\prime \prime} - \mu u^\prime, & (0, \infty) \times [0, \infty),\\
\displaystyle {\bf D}^\Upsilon_x u = (c + \frac{\mu}{2}) u , & (0, \infty) \times \{0\},\\ 
\displaystyle u(0,x) = f(x), & f \in C_b([0, \infty)).
\end{array}
\right .
\end{equation}
Notice that $\mu < 0$. As $c+\mu/2 = 0$, we get the claim.

\subsection{Proof of Theorem \ref{thm:equivE}} 
\label{proof-thm:equivE}

Under the setting of the previous sections, we consider the problem \eqref{NLBVPnodes} and the representation in Theorem \ref{thm:NLBVP}. Recall that $X^{[\nu]}_t = X^\bullet \circ V^{-1}$ is right-continuous and the jumps are obtained as the process approaches the point $\{0\}$ according with $\Upsilon$ (see Section \ref{sec:UpsilonExpJumps}). It is crucial that Assumption \ref{ASS3} holds true. 

By definition of $E$, for $f\in C_b(0, \infty)$ we can write 
\begin{align}
    \mathbf{E}_{h_0}[f(E_t), t < \tau_1] = \mathbf{E}_{h_0}[f(X^{\mu_1}_t),\, 0 \leq t < \tau^\mu_0] + \mathbf{E}_{h_0}[f(h_1),\, \tau^\mu_0 \leq t < \tau_1]
    \label{Epdeh0}
\end{align}
and, for a given $i \in \mathbb{N}$, we have that
\begin{align*}
    \mathbf{E}_{h_{i-1}}[f(E_t), \tau_{i-1} \leq t < \tau_i] = \mathbf{E}_{h_{i-1}}[f(X^{\mu_i}_t),\, 0 \leq t < \tau^{\mu_i}_0] + \mathbf{E}_{h_{i-1}}[f(h_i),\, \tau^{\mu_i}_0 \leq t < \tau_1].
\end{align*}
Recall that $X^\mu$ is a Markov process. Focus on the latter formula.\\

We consider that $\tau_i = \tau^E_i + \tau^W_i$ is the sum of the running time and the holding time at $J=h_i$ of $E$. We use the fact that (Theorem \ref{thm:hittingTime}) $\forall\, i$, $\tau^E_i$ equals in law $(\tau^\mu_0 | X^\mu_0=h_{i-1})$ with negative drift $\mu$ such that $|\mu| \in (m_1, \ldots, m_N)$. Moreover, the holding time of $E$ is given by $\tau^W_i \stackrel{d}{=} \mathcal{H}^\Psi \circ e^\mu_i$ where $e^\mu_i$ is an holding time for $X^\mu$. The reader should have in mind Theorem \ref{thm:HTrdBM} and Section \ref{sec:SeismicWavePropagation}. In particular, for the $i$-th seismic wave, we have that
\begin{align*}
\tau^E_i \stackrel{d}{=} \mathcal{H}^\Phi \circ h_{i-1}, \quad \textrm{with} \quad \mathbf{E}[h_{i-1}] = \eta_\varepsilon
\end{align*}
is a running time for $E$ and $X^{[\nu]}$ as well as
\begin{align*}
\tau^E_i \stackrel{d}{=} \mathcal{H}^\Phi \circ e^{\mu_i}_0, \quad \mu_i \in \{v_r\}_{r=1, \ldots, N}, \quad \textrm{with} \quad \mathbf{E}[e^{\mu_i}_0] = \eta_\varepsilon
\end{align*}
is an holding time for $W$ and $X^{[\varepsilon]}$. Moreover, $\tau^W_i$ is an excursion time for $W$ and $X^{[\varepsilon]}$ but also an holding time for $E$ and $X^{[\nu]}$, from this we obtain
\begin{align*}
\tau^W_i \stackrel{d}{=} \mathcal{H}^\Psi \circ e^{\mu_i}_0, \quad \mu_i \in \{-m_r\}_{r=1, \ldots, N} \quad \textrm{with} \quad \mathbf{E}[e^{\mu_i}_0] = \eta_\nu .
\end{align*}
In our construction we completely neglect the characterization of $\tau^W_i$. Thus, we do not care about $\{\mathcal{H}^\Psi \circ e^{\mu_i}_0\}_i$, we only know that the sequence exists. It is controlled by the operator $\eta_\nu D^\Psi_t$ in the boundary condition of \eqref{NLBVPnodes}.

By collecting the previous arguments, for  $i \in \{1, \ldots, \mathfrak{N}\}$, we have
\begin{align*}
\tau_i \stackrel{d}{=} \mathcal{H}^\Phi \circ h_{i-1} + \mathcal{H}^\Psi \circ e^{\mu_i}_0
\end{align*}
and we can write the right-hand side of the formula above as
\begin{align*}
\mathbf{E}_{h_{i-1}}[f(X^{\mu_i}_t),\, 0 \leq t < \mathcal{H}^\Phi_{h_{i-1}}] + \mathbf{E}_{h_{i-1}}[f(h_i),\, \mathcal{H}^\Phi_{h_{i-1}} \leq t < \mathcal{H}^\Phi_{h_{i-1}} + \mathcal{H}^\Psi_{e^{\mu_i}_0}].
\end{align*}
Since 
\begin{align*}
\mathbf{E}_x[f(X^{[\nu]}_t), 0 \leq t < \tau^\mu_0 + \mathcal{H}^\Psi_{e^{\mu}_0}] = \mathbf{E}_x[f(X^{\mu}_t), 0 \leq t < \tau^\mu_0] + \mathbf{E}_x[f(J), \tau^\mu_0 \leq t < \tau^\mu_0 + \mathcal{H}^\Psi_{e^{\mu}_0}]
\end{align*}
equals
\begin{align*}
\mathbf{E}_{x}[f(X^{\mu} \circ V^{-1}_t + A \circ \gamma \circ V^{-1}_t),\, 0 \leq t < \tau^\mu_0 + \mathcal{H}^\Psi_{e^{\mu}_0}],
\end{align*}
we get
\begin{align*}
\mathbf{E}_{x}[f(E_t), \tau_{i-1} \leq t < \tau_{i}] = \mathbf{E}_{x}[f(X^{\mu_i} \circ V^{-1}_t + A \circ \gamma \circ V^{-1}_t),\, 0 \leq t < \tau^{\mu_i}_0 + \mathcal{H}^\Psi_{e^{\mu_i}_0}] 
\end{align*}
with $V_t = t + \mathcal{H}^\Psi \circ \eta_\nu \gamma^0_t (X^\mu)$ and $(\tau^\mu_0 | X^\mu_0 = x) \stackrel{d}{=} \mathcal{H}^\Phi_x$. This proves \eqref{formula2-thm:equivE}.\\

Now we consider $h_i$ as a random variable. Recall that, $\forall\, i$, $\tau^E_i$ can be associated with $\tau^{[\nu]}_{h_i}$, that is the time at which a jump occurs (see formula \eqref{jumpXnode}). Due to $\Upsilon$ (see Section \ref{sec:UpsilonExpJumps}), the jump $J$ is an exponential r.v. with parameter $1/\eta_\varepsilon$. Under Assumption \ref{ASS3}, 
\begin{align}
h_i \, \textrm{ equals in law } \, J \, \textrm{ for every excursion $i$}. 
\label{J-thm:equivE}
\end{align}
Thus we have a link between the jumps of $E$ and $\Upsilon$ in the problem \eqref{NLBVPnodes}. Recall that (see formula \eqref{jumpOnlyXnode}) $V_t = t$ for $\tau_{i-1} \leq t < \tau^{[\nu]}_{h_i} \stackrel{d}{=}\tau^{[\nu]}_J$, that is $X^{[\nu]}_t = X^{\mu_i}_t$ for $\tau_{i-1} \leq t < \tau^{[\nu]}_{h_i}\stackrel{d}{=}\tau^{[\nu]}_J$. The right-hand side of the formula above takes the form
\begin{align*}
    \mathbf{E}_{h_{i-1}}[f(E_t), \tau_{i-1} \leq t < \tau_i] = \mathbf{E}_{h_{i-1}}[f(X^{[\nu]}_t),\, \tau_{i-1} \leq t < \tau^{[\nu]}_{h_i}] + \mathbf{E}_{h_{i-1}}[f(h_i)],\, \tau^{[\nu]}_{h_i} \leq t < \tau_i].
\end{align*}
For the first seismic wave we have
\begin{align*}
\mathbf{E}_{h_0}[f(E_t), 0 \leq t < \tau_1] = \mathbf{E}_{h_0}[f(X^{[\nu]}_t, 0 \leq t < \tau_1)]
\end{align*}
which is \eqref{Epdeh0}. By considering the (continuous) excursion of $X^{[\nu]}$ on $(0, \infty)$, 
\begin{align*}
\{X^{[\nu]}_t, \, 0 \leq t < \mathcal{H}^\Psi_{e^{\mu_i}_0} + \tau^{\mu_{i+1}}_0 \},
\end{align*}
with \eqref{J-thm:equivE} at hand, we have
\begin{align*}
\mathbf{E}_x[f(E_t)] = \mathbf{E}_x[f(X^{[\nu]}_t)], \quad t \geq 0, \; x \in [0, \infty)
\end{align*}
which is \eqref{formula1-thm:equivE}.

\subsection{Proof of Theorem \ref{thm:equivW}} 
\label{proof-thm:equivW}

First we show that \eqref{formula2-thm:equivW} solves \eqref{NLBVPedges}. Since we consider a given region $r$, then $\Phi=\Phi_r$ and $\mu_i= m_r$. We have motion on the region $r$ and this implies that $h_r>h_*$, that is
\begin{align*}
\zeta^W_{abs} > (\zeta^W_{kil} \wedge \zeta^W_{reg}) \quad \textrm{and} \quad \zeta^{[\nu]}_{abs} > (V \circ \zeta^{\mu_i} \wedge V \circ \tau^{\mu_i}_\ell).
\end{align*}
Assume for a while that $c=0$, then almost surely $\zeta^W_{kil}= \infty$ as well as $\zeta^{[\varepsilon]}=\infty$. For the $i$-th seismic wave, 
\begin{align*}
\mathbf{E}_0 [f(W_t), \, \tau_{i-1} \leq t < \tau_i ]
\end{align*}
equals
\begin{align*}
\mathbf{E}_0 [f(0), \, \tau_{i-1} \leq t < \tau_{i-1} + \tau^E_i] + \mathbf{E}_0[f(X^{\mu_i}_t),\, \tau_{i-1} + \tau^E_i \leq t < \tau_i]
\end{align*}
where
\begin{align*}
\tau^E_i \stackrel{d}{=} \mathcal{H}^\Phi \circ h_i \quad \textrm{with} \quad \mathbf{E}[h_i] = \eta_\varepsilon
\end{align*}
as hitting time $(\tau^\mu_0 | X^{\mu}_0 = h_i)$ of $X^{\mu_i}$ with $\mu_i = -m_r < 0$. Under Assumption \ref{ASS3}, 
\begin{align*}
\tau^E_i \stackrel{d}{=} \mathcal{H}^\Phi \circ e^{\mu}_0 \quad \textrm{with} \quad \mathbf{E}[e^\mu_0] = \eta_\varepsilon 
\end{align*}
which is an holding time of $X^{[\varepsilon]}_t$ introduced by $\eta_\varepsilon D^\Phi_t$ in the boundary condition of \eqref{NLBVPedges} as Corollary \ref{coro:HTverteces} entails. Thus, the processes $W$ and $X^{[\varepsilon]}$ have a Brownian excursion (given by $X^\mu$ with $\mu \in \{v_r\}_r$) only after an holding time identically distributed as $\mathcal{H}^\Phi \circ e^\mu_0$. This means that
\begin{align*}
\mathbf{E}_0[f(W_t), \,t < \zeta^W_{reg}] = \mathbf{E}_0[f(X^{[\varepsilon]}_t), \, t < \tau^{[\varepsilon]}_\ell] 
= & \mathbf{E}_x[f(X^\mu \circ V^{-1}_t),\, V^{-1}_t < \tau^\mu_\ell]\\
= & \mathbf{E}_x[f(X^\mu \circ V^{-1}_t),\, t < V \circ \tau^\mu_\ell]
\end{align*}
where $\mu=\mu_i$ is the drift for the $i$-th seismic wave. For $c>0$ we only need to consider the elastic kill, that is
\begin{align*}
\mathbf{E}_0[f(W_t), \,t < \zeta^W_{reg} \wedge \zeta^W_{kill}] = \mathbf{E}_0[f(X^{[\varepsilon]}_t), \, t < \tau^{[\varepsilon]}_\ell \wedge \zeta^{[\varepsilon]}] 
\end{align*}
which takes the form
\begin{align*}
\mathbf{E}_0[f(X^{\mu}_t), \, t < V \circ \tau^{\mu}_\ell \wedge V \circ \zeta^{\mu}] 
\end{align*}
with $\mu= \mu_i$. This proves \eqref{formula2-thm:equivW}. 

Now assume that $\zeta^W_{abs} = \infty$ as well as $\zeta^{[\varepsilon]}$. This means that $E \perp W$ as well as $X^{[\nu]} \perp X^{[\varepsilon]}$. Moreover, the process $X^{[\varepsilon]}$ can reach the level $\ell$ and switch a new process according with the rule \eqref{ruleZetaReg}. Thus, $X^{[\varepsilon]}$ starts as a new process form $X^{[\varepsilon]}_0=0$.

In case $\zeta^W_{abs}$ and therefore $\zeta^{[\varepsilon]}$ are finite, then we only need to consider these absorption times and the equality \eqref{formula1-thm:equivW} follows.

\subsection{Proof of Theorem \ref{thm:NLBVPfinalG}} 
\label{proof-thm:NLBVPfinalG}

We only consider the case $c=0$. The case $c>0$ follows after standard arguments including the killing time $\zeta(\mathsf{S})$.

Observe that 
\begin{align*}
& \mathbf{E}_{\mathsf{v}}\left[ \int_0^{\infty} e^{-\lambda t} f(\mathsf{Q}_t)\, dt \right]\\ 
= & \mathbf{E}_{(\varepsilon, x)}\left[ \int_0^{\tau^{[\Theta]}_\ell} e^{-\lambda t} f(\Theta_t, X^{[\Theta]}_t)\, dt \right]\\
     = & \mathbf{E}_{(\varepsilon, x)} \left[ \int_0^{\tau^{[\Theta]}_\ell} e^{-\lambda t} f(\Theta_t, X^{[\Theta]}_t) dt \, , \, (\tau^{[\varepsilon]}_\ell < \tau^{[\varepsilon]}_0) \cup (\tau^{[\varepsilon]}_0 < \tau^{[\varepsilon]}_\ell)  \right]  \\
     = & I_1 + I_2
\end{align*}
\begin{align*}
& \mathbf{E}_{\mathsf{v}}\left[ \int_0^{\infty} e^{-\lambda t} f(\mathsf{Q}_t)\, dt \right]\\ 
= & \mathbf{E}_{(\varepsilon, x)}\left[ \int_0^{\tau^{[\Theta]}_\ell} e^{-\lambda t} f(\Theta_t, X^{[\Theta]}_t)\, dt \right]\\
     = & \mathbf{E}_{(\varepsilon, x)} \left[ \int_0^{\tau^{[\Theta]}_\ell} e^{-\lambda t} f(\Theta_t, X^{[\Theta]}_t) dt \, , \, (\tau^{[\varepsilon]}_\ell < \tau^{[\varepsilon]}_0) \cup (\tau^{[\varepsilon]}_0 < \tau^{[\varepsilon]}_\ell)  \right]  \\
     = & I_1 + I_2
\end{align*}
where
\begin{align}
    I_1 = \mathbf{E}_{(\varepsilon, x)} \left[ \int_0^{\tau^{[\varepsilon]}_\ell} e^{-\lambda t} f(\varepsilon, X^{[\varepsilon]}_t) dt \, , \, \tau^{[\varepsilon]}_\ell < \tau^{[\varepsilon]}_0  \right]
\end{align}
and $I_2 = I_{2,1} + I_{2,2}$ with
\begin{align*}
I_{2,1} = & \mathbf{E}_{(\varepsilon, x)} \left[ \int_0^{\tau^{[\varepsilon]}_0} e^{-\lambda t} f(\varepsilon, X^{[\varepsilon]}_t) dt \, , \, \tau^{[\varepsilon]}_0 < \tau^{[\varepsilon]}_\ell  \right]  
\end{align*}
and
\begin{align*}
I_{2,2} = & \mathbf{E}_{(\varepsilon, x)} \left[ \int_{\tau^{[\varepsilon]}_0}^{\tau^{[\Theta]}_\ell} e^{-\lambda t} f(\Theta_t, X^{[\Theta]}_t) dt \, , \, \tau^{[\varepsilon]}_0 < \tau^{[\varepsilon]}_\ell  \right]
\end{align*}
We get
\begin{align*}
I_1 + I_{2,1} = \mathbf{E}_{(\varepsilon, x)} \left[ \int_0^{\tau^{[\varepsilon]}_0 \wedge \tau^{[\varepsilon]}_\ell} e^{-\lambda t} f(\varepsilon, X^{[\varepsilon]}_t) dt \right]
\end{align*}
and
\begin{align*}
I_{2,2} =  & \mathbf{E}_{(\varepsilon, x)} \left[ \int_{\tau^{[\varepsilon]}_0}^{\tau^{[\Theta]}_\ell} e^{-\lambda t} f(\Theta_t, X^{[\Theta]}_t) dt \, , \, \tau^{[\varepsilon]}_0 < \tau^{[\varepsilon]}_\ell  \right]\\ 
    = & \sum_{\varepsilon^\prime \in \mathcal{E}} \rho_{\varepsilon^\prime} \mathbf{E}_{(\varepsilon, x)} \left[ \int_{\tau^{[\varepsilon]}_0}^{\tau^{[\varepsilon^\prime]}_\ell} e^{-\lambda t} f(\varepsilon^\prime, X^{[\varepsilon^\prime]}_t) dt \, , \, \tau^{[\varepsilon]}_0 < \tau^{[\varepsilon]}_\ell  \right]\\
    = & \mathbf{E}_{(\varepsilon, x)} \left[ e^{-\lambda \tau^{[\varepsilon]}_0}, \tau^{[\varepsilon]}_0 < \tau^{[\varepsilon]}_\ell \right] \sum_{\varepsilon^\prime \in \mathcal{E}} \rho_{\varepsilon^\prime}   \mathbf{E}_{(\cdot, 0)} \left[ \int_0^{\tau^{[\varepsilon^\prime]}_\ell} e^{-\lambda t} f(\varepsilon^\prime, X^{[\varepsilon^\prime]}_t) dt \right]
\end{align*}
where $\tau^{[\varepsilon^\prime]}_\ell = \inf\{t\,:\, X^{[\varepsilon^\prime]}_t = \ell \, |\, X^{[\varepsilon^\prime]}_0 = 0\}$ and $\tau^{[\varepsilon^\prime]}_0 = \inf\{t\,:\, X^{[\varepsilon^\prime]}_t = 0 \, |\, X^{[\varepsilon^\prime]}_0 = x\}$ is such that
\begin{align*}
    \mathbf{E}_{(\varepsilon, x)} \left[ e^{-\lambda \tau^{[\varepsilon]}_0} \right] = \mathbf{E}_0 \left[ e^{-\lambda \mathcal{H}_x} \right] = e^{-x \Phi(\lambda)}, \quad \lambda>0
\end{align*}
($X^{[\varepsilon^\prime]}$ behaves like $X^\mu$ on $(0, \ell)$) and
\begin{align*}
    \mathbf{E}_{(\varepsilon, x)} \left[ e^{-\lambda \tau^{[\varepsilon]}_0} , \tau^{[\varepsilon]}_0 < \tau^{[\varepsilon]}_\ell \right] = \mathcal{K}_\lambda^{[\varepsilon]} (x), \lambda>0,\; x \in [0, a)
\end{align*}
is the solution to
\begin{align*}
    (\lambda - G_\mu) \, \mathcal{K}_\lambda^{[\varepsilon]}=0 \quad \textrm{with} \quad \mathcal{K}_\lambda^{[\varepsilon]} (0)=1,\; \mathcal{K}_\lambda^{[\varepsilon]} (a)=0.
\end{align*}
Thus, $\mathcal{K}_\lambda^{[\varepsilon]} = \mathcal{K}_\lambda \in \mathsf{K}$, $\forall\, \varepsilon \in \mathcal{E}$. We can write 
\begin{align*}
    R_\lambda f(\varepsilon, x) := \mathbf{E}_{(\varepsilon, x)}\left[ \int_0^{\tau^{[\Theta]}_\ell} e^{-\lambda t} f(\Theta_t, X^{[\Theta]}_t)\, dt \right]
\end{align*}
as
 \begin{align*}
    I_1 + I_2 
    = & I_1 + I_{2,1} + \mathcal{K}_\lambda^{[\varepsilon]} (x) \sum_{\varepsilon^\prime \in \mathcal{E}} \rho_{\varepsilon^\prime} \,   \mathbf{E}_{(\cdot, 0)} \left[ \int_0^{\tau^{[\varepsilon^\prime]}_\ell} e^{-\lambda t} f(\varepsilon^\prime, X^{[\varepsilon^\prime]}_t) dt  \right]\\
    = & \mathbf{E}_x \left[ \int_0^{\tau^{\mu}_0 \wedge \tau^{\mu}_\ell} e^{-\lambda t} f_\varepsilon(X^\mu_t) dt \right] +  \mathcal{K}_\lambda^{[\varepsilon]}(x) \sum_{\varepsilon^\prime \in \mathcal{E}} \rho_{\varepsilon^\prime} \,   R_\lambda f_{\varepsilon^\prime}(0).
\end{align*}
Now consider the Dirichlet semigroup
\begin{align*}
    R^\dagger_\lambda f_\varepsilon (x) = \mathbf{E}_{(\varepsilon, x)} \left[ \int_0^{\tau^\mu_0 \wedge \tau^\mu_\ell} e^{-\lambda t} f_{\varepsilon}(X^\mu_t) dt \right], 
\end{align*}
so that
\begin{align}
    R_\lambda f(\varepsilon, x) = R^\dagger_\lambda f_\varepsilon (x) + \mathcal{K}_\lambda^{[\varepsilon]} (x) 
 \sum_{\varepsilon^\prime \in \mathcal{E}} \rho_{\varepsilon^\prime} R_\lambda f_{\varepsilon^\prime}(0).
    \label{RrepSUM}
\end{align}
We first observe that $R_\lambda f$ belongs to $\mathsf{K}$ and
\begin{align}
    R_\lambda f(\cdot, 0) = \sum_{\varepsilon^\prime \in \mathcal{E}} \rho_{\varepsilon^\prime} R_\lambda f_{\varepsilon^\prime}(0).
\end{align}
Since $\lambda R^\dagger_\lambda f_\varepsilon - f_\varepsilon = G_\mu R^\dagger_\lambda f_\varepsilon$, it holds that
\begin{align*}
    G_\mu R_\lambda f(\varepsilon, x) = & \lambda R^\dagger_\lambda f_\varepsilon(x) - f_\varepsilon(x) + \lambda \mathcal{K}_\lambda^{[\varepsilon]} (x) \sum_{\varepsilon^\prime \in \mathcal{E}} \rho_{\varepsilon^\prime} R_\lambda f_{\varepsilon^\prime}(0)\\ 
 = & \lambda R_\lambda f(\varepsilon, x) - f(\varepsilon, x)
\end{align*}
on $(0, \ell)$. Moreover, recall that $\sum_{\varepsilon^\prime} \rho_{\varepsilon^\prime} =1$ and
\begin{align*}
    \eta \frac{\Phi(\lambda)}{\lambda} \Big( \lambda R_\lambda f(\cdot, 0) - f(\cdot, 0) \Big) =  \sum_{\varepsilon^\prime \in \mathcal{E}} \rho_{\varepsilon^\prime} \eta \frac{\Phi(\lambda)}{\lambda} \Big( \lambda  R_\lambda f_{\varepsilon^\prime}(0) - f_{\varepsilon^\prime}(0) \big)
\end{align*}
can be associated with the motion $X^{[\varepsilon^\prime]}$ on the edge $\varepsilon^\prime$ in terms of the equivalence in Theorem \ref{thm:equiv}. Indeed, 
\begin{align*}
    \eta \frac{\Phi(\lambda)}{\lambda} \Big( \lambda R_\lambda f(\cdot, 0) - f(\cdot, 0) \Big) = \sum_{\varepsilon^\prime \in \mathcal{E}} \rho_{\varepsilon^\prime} \, \eta \int_0^\infty e^{-\lambda t} D^\Phi_t \mathbf{E}_0 [f_{\varepsilon^\prime} (X^{[\varepsilon^\prime]}_t)] dt
\end{align*}
and $u \in \mathsf{D}_\phi$ so that
\begin{align*}
    \eta \int_0^\infty e^{-\lambda t} D^\Phi_t u(t, \mathsf{v}) dt = \sum_{\varepsilon^\prime \in \mathcal{E}} \rho_{\varepsilon^\prime} \, \eta \int_0^\infty e^{-\lambda t} D^\Phi_t \mathbf{E}_0 [f_{\varepsilon^\prime} (X^{[\varepsilon^\prime]}_t)] dt
\end{align*}
According with Theorem \ref{thm:NLBVPwithDir}, we have
\begin{align}
    \eta D^\Phi_t \mathbf{E}_0 [f_{\varepsilon^\prime} (X^{[\varepsilon^\prime]}_t)] = \Big(\mathbf{E}_0 [f_{\varepsilon^\prime} (X^{[\varepsilon^\prime]}_t)] \Big)^\prime - c \, \mathbf{E}_0 [f_{\varepsilon^\prime} (X^{[\varepsilon^\prime]}_t)]
\end{align}
where $\mathbf{E}_0 [f_{\varepsilon^\prime} (X^{[\varepsilon^\prime]}_t)] = u_{\varepsilon^\prime} (t, 0)$ is the projection of $u(t, \mathsf{v})$ on the edge $\varepsilon^\prime$. Then, we get
\begin{align}
    \eta D^\Phi_t u(t, 0) = \sum_{\varepsilon^\prime \in \mathcal{E}} \rho_{\varepsilon^\prime} \, \big( u^\prime_{\varepsilon^\prime} (t,0) - c\, u_{\varepsilon^\prime}  (t, 0) \big) \quad t>0.
\end{align}
Recalling that $u \in C((0, \infty) \times \mathsf{S})$, 
\begin{align}
    \sum_{\varepsilon^\prime \in \mathcal{E}} \rho_{\varepsilon^\prime} \,  u_{\varepsilon^\prime}  (t, 0) = \sum_{\varepsilon^\prime \in \mathcal{E}} \rho_{\varepsilon^\prime} \,  u(t, \varepsilon^\prime, 0) = u(t, \mathsf{v}).
\end{align}
In particular, $\eta=m_r/\sigma_r$ and $u \in \mathsf{U}^r_\phi$ with $r \in \{1,\ldots, N\}$. Uniqueness follows from Theorem \ref{thm:NLBVPwithDir} and Theorem \ref{thm:equiv}.

\section{Figures}

\begin{figure}
    \centering
    \includegraphics[width=.7\linewidth]{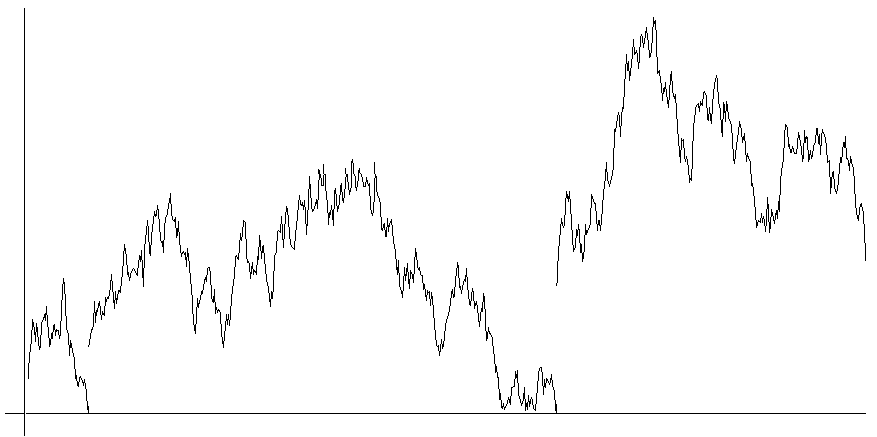}
    \caption{An example of Brownian motion $X^\mu$ with drift $\mu=0$ and non-local condition at the boundary point $\{0\}$ as in \eqref{NLBVPnodes} with $\eta=0$ (then $V_t=t$). The additive part $A\circ \gamma^0_t(X^\mu)$ gives the jump from zero to a random point in $(0,\infty)$. Due to the local time at $\{0\}$, a jump occurs as soon as the process hits the boundary point $\{0\}$. Thus, the process is pushed away from zero. Here the jump is random.}
    \label{fig:BmNLBVPjumps}
\end{figure}

\begin{figure}
    \centering
    \includegraphics[width=.7\linewidth]{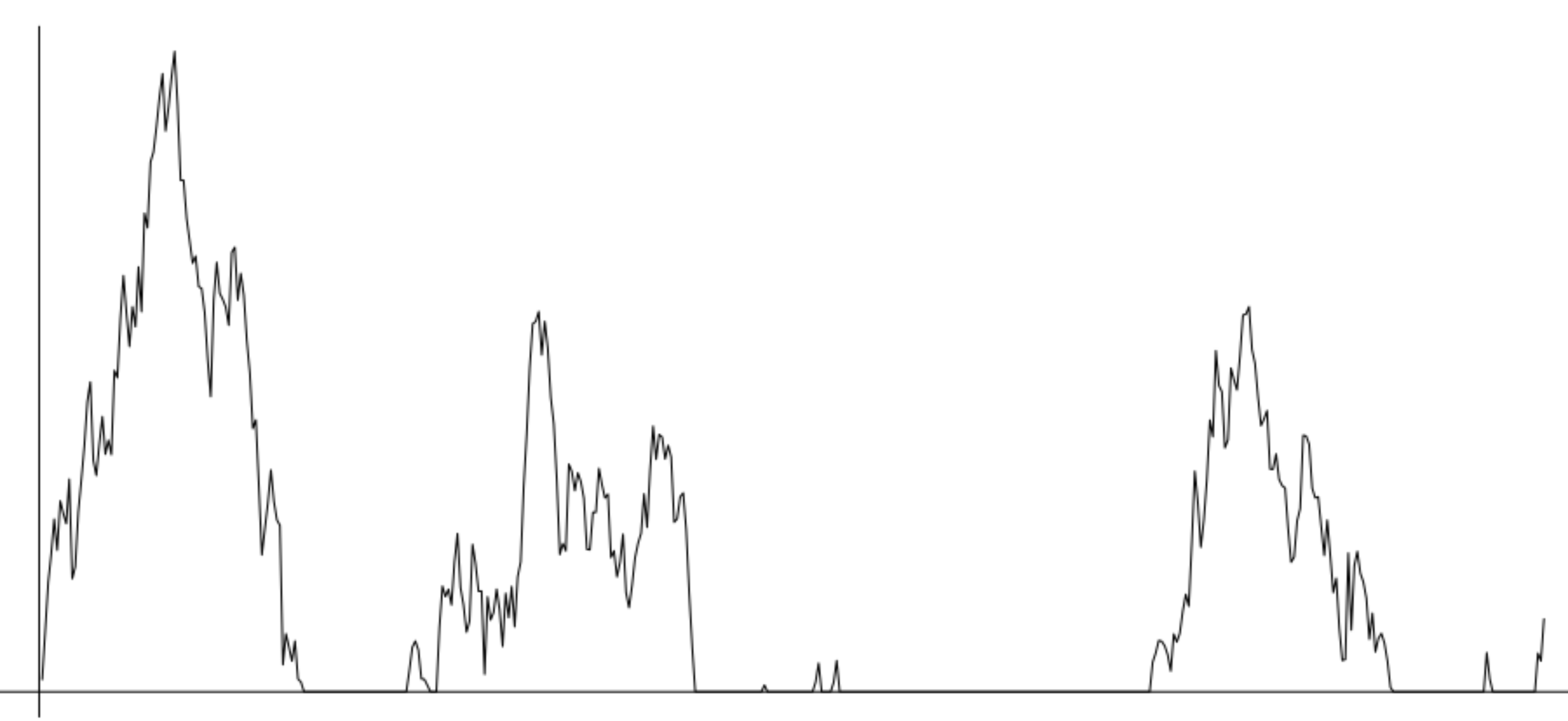}
    \caption{ An example of Brownian motion $X^\mu$ with drift $\mu=0$ under non-local condition at the boundary point $\{0\}$ as in \eqref{NLBVPnodes} with $\eta>0$ and $\Psi=Id$, that is $\mathbf{D}^\Psi_x u = u^\prime$. The behavior at $\{0\}$ is the same as in case of the boundary condition in \eqref{NLBVPedges} where the process is stopped at $\ell>0$. Due to the local time at $\{0\}$, the process is forced to stop at $\{0\}$ for a random holding time.}
    \label{fig:BmNLBVPplateaus}
\end{figure}

\begin{figure}
    \centering
    \includegraphics[width=.7\linewidth]{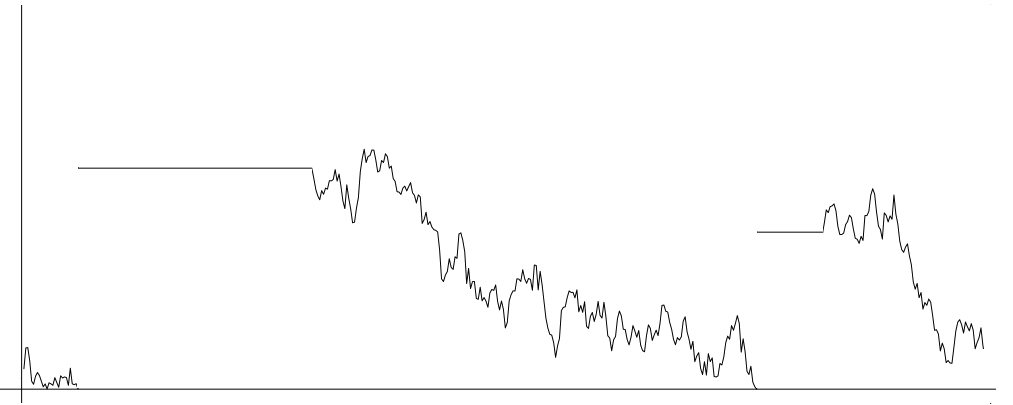}
    \caption{ An example of Brownian motion $X^\mu$ with drift $\mu=0$ under non-local condition at the boundary point $\{0\}$ as in \eqref{NLBVPnodes} with $\eta>0$ and $\sigma>0$. The effects in Figure \ref{fig:BmNLBVPjumps} and Figure \ref{fig:BmNLBVPplateaus} are combined. The process is right-continuous and the holding time is observed immediately after the jump.} 
    \label{fig:BmNLBVPjumpsANDplateaus}
\end{figure}

%

\begin{figure}
    \centering
    \includegraphics[width=.7\linewidth]{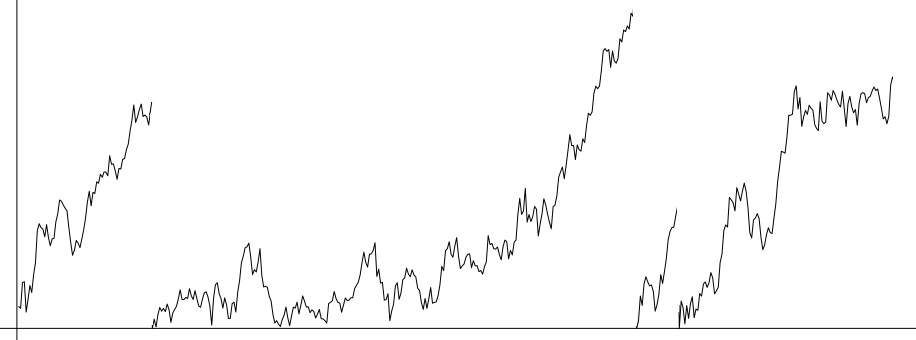}
    \caption{The reflected drifted Brownian motion $X^\mu$ on $[0, \infty)$ with drift $\mu=0.1$. The process starts from zero, it is killed at a random level $\ell>0$, then it starts from zero as a new process. We see that the process may return at the boundary point $\{0\}$ before to be killed.}
    \label{fig:RDBMrandomLevel}
\end{figure}

\begin{figure}
    \centering
    \includegraphics[width=0.5\linewidth]{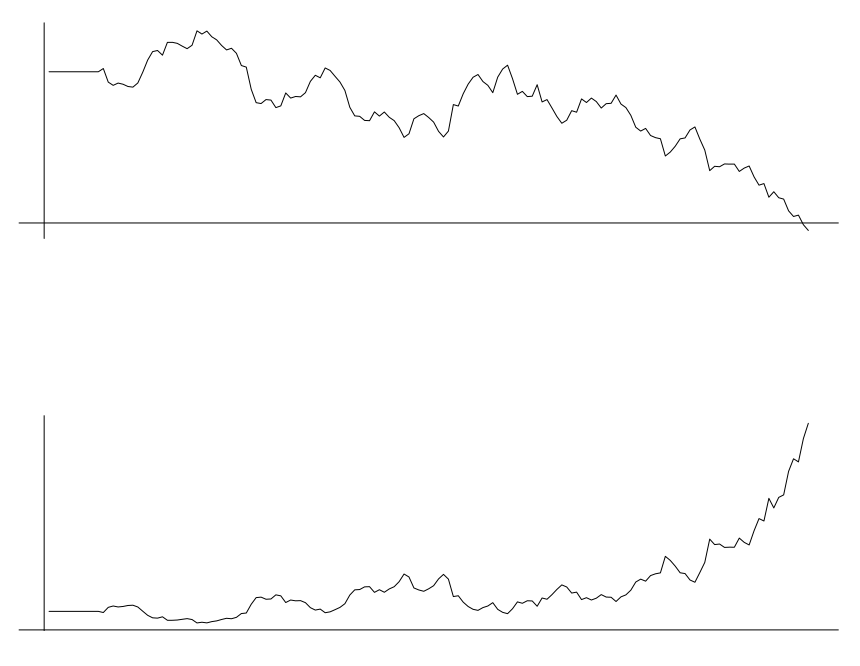}
    \caption{The transformation \eqref{g-transf} with $h=1$ and $k=0.1$. Above the path of the relaxation process $X^\bullet \circ V^{-1}$ after a jump away from $\{0\}$ and below the path of the accumulation process $g(X^\bullet \circ V^{-1})$.}
    \label{fig:g-transf}
\end{figure}

\begin{figure}
    \centering
    \includegraphics[width=0.6\linewidth]{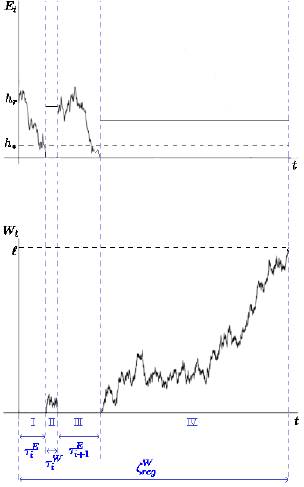}
    \caption{This is the case of two visits of $\{0\}$ for $W$. The processes $X^{[\nu]}$ on $(0, \infty)$ and $X^{[\varepsilon]}$ on $[0, \ell)$ respectively describe $W$ and $E$ in a given region (star graph $\mathsf{S}_{\nu_r}$). The process $X^{[\nu]}$ (picture at the top) starts from the random level $h_r$, it jumps to a new random level as it approaches the boundary point $\{0\}$, then it starts afresh after the holding time. The process $X^{[\varepsilon]}$ (picture at the bottom) starts from zero where it is stuck until continuous reflection on $(0, \ell)$.  It is killed at $\ell$. As the process $X^{[\varepsilon]}$ hits the level $\ell$, then it start as a new process in a new star graph. See the table above (Section \ref{sec:Stat}) for a detailed description. }
    \label{fig:relations}
\end{figure}
\end{appendix}

\newpage

{\bf Funding.}
Fausto Colantoni and Mirko D'Ovidio was supported by MUR under the project PRIN 2022 - 2022XZSAFN: Anomalous Phenomena on Regular and Irregular Domains: Approximating Complexity for the Applied Sciences - CUP B53D23009540006 - PNRR M4.C2.1.1 (\url{https://www.sbai.uniroma1.it/~mirko.dovidio/prinSite/index.html}). The research was also  partially supported by INdAM-GNAMPA and Sapienza University of Rome (Ateneo 2020).

\end{document}